\documentclass{amsart}
\usepackage[utf8]{inputenc}
\usepackage{amsfonts}
\usepackage[hidelinks]{hyperref}
\usepackage{dynkin-diagrams}
\usepackage{tikz-cd} 
\usepackage{verbatim}
\usepackage{cite}
\usepackage{marginnote}
\usepackage{float}

\usepackage{amsthm}
\usepackage{pdflscape}
\usepackage{pgfplots}
\usepackage{mathrsfs}
\usepackage{tikz}
\usepackage{changepage}
\usetikzlibrary{braids,calc,shapes.geometric,matrix, knots,     
        positioning,intersections}
\usetikzlibrary{hobby}

\usepackage{blochsphere}

\pgfplotsset{compat=newest}
\usepgfplotslibrary{fillbetween}

\usepackage{thmtools, thm-restate}
\declaretheorem[name=Theorem,numberwithin=section]{thm}

\newtheorem{prop}[thm]{Proposition}
\newtheorem{lem}[thm]{Lemma}

\theoremstyle{definition}
\newtheorem{defn}[thm]{Definition}

\newtheorem{exmp}[thm]{Example}

\theoremstyle{remark}
\newtheorem{rem}[thm]{Remark}

\makeatletter
\let\c@equation\c@thm
\makeatother
\numberwithin{equation}{section}

\newcommand{\cat}{\mathrm{CAT}}
\newcommand{\knotlinewidth}{1pt}
\newcommand{\ssim}{\sim_s}
\newcommand{\Proj}{\mathrm{Proj}}

\tikzset{overcross/.style={double, line width=1.5, white, double=black, double distance=\knotlinewidth}}

\title{The Deligne Complex for the $B_3$ Artin Group}

\author{Katherine Goldman, Amy Herron}

\begin{document}

\begin{abstract}
We show that the piecewise Euclidean Moussong metric on the Deligne complex of the Artin group of type $B_3$ is $\mathrm{CAT}(0)$. We do this by establishing a criteria for a complex made of $B_3$ simplices to be $\mathrm{CAT}(1)$ in terms of embedded edge paths, which in particular applies to the spherical Deligne complex of type $B_3$. This provides one more step to showing that the Moussong metric is $\mathrm{CAT}(0)$ for any 3-dimensional Artin group.  
\end{abstract}

\maketitle
\tableofcontents

\section{Introduction}

Artin groups are closely tied to Coxeter groups. 
While Coxeter groups are largely well-understood, there are still a plethora of open questions regarding the Artin groups. 
One of the most prominent of these is the $K(\pi,1)$ conjecture, which essentially asks if a specific space is an Eilenberg-MacLane space for a given Artin group. 
This turns out to be equivalent to asking if an object called the ``(modified) Deligne complex'' is contractible \cite{cd1995k}. 
One common technique in geometric group theory and metric geometry for showing spaces are contractible is to show that they possess some notion of non-positive curvature.
Recent success has been had using various new ideas of non-positive curvature \cite{paolini2021proof,haettel2023lattices,huang2024labeled}.
But it has long been conjectured that Deligne complexes are $\mathrm{CAT}(0)$ under a specific metric, typically called the Moussong metric, defined in terms of the Davis complex of the associated Coxeter group \cite[Conj.~3]{cd1995k}. Very few Artin groups are known to satisfy this; in fact, it is only known when the group is 2-dimensional \cite[Thm.~4]{cd1995k} or if the only ``spherical-type'' subdiagrams of its Coxeter-Dynkin diagram are either edges or $A_3$ \cite{charney2004deligne}.

It is notoriously difficult to show that piecewise spherical cell complexes of dimension $\geq 2$ are $\cat(1)$. 
Broadly, this is because it requires establishing that any loop of length $< 2\pi$ is not a (local) geodesic. In \cite{elder2002curvature}, a method is presented for showing arbitrary 2-dimensional piecewise spherical simplicial complexes are $\cat(1)$, via excluding certain ``galleries''. Although an explicit list of galleries is given, it is quite large, and generally unfeasible to apply to Artin groups directly.
In \cite{charney2004deligne}, this is overcome by cleverly using the geometry of $A_3$ simplices and the relationship of the Deligne complex with the Coxeter complex to reduce the study of all galleries to just a few which stay in the 1-skeleton. This is the inspiration for our current work.  

Our method for simplifying the study of general loops to edge paths is made somewhat more simple by using a slightly altered criteria developed by Bowditch for a locally $\cat(1)$ space to be $\cat(1)$. This criteria says that in order for a locally $\cat(1)$ space to be $\cat(1)$, all short loops must be ``shrinkable'' to the trivial loop. Since shrinkability is an equivalence relation, we are able to easily reduce to the edge paths by showing that an arbitrary short loop is shrinkable to a short closed edge path.

While we perform a similar reduction to edge paths as in the $A_3$ case, there are a number of other paths that now must be ruled out. 
To show that the spherical Deligne complex for the braid group on four strands is CAT(1), Charney needed to essentially consider two configurations of geodesic edge paths; one of edge length 4, and one of edge length 6.  
For the spherical Deligne complex of type $B_3$, one must consider paths of edge length 4, 6, 8, and 10.  
The geometry of the spherical Deligne complex of type $B_3$ also becomes much less friendly than type $A_3$, which requires the implementation of new tools to analyze these edge paths.  
We are able to confirm that all edge paths are shrinkable, and therefore, we add $B_3$ to the collection of Artin groups with $\cat(0)$ Deligne complex:

\begin{restatable*}{thm}{moussongiscatone} \label{cor:edgesareshrinkable}
    The (piecewise spherical) Moussong metric on the spherical Deligne complex of type $B_3$ is $\cat(1)$. 
\end{restatable*}

The most general setting in which this may be applied is the following.

\begin{restatable*}{thm}{generalcatzero} \label{thm:generalcat0}
    Suppose $A(\Gamma)$ is a 3-dimensional Artin group and $\Gamma$ contains no subdiagram of type $H_3$. Then the (piecewise Euclidean) Moussong metric on the Deligne complex for $A(\Gamma)$ is $\cat(0)$ and, in particular, $A(\Gamma)$ satisfies the $K(\pi,1)$ conjecture.
\end{restatable*}

\begin{figure}[!ht]

\begin{tikzpicture}[scale=0.49]
  \matrix[matrix of nodes,column sep=15pt,nodes={anchor=center, minimum height=0.1cm, align=flush center}]{
\begin{tikzpicture}
    \filldraw[] (45:1) circle(0.5mm) (135:1) circle(0.5mm) (225:1) circle(0.5mm) 
        (315:1) circle(0.5mm);
    \draw (45:1) -- (135:1) -- (225:1) -- (315:1) --cycle;
    \node at (90:0.9) {$4$};   
\end{tikzpicture}
&
\begin{tikzpicture}
    \filldraw[] (45:1) circle(0.5mm) (135:1) circle(0.5mm) (225:1) circle(0.5mm) 
        (315:1) circle(0.5mm);
    \draw (45:1) -- (135:1) -- (225:1) -- (315:1) --cycle;
    \node[label=above:$4$] at (-90:1.3) {};    
    \node[label=below:$4$] at (90:1.3) {};    
\end{tikzpicture}
&
\begin{tikzpicture}
    \filldraw[] (45:1) circle(0.5mm) (135:1) circle(0.5mm) (225:1) circle(0.5mm) 
        (315:1) circle(0.5mm);
    \draw (45:1) -- (135:1) -- (225:1) -- (315:1) --cycle;
    \node[label=left:$4$] at (0:1.25) {};    
    \node[label=below:$4$] at (90:1.3) {};    
\end{tikzpicture}
\\
\begin{tikzpicture}
    \filldraw[] (45:1) circle(0.5mm) (135:1) circle(0.5mm) (225:1) circle(0.5mm) 
        (315:1) circle(0.5mm) (45:2) circle(0.5mm) (225:2) circle(0.5mm);
    \draw (45:1) -- (135:1) -- (225:1) -- (315:1) --cycle;
    \draw (45:1) -- (45:2);
    \draw (225:1) -- (225:2);
    
    \node[label=right:$4$] at (180:1.25) {};    
    \node[label=below:$4$] at (90:1.3) {};    
\end{tikzpicture}
&
\begin{tikzpicture}
    \filldraw[] (180:1) circle(0.5mm) (0:0) circle(0.5mm) (45:1) circle(0.5mm) 
        (-45:1) circle(0.5mm);
    \draw (180:1) -- (0:0) -- (45:1);
    \draw (0:0) -- (-45:1);
    % \node[label=above:$4$] at (180:0.5) {};
    \node at (-0.5,0.2) {$4$};
\end{tikzpicture}
&
\begin{tikzpicture}
    \filldraw[] (180:1) circle(0.5mm) (0:0) circle(0.5mm) (0:1) circle(0.5mm) 
        (0:2) circle(0.5mm);
    \draw (180:1) -- (0:0) -- (0:1) -- (0:2);
    % \node[label=above:$4$] at (-180:0.5) {};    
    % \node[label=above:$4$] at (0:1.5) {};    
    \node at (-0.5,0.2) {$4$};
    \node at (1.5,0.2) {$4$};
\end{tikzpicture}
\\
}
;
\end{tikzpicture}
\caption{New Coxeter-Dynkin diagrams}
\label{fig:newexamples}
\end{figure}
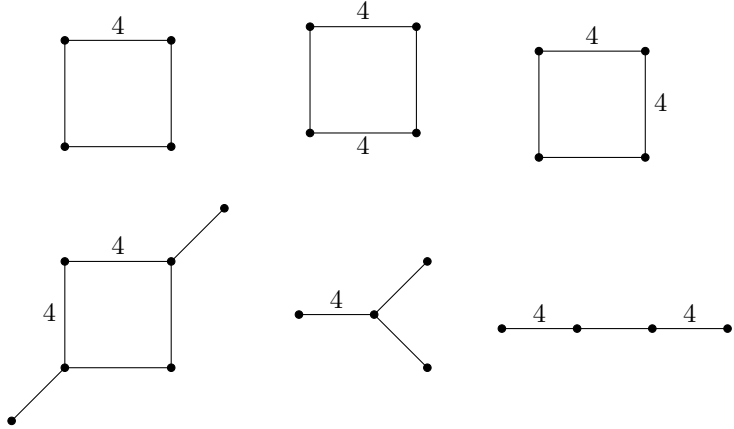

\begin{exmp}
    Some new examples of diagrams $\Gamma$ which give Artin groups $A(\Gamma)$ whose Deligne complexes are now known to have a $\cat(0)$ Moussong metric can be seen in Figure \ref{fig:newexamples}.
\end{exmp}

It turns out, in fact, that our method for showing this spherical Deligne complex is $\cat(1)$ may also be applied to a more general class of 2-dimensional simplicial complexes. For clarity, and for broader application, we have extracted the necessary criteria thusly. The full definitions are given in Section \ref{sec:ComCrit}.

\begin{restatable*}{thm}{catonecrit}
\label{thm:B3 CAT1 Criteria}
    Suppose $X$ is a $B_3$ simplicial complex satisfying each of the following conditions:
    \begin{enumerate}
        \item The link of every vertex is nonempty and connected.
        \item If $v$ is a vertex of type $\hat s_1$, then $lk(v,X)$ has girth at least $8$.
        \item If $v$ is a vertex of type $\hat s_2$, then $lk(v,X)$ is complete bipartite, partitioned along types of vertices, and contains at least one embedded 4-cycle.
        \item If $v$ is a vertex of type $\hat s_3$, then $lk(v,X)$ has girth at least $6$.
        \item Any path in the 1-skeleton of $X$ which passes through vertices of the type seen in Figure \ref{fig:checking edge loops} is contained in a subcomplex of $X$ of the corresponding type in Figure \ref{fig:filling the edge loops}.
        \item The 10-cycles of the form in Figure \ref{fig:bad 10 cycle} are not embedded.
    \end{enumerate}
    Then the $B_3$ metric on $X$ is $\cat(1)$.
\end{restatable*}

Following this,
in Section \ref{sec:Definitions}, we will discuss the necessary background and general definitions surrounding Coxeter and Artin groups, along with some of their respective complexes. 
Section \ref{sec:geometry} includes more specific properties of these complexes in the case of type $B_n$ that will be needed later in the paper, notably including a novel proof that the Artin group of type $B_n$ embeds into type $A_{2n-1}$ in a certain well-behaved way. The majority of this section is devoted to this new proof, which is done using mapping class groups. 
Section \ref{sec:shrinking} uses new techniques to show certain paths are shrinkable, implying the spherical Deligne complex of type $B_3$ satisfies the hypotheses of Theorem \ref{thm:B3 CAT1 Criteria}, completing the proof that the metric is $\cat(1)$.

\subsection*{Acknowledgements}

The authors would like to thank Jingyin Huang, Johanna Mangahas, and Piotr Przytycki for their invaluable input. They would also like to extend their gratitude to AIM for the opportunity to start this collaboration. The first author is supported by NSF DMS-2402105.  The second author received travel support from Simons Foundation (965204, JM).

\section{A combinatorial criteria for \texorpdfstring{$\cat(1)$ $B_3$}{CAT(1) B\_3} complexes}
\label{sec:ComCrit}

Let $\Delta$ denote a spherical 2-simplex on vertex set $\{\hat s_1, \hat s_2, \hat s_3\}$, where the internal angle at $\hat s_1$ is $\pi/4$, the angle at $\hat s_2$ is $\pi/2$, and the angle at $\hat s_3$ is $\pi/3$. We will call a pure 2-dimensional simplicial complex $X$ a \emph{$B_3$ (simplicial) complex} if for every 2-simplex $\sigma$ of $X$, there is an isomorphism $m_\sigma : \sigma \to \Delta$, and if for every vertex $v$ and 2-simplices $\sigma$ and $\tau$ containing $v$, we have $m_\sigma(v) = m_\tau(v)$. (In other words, the 1-skeleton of $X$ has a coloring by $\{\hat s_1, \hat s_2, \hat s_3\}$.) If $v$ is a vertex of a 2-simplex $\sigma$, we will call $m_\sigma(v)$ the \emph{type} of the vertex $v$. There is a canonical metric on $X$ obtained by pulling back the metric of $\Delta$ along the maps $m_\sigma$, which we call the $B_3$ metric on $X$. Our main theorem is the following.

\catonecrit

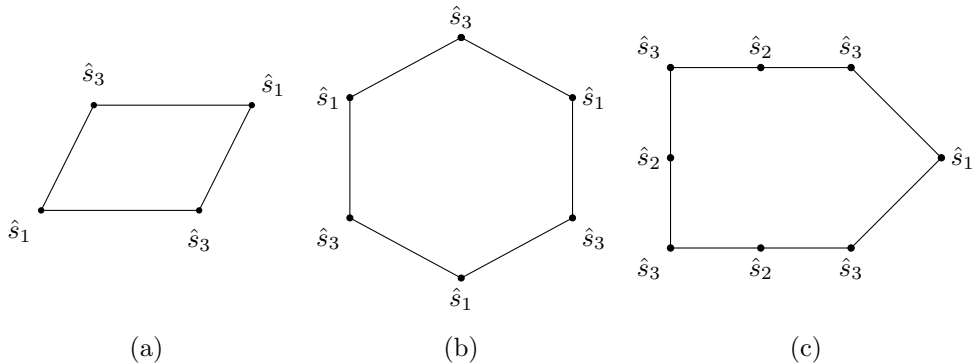
\begin{figure}[!ht]
\centering
\begin{tikzpicture}[scale=0.49]
  \matrix[matrix of nodes,column sep=-5pt,nodes={anchor=center, minimum height=0.1cm, align=flush center}]{
    \begin{tikzpicture}[scale=0.7]
        \coordinate (s3 top) at (1,2) ;
        \coordinate (s1 top) at (4,2);
        \coordinate (s1 bottom) at (0,0) ;
        \coordinate (s3 bottom) at (3,0) ;
        
        \draw (s3 top) -- (s1 top)  -- (s3 bottom) -- 
            (s1 bottom) -- (s3 top);

        \filldraw (s3 top) circle (0.05cm);
        \filldraw (s1 top) circle (0.05cm);
        \filldraw (s1 bottom) circle (0.05cm);
        \filldraw (s3 bottom) circle (0.05cm);

        \node[label={[label distance=0mm]above: {$\hat s_3$}}] 
            at (s3 top)  {} ;
        \node[above right] at (s1 top) {$\hat s_1$} ;
        \node[below left] at (s1 bottom) {$\hat s_1$} ;
        \node[label={[label distance=0mm]below: {$\hat s_3$}}] 
            at (s3 bottom)  {} ;       
    \end{tikzpicture}
    &
     \begin{tikzpicture}[scale=0.8]
    \coordinate (s1 top) at (0,2) ;
    \coordinate (s1 right) at (1.85,-1);
    \coordinate (s1 left) at (-1.85,-1) ;
    \coordinate (s2 bottom) at (0,-2) ;
    \coordinate (s2 right) at (1.85,1) ;
    \coordinate (s2 left) at (-1.85,1) ;
    
    \draw (s1 top) -- (s2 right)  -- (s1 right) -- 
        (s2 bottom) -- (s1 left) -- (s2 left) -- (s1 top);

    \filldraw (s1 top) circle (0.05cm);
    \filldraw (s2 right) circle (0.05cm);
    \filldraw (s1 right) circle (0.05cm);
    \filldraw (s2 bottom) circle (0.05cm);
    \filldraw (s1 left) circle (0.05cm);
    \filldraw (s2 left) circle (0.05cm);
    \filldraw (s1 top) circle (0.05cm);

    \node[above] at (s1 top) {$\hat s_3$} ;
    \node[right] at (s2 right) {$\hat s_1$} ;
    \node[below right] at (s1 right) {$\hat s_3$} ;
    \node[below] at (s2 bottom)  {$\hat s_1$} ;
    \node[below left] at (s1 left) {$\hat s_3$} ;
    \node[left] at (s2 left) {$\hat s_1$} ;
    \end{tikzpicture}
    &
     \begin{tikzpicture}[scale=1.2]
        \coordinate (s31) at (0,0) ;
        \coordinate (s21) at (1,0) ;
        \coordinate (s32) at (2,0);
        \coordinate (s1) at (3,-1) ;
        \coordinate (s33) at (2,-2) ;
        \coordinate (s22) at (1,-2) ;
        \coordinate (s34) at (0,-2) ;
        \coordinate (s23) at (0,-1) ;
        
        \filldraw (s31) circle (0.035cm);
        \filldraw (s21) circle (0.035cm);
        \filldraw (s32) circle (0.035cm);
        \filldraw (s1)  circle (0.035cm);
        \filldraw (s33) circle (0.035cm);
        \filldraw (s22) circle (0.035cm);
        \filldraw (s34) circle (0.035cm);
        \filldraw (s23) circle (0.035cm);

        \draw (s31) -- (s21)  -- (s32) -- (s1) --(s33) -- (s22) -- (s34)  -- (s23) -- (s31);

        \node[above left] at (s31) {$\hat s_3$} ;
        \node[above] at (s21) {$\hat s_2$} ;
        \node[above] at (s32) {$\hat s_3$} ;
        \node[right] at (s1)  {$\hat s_1$} ;
        \node[below] at (s33) {$\hat s_3$} ;
        \node[below] at (s22) {$\hat s_2$} ;
        \node[below left] at (s34) {$\hat s_3$} ;
        \node[left] at (s23) {$\hat s_2$} ;
        \end{tikzpicture}
    \\
    (a) & (b) & (c)
    \\
    };
    \end{tikzpicture}
    
    \caption{Some short closed edge paths}
    \label{fig:checking edge loops}
\end{figure}

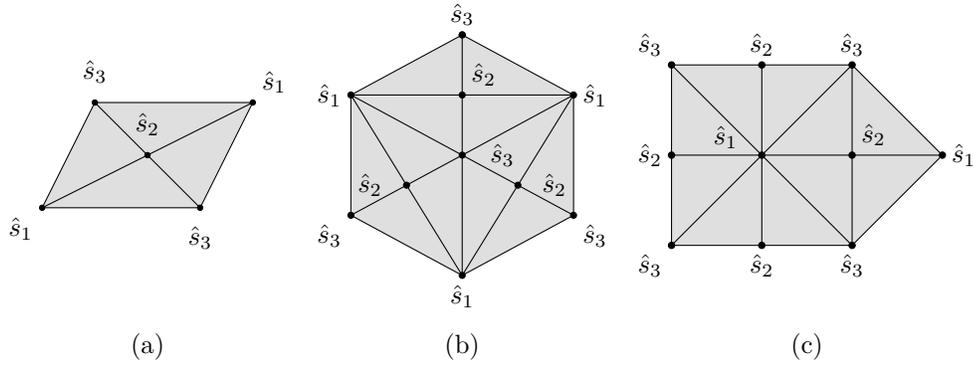
\begin{figure}[!ht]
\centering
\begin{tikzpicture}[scale=0.49]
  \matrix[matrix of nodes,column sep=-5pt,nodes={anchor=center, minimum height=0.1cm, align=flush center}]{
    \begin{tikzpicture}[scale=0.7]
        \coordinate (s3 top) at (1,2) ;
        \coordinate (s1 top) at (4,2);
        \coordinate (s1 bottom) at (0,0) ;
        \coordinate (s3 bottom) at (3,0) ;
        
        \draw[fill=gray!25] (s3 top) -- (s1 top)  -- (s3 bottom) -- 
            (s1 bottom) -- (s3 top);
        \draw[name path=A] (s3 top) -- (s3 bottom);
        \draw[name path=B] (s1 top) -- (s1 bottom);
        
        \path [name intersections={of=A and B,by=s2 middle}];

        \filldraw (s3 top) circle (0.05cm);
        \filldraw (s1 top) circle (0.05cm);
        \filldraw (s1 bottom) circle (0.05cm);
        \filldraw (s3 bottom) circle (0.05cm);
        \filldraw (s2 middle) circle (0.05cm);

        \node[label={[label distance=0mm]above: {$\hat s_3$}}] 
            at (s3 top)  {} ;
        \node[above right] at (s1 top) {$\hat s_1$} ;
        \node[below left] at (s1 bottom) {$\hat s_1$} ;
        \node[label={[label distance=0mm]below: {$\hat s_3$}}] 
            at (s3 bottom)  {} ; 
        \node[label={[label distance=0mm]90:{$\hat s_2$}}] 
            at (s2 middle) {} ;  
    \end{tikzpicture}
    &
    \begin{tikzpicture}[scale=0.8]
    \coordinate (s1 top) at (0,2) ;
    \coordinate (s1 right) at (1.85,-1);
    \coordinate (s1 left) at (-1.85,-1) ;
    \coordinate (s2 bottom) at (0,-2) ;
    \coordinate (s2 right) at (1.85,1) ;
    \coordinate (s2 left) at (-1.85,1) ;
    \coordinate (center) at (0,0);
    \coordinate (new top) at (0,1);
    \coordinate (new left) at (-0.925,-0.5);
    \coordinate (new right) at (0.925,-0.5);
    
    \draw[fill=gray!25] (s1 top) -- (s2 right)  -- (s1 right) -- 
        (s2 bottom) -- (s1 left) -- (s2 left) -- (s1 top);

    \draw (s1 top) -- (s2 bottom);
    \draw (s1 right) -- (s2 left);
    \draw (s1 left) -- (s2 right);
    \draw (s2 left) -- (s2 right);
    \draw (s2 left) -- (s2 bottom);
    \draw (s2 right) -- (s2 bottom);

    \filldraw (s1 top) circle (0.05cm);
    \filldraw (s2 right) circle (0.05cm);
    \filldraw (s1 right) circle (0.05cm);
    \filldraw (s2 bottom) circle (0.05cm);
    \filldraw (s1 left) circle (0.05cm);
    \filldraw (s2 left) circle (0.05cm);
    \filldraw (s1 top) circle (0.05cm);
    \filldraw (center) circle (0.05cm);
    \filldraw (new top) circle (0.05cm);
    \filldraw (new left) circle (0.05cm);
    \filldraw (new right) circle (0.05cm);

    \node[above] at (s1 top) {$\hat s_3$} ;
    \node[right] at (s2 right) {$\hat s_1$} ;
    \node[below right] at (s1 right) {$\hat s_3$} ;
    \node[below] at (s2 bottom)  {$\hat s_1$} ;
    \node[below left] at (s1 left) {$\hat s_3$} ;
    \node[left] at (s2 left) {$\hat s_1$} ;
    \node[right=0.25] at (center) {$\hat s_3$} ;
    \node[above right] at (new top) {$\hat s_2$};
    \node[right=0.2] at (new right) {$\hat s_2$};
    \node[left=0.2] at (new left) {$\hat s_2$};
    \end{tikzpicture}
    &
    \begin{tikzpicture}[scale=1.2]
        \coordinate (s31) at (0,0) ;
        \coordinate (s21) at (1,0) ;
        \coordinate (s32) at (2,0);
        \coordinate (s1) at (3,-1) ;
        \coordinate (s33) at (2,-2) ;
        \coordinate (s22) at (1,-2) ;
        \coordinate (s34) at (0,-2) ;
        \coordinate (s23) at (0,-1) ;

        \coordinate (s11) at (1,-1) ;
        \coordinate (s24) at (2,-1) ;
        
        \draw[fill=gray!25] (s31) -- (s21)  -- (s32) -- (s1) --(s33) -- (s22) -- (s34)  -- (s23) -- (s31);
        
        \filldraw (s31) circle (0.035cm);
        \filldraw (s21) circle (0.035cm);
        \filldraw (s32) circle (0.035cm);
        \filldraw (s1)  circle (0.035cm);
        \filldraw (s33) circle (0.035cm);
        \filldraw (s22) circle (0.035cm);
        \filldraw (s34) circle (0.035cm);
        \filldraw (s23) circle (0.035cm);
        \filldraw (s24) circle (0.035cm);
        \filldraw (s11) circle (0.035cm);

        \draw (s31) -- (s33);
        \draw (s32) -- (s34);
        \draw (s32) -- (s33);
        \draw (s21) -- (s22);
        \draw (s23) -- (s1);

        \node[above left] at (s31) {$\hat s_3$} ;
        \node[above] at (s21) {$\hat s_{2}$} ;
        \node[above] at (s32) {$\hat s_3$} ;
        \node[right] at (s1)  {$\hat s_1$} ;
        \node[below] at (s33) {$\hat s_3$} ;
        \node[below] at (s22) {$\hat s_{2}$} ;
        \node[below left] at (s34) {$\hat s_3$} ;
        \node[left] at (s23) {$\hat s_2$} ;
        \node[above right] at (s24) {$\hat s_{2}$} ;
        \node at (0.6,-0.8) {$\hat s_1$};
     \end{tikzpicture}
    \\
    (a) & (b) & (c)
    \\
    };
    \end{tikzpicture}
    
    \caption{Filled edge paths}
    \label{fig:filling the edge loops}
\end{figure}

\begin{figure}[!ht]
\centering
    \begin{tikzpicture}
    \coordinate (s31) at (-1,0   );
    \coordinate (s32) at ( 1,0   );
    \coordinate (s33) at ( 2,2   );
    \coordinate (s34) at ( 0,3.75);
    \coordinate (s35) at (-2,2   );

    \filldraw (s31) circle (0.05cm);
    \filldraw (s32) circle (0.05cm);
    \filldraw (s33) circle (0.05cm);
    \filldraw (s34) circle (0.05cm);
    \filldraw (s35) circle (0.05cm);

    \coordinate (s21) at (0,0);
    \coordinate (s22) at (1.5,1);
    \coordinate (s23) at (1,2.875);
    \coordinate (s24) at (-1,2.875);
    \coordinate (s25) at (-1.5,1);

    \filldraw (s21) circle (0.05cm);
    \filldraw (s22) circle (0.05cm);
    \filldraw (s23) circle (0.05cm);
    \filldraw (s24) circle (0.05cm);
    \filldraw (s25) circle (0.05cm);

    \draw (s31) --(s32) --(s33) --(s34) --(s35) -- (s31);

    \node[below] at (s31) {$\hat s_3$};
    \node[below] at (s32) {$\hat s_3$};
    \node[right] at (s33) {$\hat s_3$};
    \node[above] at (s34) {$\hat s_3$};
    \node[left ] at (s35) {$\hat s_3$};

    \node[below] at (s21) {$\hat s_2$};
    \node[below right] at (s22) {$\hat s_2$};
    \node[above right] at (s23) {$\hat s_2$};
    \node[above left ] at (s24) {$\hat s_2$};
    \node[below left ] at (s25) {$\hat s_2$};
\end{tikzpicture}
    \caption{A non-admissible loop}
    \label{fig:bad 10 cycle}
\end{figure}
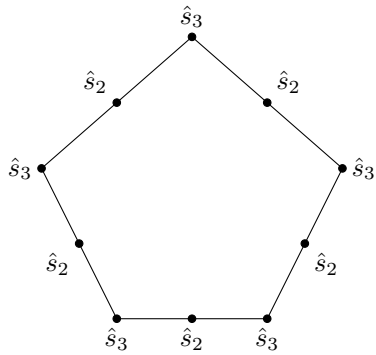

In order to show that these complexes are $\cat(1)$, we utilize the following criteria of Bowditch \cite[\S 3]{bowditch1995notes}

\begin{prop} \label{prop:bowditchcriteria}
    If $X$ is locally $\cat(1)$, then $X$ is $\cat(1)$ if and only if every rectifiable loop of length $< 2\pi$ is ``shrinkable''.
\end{prop}

We will call a rectifiable path \emph{short} if it has length $< 2\pi$. A loop is a closed path.
We say that rectifiable loops $\alpha$ and $\beta$ are \emph{$2\pi$-homotopic} (or \emph{short homotopic}) and write $\alpha \ssim \beta$ if there is a (free) homotopy from $\alpha$ to $\beta$ through short loops (i.e., each loop in the homotopy is a short rectifiable loop). If $\gamma \ssim 0$ (where $0$ is any constant loop), we say that $\gamma$ is \emph{shrinkable}. 

The conditions (2), (3), and (4) guarantee that the $B_3$ metric on $X$ is locally $\cat(1)$, so we only need to show that short loops are shrinkable. 
In order to show that there are no unshrinkable short loops in $X$, we first show that an arbitrary short loop is short homotopic to a loop remaining in the 1-skeleton of $X$ (an ``edge loop''), then show that all short edge loops are short homotopic to one of the loops in Figure \ref{fig:checking edge loops} or Figure \ref{fig:bad 10 cycle}. Then conditions (5) and (6) guarantee these loops are shrinkable. 

For the rest of Section \ref{sec:ComCrit}, we assume $X$ is a $B_3$ simplicial complex endowed with its $B_3$ metric, satisfying the hypotheses of Theorem \ref{thm:B3 CAT1 Criteria}.

Before continuing, we provide a quick observation.

\begin{lem}
    $X$ is a flag complex.
\end{lem}

\begin{proof}
    Let $\{v_1,v_2,v_3\}$ be a set of vertices of $X$ which are pairwise adjacent, with $e_{ij}$ the edge between $v_i$ and $v_j$. Since $X$ is a pure 2-dimensional complex, each $e_{ij}$ is contained in a 2-simplex $\sigma_{ij}$. By considering the maps $m_{\sigma_{ij}}$, each of $v_1$, $v_2$, and $v_3$ must have distinct types. Without loss of generality, suppose $v_i$ has type $\hat s_i$ for each $i$. Then since $lk(v_2,X)$ is complete bipartite, there is an edge between the points in the link representing $v_1$ and $v_3$; this edge represents the fact that $\{v_1,v_2,v_3\}$ spans a simplex.

    Suppose $\{v_1,v_2,v_3, \dots v_n\}$ is a set of pairwise adjacent vertices. By similar reasoning to above, each of these vertices must have distinct labels. However, there are only $3$ labels, so we cannot have $n > 3$. 
\end{proof}

\subsection{Shrinking to edge loops}

We now perform the reduction from arbitrary short loops to loops which stay in the 1-skeleton. 

The proof of the following lemma is identical to the proof of \cite[Lemma 5.5]{goldman2023cat0} given by the first author.

\begin{lem} \label{lem:intersectionofstars}
    Suppose $v_1,v_2 \in X$ are distinct vertices of type $\hat s_2$. Then $St(v_1) \cap St(v_2)$ is either empty, a single vertex, or a single (closed) edge.
\end{lem}

Let $C(B_3)$ denote the unit 2-sphere with its standard tessellation by copies of the $B_3$ simplex $\Delta$, along with the vertex labeling induced by $\Delta$. 
Important to the forthcoming proofs is the notion of \emph{developing} a geodesic from $X$ to $C(B_3)$ \cite{charney2004deligne}. The notion of developing in the context of locally $\cat(1)$ spaces originated in \cite{elder2002curvature}. Charney was able to simplify this method by developing these geodesics onto the spherical Coxeter complex rather than an arbitrary 2-sphere.
 
Suppose $\gamma$ is a local geodesic of $X$ which does not intersect any vertices. Then the sequence of 2-simplices which intersect $\gamma$ can be mapped one-by-one to a sequence of adjacent (and distinct) 2-simplices of $C(B_3)$ such that the vertex labeling is preserved. Since adjacency is preserved, these maps may be glued to form a locally injective map from a subcomplex of $X$ onto a subcomplex of $C(B_3)$. The image of $\gamma$ under this map is a local geodesic $\overline \gamma$ in $C(B_3)$, called the \emph{development of $\gamma$ (onto $C(B_3)$)}. %

\begin{figure}[!ht]
    \centering
    \begin{tikzpicture}
        \coordinate (A) at (0,0);
        \coordinate (B) at (2,0);
        \coordinate (C) at (0,2);
        \coordinate (D) at (2,2);
        \coordinate (E) at (1,1);

        \draw[fill=gray!25] (A) to[bend right] (B) to[bend right] (D) to[bend right] (C) to[bend right] (A);

        \draw (E) -- (A);
        \draw (E) -- (B);
        \draw (E) -- (C);
        \draw (E) -- (D);

        \filldraw (A) circle (0.05cm);
        \filldraw (B) circle (0.05cm);
        \filldraw (C) circle (0.05cm);
        \filldraw (D) circle (0.05cm);
        \filldraw (E) circle (0.05cm);

        \node[below left] at (A) {$\hat s_1$};
        \node[above right] at (D) {$\hat s_1$};
        \node[above left] at (C) {$\hat s_3$};
        \node[below right] at (B) {$\hat s_3$};

        \node[right=0.2cm] at (E) {$\hat s_2$};

        \coordinate (G1) at (1.1,2.285);
        \coordinate (G2) at (1.15,-0.29);
        \draw[line width=0.05cm, color=red] (G1) to[in=100,out=-100] (G2);

        \node[color=red] at (0.8,1.75) {$\gamma$};
    \end{tikzpicture}
    \caption{Developing a geodesic through a vertex of type $\hat s_2$}
    \label{fig:quadofB}
\end{figure}
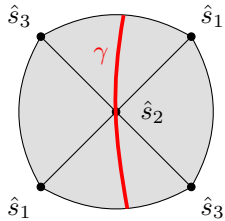

We may also develop local geodesics which intersect vertices of type $\hat s_2$ in the following way. 
If $\gamma$ is a local geodesic which passes through a vertex $v$ of type $\hat s_2$, then it intersects the $\varepsilon$-sphere of $v$ at two points which are distance at least $\pi$ apart in the round/unit/link metric on the $\varepsilon$-sphere. 
By its definition, this sphere is isometric to the link of $v$, so the points on the sphere correspond to points in the link which are distance no less than $\pi$ apart.
Since the link is complete bipartite and the lengths of edges are $\pi/2$, the diameter of the link is precisely $\pi$, so these points are exactly distance $\pi$ apart.
Moreover, these points are contained in a loop of edge length $4$, corresponding to four 2-simplices of $X$ which form a quadrilateral as in Figure \ref{fig:quadofB}. These simplices can then be mapped down to $C(B_3)$, and the image $\overline \gamma$ of $\gamma$ will still be locally geodesic at the image of $v$.
Because we can develop over a vertex of type $\hat s_2$ but not $\hat s_1$ or $\hat s_3$ in general, we will call vertices of type $\hat s_2$ \emph{nonsingular} and vertices of type $\hat s_1$ and $\hat s_3$ \emph{singular}.

\begin{lem} \label{lem:noverteximplieslength}
    If $\gamma$ is a nontrivial short closed local geodesic which intersects no singular vertices, then $\ell(\gamma) \geq \pi$.
\end{lem}

\begin{figure}[!ht]
    \begin{tikzpicture}[scale=0.49]
     \matrix[matrix of nodes,column sep=-5pt,nodes={anchor=center, minimum height=0.1cm, align=flush center}]{
    \begin{tikzpicture}[scale=0.5]
        \coordinate (left s3 top) at (-1.3,1.5);
        \coordinate (left s1 left) at (-2.3,-0.5);
    
        \coordinate (middle s1 top) at (1,2) ;
        \coordinate (middle s3 left) at (0,0);
        \coordinate (middle s3 right) at (2,0) ;
        \coordinate (middle s1 bottom) at (1,-2) ;

        \coordinate (right s3 top) at (3.3,1.5) ;
        \coordinate (right s1 bottom) at (4.3,-0.5) ;

        \coordinate (left dashed) at (-1.8,0.5) ;
        \coordinate (right dashed) at (3.8,0.5) ;

        \draw (middle s1 top) -- (left s3 top) -- 
            (left s1 left) --(middle s3 left);
        
        \draw (middle s1 top) -- (middle s3 left)  -- (middle s1 bottom) -- 
            (middle s3 right) -- (middle s1 top);

        \draw (middle s1 top) -- (right s3 top) --
            (right s1 bottom) -- (middle s3 right);

        \draw[dashed] (left dashed) -- (right dashed);

        \filldraw (left s3 top) circle (0.05cm);
        \filldraw (left s1 left) circle (0.05cm);

        \filldraw (middle s1 top) circle (0.05cm);
        \filldraw (middle s3 left) circle (0.05cm);
        \filldraw (middle s3 right) circle (0.05cm);
        \filldraw (middle s1 bottom) circle (0.05cm);

        \filldraw (right s3 top) circle (0.05cm);
        \filldraw (right s1 bottom) circle (0.05cm);

        \node[above] at (left s3 top) {$\hat s_3$} ;
        \node[left] at (left s1 left) {$\hat s_1$} ;

        \node[label={ [label distance=0mm]above: {$\hat s_1$} }] 
            at (middle s1 top)  {} ;
        \node[below left] at (middle s3 left) {$\hat s_3$} ;
        \node[below right] at (middle s3 right) {$\hat s_3$} ;
        \node[label={[label distance=0mm]below: {$\hat s_1$}}] 
            at (middle s1 bottom)  {} ;          

        \node[above] at (right s3 top) {$\hat s_3$} ;
        \node[right] at (right s1 bottom) {$\hat s_1$} ;
        
    \end{tikzpicture}

    &

    \begin{tikzpicture}[scale=0.49]
        \coordinate (left s3 top) at (-1.3,1.5);
        \coordinate (left s1 left) at (-2.3,-0.5);
    
        \coordinate (middle s1 top) at (1,2) ;
        \coordinate (middle s3 left) at (0,0);
        \coordinate (middle s3 right) at (2,0) ;
        \coordinate (middle s1 bottom) at (1,-2) ;

        \coordinate (right s1 top) at (4.3,0.5) ;
        \coordinate (right s3 right) at (3.3,-1.5) ;

        \coordinate (left dashed) at (-1.55,1);
        \coordinate (right dashed) at (3.55,0-1);

        \draw (middle s1 top) -- (left s3 top) -- 
            (left s1 left) --(middle s3 left);
        
        \draw (middle s1 top) -- (middle s3 left)  -- (middle s1 bottom) -- 
            (middle s3 right) -- (middle s1 top);

        \draw (middle s3 right) -- (right s1 top) --
            (right s3 right) -- (middle s1 bottom);

        \draw[dashed] (left dashed) -- (right dashed);

        \filldraw (left s3 top) circle (0.05cm);
        \filldraw (left s1 left) circle (0.05cm);

        \filldraw (middle s1 top) circle (0.05cm);
        \filldraw (middle s3 left) circle (0.05cm);
        \filldraw (middle s3 right) circle (0.05cm);
        \filldraw (middle s1 bottom) circle (0.05cm);

        \filldraw (right s1 top) circle (0.05cm);
        \filldraw (right s3 right) circle (0.05cm);

        \node[above] at (left s3 top) {$\hat s_3$} ;
        \node[left] at (left s1 left) {$\hat s_1$} ;

        \node[label={[label distance=0mm]above: {$\hat s_1$}}] 
            at (middle s1 top)  {} ;
        \node[below left] at (middle s3 left) {$\hat s_3$} ;
        \node[above right] at (middle s3 right) {$\hat s_3$} ;
        \node[label={[label distance=0mm]below: {$\hat s_1$}}] 
            at (middle s1 bottom)  {} ;          

        \node[above] at (right s1 top) {$\hat s_1$} ;
        \node[right] at (right s3 right) {$\hat s_3$} ;
        
    \end{tikzpicture} 
    
    &

    \begin{tikzpicture}[scale=0.49]
        \coordinate (left s3 top) at (-1.3,1.5);
        \coordinate (left s1 left) at (-2.3,-0.5);
    
        \coordinate (middle s1 top) at (1,2) ;
        \coordinate (middle s3 left) at (0,0);
        \coordinate (middle s3 right) at (2,0) ;
        \coordinate (middle s1 bottom) at (1,-2) ;

        \coordinate (bottom s3 bottom) at (-1.1,-2) ;

        \draw (middle s1 top) -- (left s3 top) -- 
            (left s1 left) --(middle s3 left);
        
        \draw (middle s1 top) -- (middle s3 left)  -- (middle s1 bottom) -- 
            (middle s3 right) -- (middle s1 top);

        \draw (middle s1 bottom) -- (bottom s3 bottom) -- (left s1 left);

        \filldraw (left s3 top) circle (0.05cm);
        \filldraw (left s1 left) circle (0.05cm);

        \filldraw (middle s1 top) circle (0.05cm);
        \filldraw (middle s3 left) circle (0.05cm);
        \filldraw (middle s3 right) circle (0.05cm);
        \filldraw (middle s1 bottom) circle (0.05cm);

        \filldraw (bottom s3 bottom) circle (0.05cm);

        \node[above] at (left s3 top) {$\hat s_3$} ;
        \node[left] at (left s1 left) {$\hat s_1$} ;

        \node[label={[label distance=0mm]above: {$\hat s_1$}}] 
            at (middle s1 top)  {} ;
        \node[below left] at (middle s3 left) {$\hat s_3$} ;
        \node[above right] at (middle s3 right) {$\hat s_3$} ;
        \node[label={[label distance=0mm]below: {$\hat s_1$}}] 
            at (middle s1 bottom)  {} ;          

        \node[below] at (bottom s3 bottom) {$\hat s_3$} ;
        
    \end{tikzpicture} \\
    };
    \end{tikzpicture}
    
    \caption{Quadrilateral galleries of length three}
    \label{fig:possible galleries of quadrilaterals}
    \end{figure}
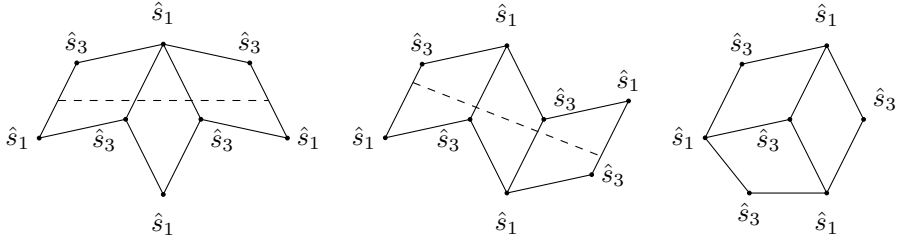

\begin{proof}
    If $\ell(\gamma) < \pi$, then there are no more than 3 (not necessarily distinct) vertices $v_1,v_2,v_3$ of type $\hat s_2$ such that $\gamma \cap st(v_i) \not=\varnothing$. 
    (Otherwise if $\gamma$ intersects 4 or more stars of distinct $\hat s_2$-type vertices, then by developing, we see that $\ell(\gamma) > \pi$.) 
    We know $|\{v_1,v_2,v_3\}| \not= 1$, since this would contradict the fact that $X$ is simplicial, and we know $|\{v_1,v_2,v_3\}| \not= 2$ since this would contradict Lemma \ref{lem:intersectionofstars}. 
    Thus each vertex $v_1$, $v_2$, and $v_3$ are distinct. By developing, this gives one of the three galleries of quadrilaterals in Figure \ref{fig:possible galleries of quadrilaterals}. In this figure, the dashed lines are examples of geodesics, and one should think of these quadrilaterals as ``filled'' as in 
    Figure \ref{fig:checking edge loops}(a), where the central vertices are the images of $v_1$, $v_2$, and $v_3$. 

    In $C(B_3)$, the development $\overline \gamma$ is a local (hence global) geodesic of length $< \pi$, so cannot be closed since $C(B_3)$ is $\cat(1)$. Thus $\overline \gamma$ is a geodesic segment in one of these galleries with endpoints on the boundary, and the edges/vertices containing these endpoints are identified in $X$.
    If any edges or vertices in the left-most or right-most diagrams are identified in $X$, this would contradict Lemma \ref{lem:intersectionofstars}, so $\gamma$ cannot develop onto galleries of these forms. 
    
    Consider the middle gallery. None of the vertices of the middle quadrilateral of this gallery can be identified with any other vertices in the gallery, so the only possible identifications are left-most vertices (or edge) with right-most vertices (resp., edge) of the diagram. But in each case, this induces a loop of length 4 of the type in Figure \ref{fig:checking edge loops}(a) which can be filled by a vertex $v_4$ of type $\hat s_2$, by condition (5). If $v_4 = v_i$ for some $i = 1,2,3$, then this contradicts Lemma \ref{lem:intersectionofstars}, since some pair $v_i$ and $v_j$ ($i,j = 1,2,3$ distinct) would have $St(v_i) \cap St(v_j)$ either the union of an edge and a (disjoint) vertex, or the union of two (distinct) edges. If $v_4 \not= v_i$ for each $i = 1,2,3$, then this again contradicts Lemma \ref{lem:intersectionofstars}: $St(v_4)$ will intersect one of the $St(v_i)$ ($i = 1,2,3$) in two distinct edges, depending on the labeling of the quadrilaterals. Regardless, we know $\gamma$ cannot develop onto a gallery of this shape, and thus $\ell(\gamma) \geq \pi$.
\end{proof}

\begin{lem} \label{lem:novertextovertex}
    Suppose $\gamma$ is a short closed local geodesic. Then there is a short closed loop $\gamma' \ssim \gamma$ such that $\gamma'$ intersects a singular vertex of $X$.
\end{lem}

\begin{proof}
    Of course, if $\gamma$ already intersects a singular vertex, there is nothing to show. So, assume $\gamma$ intersects no singular vertices of $X$.
    By Lemma \ref{lem:noverteximplieslength}, we know $\ell(\gamma) \geq \pi$, so let $a$ and $b$ be two values such that $d(\gamma(a), \gamma(b)) = \pi$. 
    Develop $\gamma$ to a local geodesic $\overline \gamma$ in $C(B_3)$.
    Then $\overline\gamma(a)$ and $\overline\gamma(b)$ are antipodal points of the unit sphere $C(B_3)$.
    There is a natural homotopy of the segment $\overline\gamma|_{[a,b]}$ in $C(B_3)$ obtained by rotating relative to its endpoints until the segment intersects a singular vertex of $C(B_3)$. Each curve in the homotopy (until the final curve) contains no singular vertices in its interior, and thus can be lifted back to $X$ to obtain a homotopy of $\gamma|_{[a,b]}$ relative to its endpoints through segments of length $\pi$ to a segment which contains a vertex. (The final curve can be added by taking a pointwise limit of this homotopy.) By concatenating this homotopy with the rest of $\gamma$, we obtain a short homotopy from $\gamma$ to a loop containing a vertex.
\end{proof}

\begin{figure}[!ht]
    \centering
    
        \begin{tikzpicture}[scale=0.9,every node/.style={scale=0.9}]
            \matrix[matrix of nodes,column sep=-2pt,nodes={anchor=center, minimum height=0.1cm, align=flush center}]{
    \begin{tikzpicture}[x=1cm,y=1cm,rotate=90]

        \begin{scope}
          \draw[even odd rule,inner color=white,outer color=gray!45]
          (0,0) 
          arc [start angle=45, end angle=135, radius=4] 
          arc [start angle=-135, end angle=-45, radius=4] 
          ;
    
        \path
        (0,0)
        arc [start angle=45, end angle=135, radius=4]
        node[pos=0, circle, fill=black, inner sep=1pt] (right s1) {}
        node[pos=0.5, circle, fill=black, inner sep=1pt] (top s2) {}
        node[pos=0.25, circle, fill=black, inner sep=1pt] (top right s3) {}
        node[pos=0.75, circle, fill=black, inner sep=1pt] (top left s3) {}
        node[pos=1, circle, fill=black, inner sep=1pt] (left s1) {}
        ;
    
        \path
        (left s1)
        arc [start angle=-135, end angle=-45, radius=4]
        node[pos=0.5, circle, fill=black, inner sep=1pt] (bottom s1) {}
        node[pos=0.2, circle, fill=black, inner sep=1pt] (bottom left s2) {}
        node[pos=0.8, circle, fill=black, inner sep=1pt] (bottom right s2) {}
        ;
       
        \end{scope}
    
        \node[below] at (left s1) {$\hat s_1$};
        \node[above] at (right s1) {$\hat s_1$};
        \node[left] at (top s2) {$\hat s_2$};
        \node[left] at (top right s3) {$\hat s_3$};
        \node[left] at (top left s3) {$\hat s_3$};
        \node[right] at (bottom s1) {$\hat s_1$};
        \node[right] at (bottom left s2) {$\hat s_2$};
        \node[right] at (bottom right s2) {$\hat s_2$};
    
        \draw (top s2) -- (bottom s1);
        \draw (bottom s1) to[out=135,in=-60] (top left s3);
        \draw (bottom s1) to[out=45,in=240] (top right s3);
        \draw (top left s3) to[out=240,in=100] (bottom left s2);
        \draw (top right s3) to[out=-60,in=80] (bottom right s2);
    
        \draw[line width=0.05cm, color=red] (right s1) to[out=160,in=20] 
            node[pos=0.4, right] (path){$\gamma_0$} (left s1);
    
    \end{tikzpicture}
    
        &
    \begin{tikzpicture}[x=1cm,y=1cm,rotate=90]

        \begin{scope}
    
          \draw[even odd rule,inner color=white,outer color=gray!45] 
          (0,0) 
          arc [start angle=30, end angle=150, radius=3.3] 
          arc [start angle=-150, end angle=-30, radius=3.3] 
          ;
    
        \path
        (0,0)
        arc [start angle=30, end angle=150, radius=3.3]
        node[pos=0, circle, fill=black, inner sep=1pt] (right s2) {}
        node[pos=0.5, circle, fill=black, inner sep=1pt] (top s1) {}
        node[pos=0.18, circle, fill=black, inner sep=1pt] (top right s3) {}
        node[pos=0.82, circle, fill=black, inner sep=1pt] (top left s3) {}
        node[pos=1, circle, fill=black, inner sep=1pt] (left s2) {}
        ;
    
        \path
        (left s2)
        arc [start angle=-150, end angle=-30, radius=3.3]
        node[pos=0.5, circle, fill=black, inner sep=1pt] (bottom s2) {}
        node[pos=0.225, circle, fill=black, inner sep=1pt] (bottom left s1) {}
        node[pos=0.775, circle, fill=black, inner sep=1pt] (bottom right s1) {}
        ;
       
        \end{scope}

        \node[below] at (left s2) {$\hat s_2$};
        \node[above] at (right s2) {$\hat s_2$};
        \node[left] at (top s1) {$\hat s_1$};
        \node[left] at (top right s3) {$\hat s_3$};
        \node[left] at (top left s3) {$\hat s_3$};
        \node[right] at (bottom s2) {$\hat s_2$};
        \node[right] at (bottom left s1) {$\hat s_1$};
        \node[right] at (bottom right s1) {$\hat s_1$};
    
        \draw (top s1) -- (bottom s2);
        \draw (bottom left s1) to[out=125,in=-110] (top left s3);
        \draw (bottom right s1) to[out=55,in=290] (top right s3);
        \draw [name path=curve 1] (bottom left s1) to[out=80,in=210] (top s1);
        \draw [name path=curve 4] (bottom right s1) to[out=100,in=-30] (top s1);
        \draw [name path=curve 3] (bottom left s1) to[out=15,in=220] (top right s3);
        \draw [name path=curve 2] (bottom right s1) to[out=165,in=-40] (top left s3);
    
        \fill[name intersections={of=curve 1 and curve 2, by={middle left s2}}]
            (middle left s2) circle (1.5pt); 
        \node[anchor=-170,outer sep=3pt] at (middle left s2) {$\hat s_2$};
    
        \fill[name intersections={of=curve 3 and curve 2, by={middle s2}}]
            (middle s2) circle (1.5pt); 
        \node[anchor=-150,outer sep=3pt] at (middle s2) {$\hat s_3$};
        
        \fill[name intersections={of=curve 4 and curve 3, by={middle right s2}}]
            (middle right s2) circle (1.5pt); 
        \node[anchor=165,outer sep=3pt] at (middle right s2) {$\hat s_2$};

        \draw[line width=0.05cm, color=red] (right s2) to[out=140,in=40]    
            node[pos=0.4, right] (path){$\gamma_0$} (left s2);
    
    \end{tikzpicture}
    &
    
    \begin{tikzpicture}[x=1cm,y=1cm,rotate=90]

        \begin{scope}
    
          \draw[even odd rule,inner color=white,outer color=gray!45] 
          (0,0) 
          arc [start angle=38, end angle=142, radius=3.6] 
          arc [start angle=-142, end angle=-38, radius=3.6] 
          ;
    
        \path
        (0,0)
        arc [start angle=38, end angle=142, radius=3.6]
        node[pos=0, circle, fill=black, inner sep=1pt] (right s3) {}
        node[pos=0.4, circle, fill=black, inner sep=1pt] (top s3) {}
        node[pos=0.16, circle, fill=black, inner sep=1pt] (top s2) {}
        node[pos=0.8, circle, fill=black, inner sep=1pt] (top s1) {}
        node[pos=1, circle, fill=black, inner sep=1pt] (left s3) {}
        ;
    
        \path
        (left s1)
        arc [start angle=-142, end angle=-38, radius=3.6]
        node[pos=0.4, circle, fill=black, inner sep=1pt] (bottom s3) {}
        node[pos=0.16, circle, fill=black, inner sep=1pt] (bottom s2) {}
        node[pos=0.8, circle, fill=black, inner sep=1pt] (bottom s1) {}
        ;
       
        \end{scope}
    
        \node[below] at (left s3) {$\hat s_3$};
        \node[above] at (right s3) {$\hat s_3$};
        \node[left] at (top s3) {$\hat s_3$};
        \node[left] at (top s2) {$\hat s_2$};
        \node[left] at (top s1) {$\hat s_1$};
        \node[right] at (bottom s3) {$\hat s_3$};
        \node[right] at (bottom s2) {$\hat s_2$};
        \node[right] at (bottom s1) {$\hat s_1$};
    
        \draw (bottom s2) to[out=110,in=235] (top s1);    
        \draw (bottom s1) to[out=55,in=-70] (top s2);    
        \draw (bottom s3) to[out=140,in=-80] (top s1);
        \draw (bottom s1) to[out=100,in=-40] (top s3);
        \draw [name path=curve 1] (bottom s3) to[out=45,in=260] (top s3);
        \draw [name path=curve 2] (bottom s1) to[out=135,in=-10] (top s1);
    
        \fill[name intersections={of=curve 1 and curve 2, by={middle s2}}]
            (middle s2) circle (1.5pt); 
        \node[anchor=-160,outer sep=3pt] at (middle s2) {$\hat s_2$};

        \draw[line width=0.05cm, color=red] (right s3) to[out=150,in=30] 
            node[pos=0.7, right] (path){$\gamma_0$} (left s3);

    \end{tikzpicture}
    
    \\
    (1) & (2) & (3) \\
    };
    \end{tikzpicture}
    \caption{The lunes with terminal vertices of type $\hat s_1$, $\hat s_2$, and $\hat s_3$, resp.}
    \label{fig:lunes}
\end{figure}
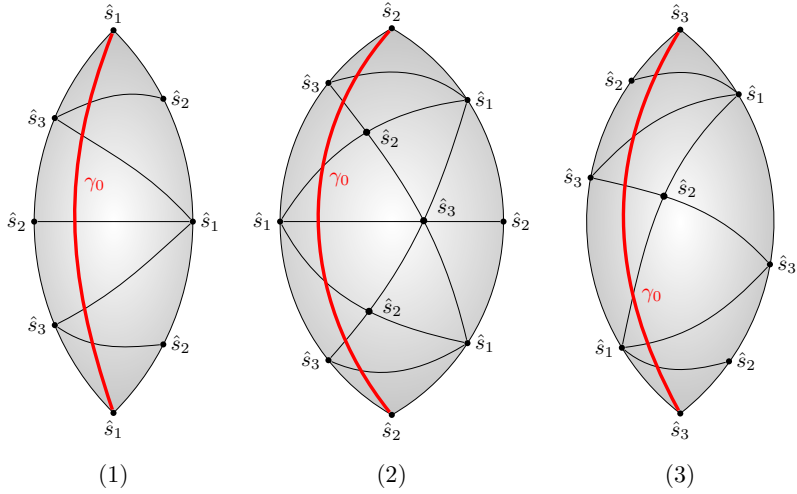

\begin{lem} \label{lem:vertextoedgepath}
    Suppose $\gamma$ is a short closed local geodesic which intersects a singular vertex. Then there is a short closed edge path $\gamma'$ with $\gamma' \ssim \gamma$.
\end{lem}

\begin{proof}
    Decompose $\gamma$ into a concatenation $\gamma_0\gamma_1\dots\gamma_n$ of geodesic segments which contain no singular vertices in their interior. Notice that for each $i$, this means either $\gamma_i$ is an edge path or contains no edges in its interior. Suppose $\gamma_i$ contains no edges. We can develop $\gamma_i$ onto $C(B_3)$, resulting in one of the diagrams (1) or (3) of Figure \ref{fig:lunes}. (The development can pass through the vertex of type $\hat s_2$ with no issues.) The endpoints of this development are at distance $\pi$ from each other, so we may rotate the development to lay on the boundary of its respective lune, and lift this rotation to get a length-preserving homotopy of $\gamma_i$ to an edge path. 
\end{proof}

The following lemma holds in much greater generality than our current setting, so we provide this more general proof here for reference.
 
\begin{prop} \label{prop:maintool}
    Suppose $Y$ is a locally $\cat(1)$ piecewise spherical simplicial complex with $\mathrm{Shapes}(Y)$ finite. Let $\gamma$ be a short loop based at $x_0 \in Y$. Then either $\gamma \ssim 0$, or there exists a closed local geodesic $\gamma'$ based at $x_0$ such that $\gamma \ssim \gamma'$.
\end{prop}

\begin{proof}
    Let $\Omega = \Omega(Y, x_0, 2\pi)$ denote the space of loops in $Y$ based at $x_0$ with length $< 2\pi$, equipped with the supremum metric $d_{\Omega}(\alpha,\beta) = \sup d_X(\alpha(t), \beta(t))$. Note that the length function $\ell : \Omega \to [0,\infty)$ is lower semi-continuous with respect to this metric. 
    Let $\Omega'$ denote the connected component of $\Omega$ containing $\gamma$, 
    and let $I = \inf \ell(\Omega')$. 
    Choose a sequence $\{\gamma_n\}$ in $\Omega'$ such that $\lim \ell(\gamma_n) \to I$.

    Since $Y$ is locally $\cat(1)$ and $\mathrm{Shapes}(Y)$ is finite, $Y$ is actually ``uniformly locally $\cat(1)$'', meaning that for sufficiently small $\varepsilon$, given any point $y \in Y$, the $\varepsilon$-neighborhood about $x$ is $\cat(1)$. (This follows quickly from Theorem I.7.39 and Lemma I.7.56 from \cite{bridson2013metric}.) This implies we can shrink the elements of the sequence $\{\gamma_n\}$ to piecewise geodesic with a uniformly bounded number of geodesic subsegments. In more detail:
    fix $n$, and let $U_{n,0}$ denote the $\varepsilon$ neighborhood about $x_{n,0} \coloneqq x_0$. Choose a point $x_{n,1}$ which is distance $> \varepsilon/2$ from $x_0$ and in the same connected component of $U_{n,0} \cap \gamma_n$ as $x_0$. Let $U_{n,1}$ denote the $\varepsilon$ ball around $x_{n,1}$. Continuing in this fashion, we have a finite set of points $\{x_{n,0}, \dots, x_{n,k_n}\}$ on $\gamma_n$ and $\varepsilon$-balls $\{U_{n,0},\dots,U_{n,k_n}\}$ covering $\gamma_n$, with $U_{n,i}$ centered at $x_{n,i}$. But since the lengths of the $\gamma_n$ are uniformly bounded, the sequence $\{k_n\}$ must be bounded as well, say $k_n \leq M$. For any $n$, if $i > k_n$, define $x_{n,i} = x_0$ and $U_{n,i} = U_{n,0}$.
    Notice that by our definition, the segment of $\gamma_n$ between $x_{n,i}$ and $x_{n,i+1}$ is contained in $U_{n,i}$, which is $\cat(1)$. Hence this segment can be shrunken to the (unique) geodesic segment between $x_{n,i}$ and $x_{n,i+1}$. Doing this for each $i$ yields a piecewise geodesic loop $\overline \gamma_n$ which is short homotopic to $\gamma_n$. Note also that the number of points at which the $\gamma_n$ are not geodesic (which we will call the \emph{break points}) is uniformly bounded. By abuse of notation, we replace $\gamma_n$ with this piecewise geodesic $\overline\gamma_n$.

    Now since $\gamma_n$ is piecewise geodesic, it is contained in a finite subcomplex $K_n \subseteq Y$.
    Because the number of break points of the $\gamma_n$ are uniformly bounded and $\mathrm{Shapes}(Y)$ is finite, there is a uniform bound on the number of cells in the $K_n$.  
    Then since $\mathrm{Shapes}(Y)$ is finite, 
    this implies there is a subsequence $\{\gamma_{n_i}\}$ such that $K_{n_i}$ is isometrically isomorphic to $K_{n_j}$ for all $i$ and $j$, via some map $\varphi_{ij} : K_{n_i} \to K_{n_j}$.
    For each $i$ and $j$, let $\gamma_{ij} = \varphi_{ij}(\gamma_{n_i})$, and for $k = 0,\dots,M$, let $x_{ij}^{(k)} = \varphi_{ij}(x_{i,k})$, so that $\{x_{ij}^{(0)},\dots, x_{ij}^{(M)}\}$ forms the set of break points of $\gamma_{ij}$. Then for $n \geq 0$ and $k = 0, \dots, M$, the sequence $\{x_{in}^{(k)}\}_{i=1}^\infty$ has a convergent subsequence (it is contained in the compact space $K_j$). In particular, we can find a common subsequence $\{x_{i_jn}^{(k)}\}_{j=1}^\infty$ which converges for each $k = 0,\dots,M$. Let $x_n^{(k)} = \lim_{j\to\infty} x_{i_jn}^{(k)}$ for each $k = 0,\dots,M$, and let $\hat \gamma_n$ be the concatenation of the (unique) geodesic segments from $x_n^{(k)}$ to $x_n^{(k+1)}$.

    Choose $N$ large enough so that $d_X( x_{i_jN}^{(k)}, x_N^{(k)} ) < \varepsilon/2$ for all $j \geq N$ and all $k = 0,\dots,M$.
    Let $K = K_N$ and consider the sequence $\{\gamma_{i_jN}\}_{j=N}^\infty$.
    Then, for all $j \geq N$ and each $k = 0,\dots,M$, since $d_X( x_{i_{j}N}^{(k)}, x_{i_{j}N}^{(k+1)} ) < \varepsilon/2$ (by our original definition of the piecewise geodesics $\gamma_n$), and $d_X( x_{i_jN}^{(k)}, x_N^{(k)} ) < \varepsilon/2$, it follows that $d_X(x_{i_{j}N}^{(k+1)}, x_N^{(k)}) < \varepsilon$. 
    Note that we also have $d_X(x_N^{(k)}, x_N^{(k+1)}) < \varepsilon/2$. In summary, the points $x_{i_{j}n}^{(k)}, x_{i_{j}N}^{(k+1)}$, and $x_N^{(k+1)}$ are contained in the $\varepsilon$-ball centered at $x_N^{(k)}$. 
    In particular, this ball also contains the geodesics between respective pairs of these points. Since each such ball is $\cat(1)$, the convergence of the break points of the $\gamma_{i_jN}$ implies uniform convergence of the sequence $\{\gamma_{i_jN}\}$ to $\hat \gamma_N$ in $X$ (i.e., convergence in the sup metric). Then since the length function is lower semicontinuous, 
    \[
        I = \liminf \ell(\gamma_{i_jN}) \geq \ell(\lim \gamma_{i_jN}) = \ell(\hat\gamma_N) \geq I,
    \]
    so $\ell(\hat\gamma_N) = I = \inf \ell(\Omega_0)$. 
    But in addition, again since this $\varepsilon$-ball is $\cat(1)$, the homotopy from $x_{i_{j}N}^{(k)}$ to $x_N^{(k)}$ along the geodesic segment between them, and the homotopy from $x_{i_{j}N}^{(k+1)}$ to $x_N^{(k+1)}$ along the geodesic segment between them, together induce a short (free) homotopy from the (unique) geodesic segment between  $x_{i_{j}N}^{(k)}$ and $x_{i_{j}N}^{(k+1)}$ to the (unique) geodesic segment between $x_N^{(k)}$ and $x_N^{(k+1)}$. In other words, for $j \geq N$, $\gamma_{i_jN} \ssim \hat \gamma_N$. This means $\hat\gamma_N \ssim \gamma$.

    If $I = 0$, then we must have that $\hat\gamma_N$ is a constant loop, so $\gamma \ssim 0$. Suppose then that $\gamma$ is not shrinkable, i.e., that $I > 0$. Assume also that $\hat\gamma_N$ is not locally geodesic. Then there are points $p,q$ on $\hat\gamma_N$ so that one of the segments of $\hat\gamma_N$ between $p$ and $q$ is not a geodesic. Replacing this segment with the unique geodesic between $p$ and $q$ would give a loop in $\Omega_0$ which has strictly shorter length than $\hat\gamma_N$, contradicting the fact that $\ell(\hat\gamma_N) = \inf \ell(\Omega_0)$. Thus $\hat\gamma_N$ is locally geodesic.
\end{proof}

\begin{thm}\label{prop:edgereduction}
    If $\gamma$ is a short loop, then there is a short closed edge path $\gamma'$ with $\gamma \sim \gamma'$.
\end{thm}

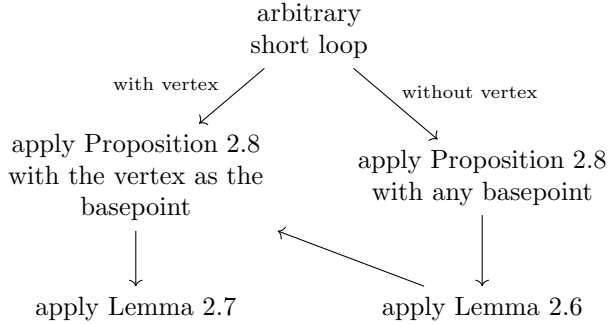
\begin{figure}[!ht]
\centering
\begin{tikzcd}[row sep=0.75cm, column sep=-1cm]
    & \parbox{2cm}{\centering arbitrary short loop}
            \arrow[ld, swap, "\text{with vertex }"]
            \arrow[rd, "\text{without vertex}"] 
            \\
        \parbox{4cm}{\centering apply Proposition \ref{prop:maintool} with the vertex as the basepoint}
            \arrow[d] 
        & & \parbox{4cm}{\centering apply Proposition \ref{prop:maintool} with any basepoint} 
            \arrow[d]
            \\
        \parbox{4cm}{\centering apply Lemma \ref{lem:vertextoedgepath}}
        & & \parbox{4cm}{\centering apply Lemma \ref{lem:novertextovertex}} 
            \arrow[llu]
            \\
\end{tikzcd}
\caption{A flow chart for the proof of Theorem \ref{prop:edgereduction}}
\end{figure}

\begin{proof}
    If $\gamma \ssim 0$, there is nothing to show, so suppose $\gamma$ is not shrinkable.
    Suppose then that $\gamma$ intersects some singular vertex $v$ of $X$. Then we may apply Proposition \ref{prop:maintool} with basepoint $v$ to obtain a short closed local geodesic $\gamma_1$ based at $v$ with $\gamma \ssim \gamma_1$. By Lemma \ref{lem:vertextoedgepath}, there is an edge path $\gamma'$ with $\gamma' \ssim \gamma_1$, and so $\gamma \ssim \gamma'$. 

    Now suppose $\gamma$ does not intersect a singular vertex of $X$. Then apply Proposition \ref{prop:maintool} (with any basepoint) to obtain a short closed local geodesic $\gamma_1$ with $\gamma_1 \ssim \gamma$. Lemma \ref{lem:novertextovertex} gives a short loop $\gamma_2$ with $\gamma_1 \ssim \gamma_2$ such that $\gamma_2$ intersects some singular vertex $v \in D(B_3)$. Applying Proposition \ref{prop:maintool} again, this time based at $v$, gives a short closed geodesic $\gamma_3$ based at $v$ with $\gamma_3 \ssim \gamma_2$. We can then apply Lemma \ref{lem:vertextoedgepath} to get an edge path $\gamma'$ with $\gamma' \ssim \gamma_3$, and in particular, $\gamma' \ssim \gamma$. 
\end{proof}

\subsection{Further edge path reduction}

Next, we work to reduce arbitrary short closed edge paths to the list in Figure \ref{fig:checking edge loops}.

Let $s_i$ denote the edge of $\Delta$ opposite $\hat s_i$ for $i=1,2,3$.  
Then let $\alpha$ be the length of $s_1$, $\beta$ the length of $s_2$, and $\delta$ the length of $s_3$. 
Then 
\begin{align*}
    \alpha &= \arccos\left( \frac{\sqrt{2}}{\sqrt{3}} \right) \approx 0.615 \\
    \beta  &= \arccos\left( \frac{1}{\sqrt{3}} \right) \approx 0.955 \\
    \delta &= \arccos\left( \frac{1}{\sqrt{2}} \right) \approx 0.785.
\end{align*}
The maps $m_\sigma$ give rise an edge labeling of $X$. We will call an edge $e$ \emph{type $s_i$} whenever $m_\sigma(e)$ is the edge $s_i$ in $\Delta$.

We can rule out certain combinations of edges as possible loops immediately.
For example, since $X$ is a flag complex, any loop of edge length 3 must be filled by a simplex and thus is shrinkable.
Similarly, $X$ is simplicial, so it contains no digons and every edge is embedded. Hence we may restrict ourselves to loops of edge length $\geq 4$.
We may also restrict our consideration to embedded edge paths; if a path is not embedded, we may reparameterize it to be a union of embedded closed edge paths which we may shrink separately. Thus in the following, all closed edge paths are assumed to be embedded unless stated otherwise.

For the following further reduction, for $i = 1$ or $3$, say that a segment in the 1-skeleton of $X$ is ``type $2s_i$'' if it is the union of a pair of (distinct) edges of type $s_i$ which meet at a vertex of type $\hat s_2$. 

\begin{lem} \label{lem: first edge restriction}
    Any closed edge path of $X$ can be shrunken to a closed edge path consisting only of (distinct) segments of type $2s_i$ for $i = 1,3$, and an even number of edges of type $s_2$. 
\end{lem}

\begin{proof}
    Let $\gamma$ be an edge path. Suppose there are adjacent edges $e_1$ and $e_3$ in $\gamma$ of type $s_1$ and $s_3$, respectively. These edges meet at a vertex $v$ in $\gamma$ of type $\hat s_2$, and give rise to a vertex of type $\hat s_3$ and $\hat s_1$, resp., in the link of $v$. Since the link of a type $\hat s_2$ vertex is complete bipartite, there is an edge in the link connecting these vertices, which corresponds to a 2-simplex containing $e_1$ and $e_3$. Let $e_2$ be the other edge of this 2-simplex, necessarily of type $s_2$. Then $\gamma$ can be shrunken by replacing the concatenation of $e_1$ and $e_3$ with $e_2$. In other words, we may assume that for any closed edge path, the two edges adjacent to an edge of type $s_1$ must be type $s_1$ and type $s_2$, and similarly the two edges adjacent to an edge of type $s_3$ must be type $s_3$ and type $s_2$. Since there are no digons, this implies such a path contains only edges of type $2s_1$, $s_2$, and $2s_3$. 

    Note that the endpoints of a segment of type $2s_1$ are vertices of type $\hat s_3$, and the endpoints of a segment of type $2s_3$ are vertices of type $\hat s_1$. But the segments of type $s_2$ have exactly one endpoint which is a vertex of type $\hat s_1$ and exactly one endpoint which is a vertex of type $\hat s_3$. Since $\hat s_i$ vertices cannot be adjacent to $\hat s_i$ vertices (for $i = 1,2,3$), this means each possible edge loop must have an even number of edges of type $s_2$. 
\end{proof}

In particular, this means every closed edge path is shrinkable to a closed edge path with an even number of edges, so we may restrict our attention to such paths. With these restrictions, we list the remaining possible edge paths. 

Suppose $\gamma$ has $2n_\alpha$ edges of type $s_1$, $2n_\beta$ edges of type $s_2$, and $2n_\delta$ edges of type $s_3$. Then $\ell(\gamma) < 2\pi$ if and only if $2n_\alpha \alpha + 2n_\beta \beta + 2n_\delta \delta < 2\pi$.  
The triples $(n_\alpha,n_\beta,n_\delta)$ of non-negative integers which provide solutions to the latter inequality are shown in Table \ref{tab:first list}.
\begin{table}[!ht]
    \centering
    \begin{align*}
    (i,0,0), &\ i = 2,3,4,5 &
    (0,i,0), &\ i = 2,3 &
    (0,0,i), &\ i = 2,3 \\
    (0,1,i), &\ i = 1,2 &
    (0,2,i), &\ i = 1  &
    (1,0,i), &\ i = 1,2,3 \\
    (1,1,i), &\ i = 0,1 &
    (1,2,0), & &
    (2,0,i), &\ i = 1,2 \\
    (2,1,i), &\ i = 0,1 &
    (3,0,1), & &
    (3,1,0)
\end{align*}
    \caption{$(n_\alpha,n_\beta,n_\delta)$ satisfying $2n_\alpha \alpha + 2n_\beta \beta + 2n_\delta \delta < 2\pi$}
    \label{tab:first list}
\end{table}

A segment of type $2s_1$ cannot be adjacent to a segment of type $2s_3$ (and vice versa), so the triples $(i,0,j)$ for any $i > 0$ and $j > 0$ cannot correspond to an embedded loop. 
By Lemma \ref{lem:intersectionofstars}, there are no embedded loops consisting of exactly two segments of type $2s_1$ or two segments of type $2s_3$. 
This shortens the list to the triples in Table \ref{tab:second shortened list}.
\begin{table}[!ht]
    \centering
\begin{align*}
    (i,0,0), &\ i = 3,4,5 &
    (0,i,0), &\ i = 2,3 &
    (0,0,3), \\
    (0,1,i), &\ i = 1,2  &
    (0,2,i), &\ i = 1   &
    (1,1,i), &\ i = 0,1 \\
    (1,2,0), & &
    (2,1,i), &\ i = 0,1 &
    (3,1,0)
\end{align*}
    \caption{A reduced list of triples $(n_\alpha,n_\beta,n_\delta)$}
    \label{tab:second shortened list}
\end{table}

We can now show how certain cycles may be shrunken to those in Figure \ref{fig:checking edge loops}. First, the 4-cycles.
The only short embedded 4-cycles which do not appear in Figure \ref{fig:checking edge loops} are those appearing in Figure \ref{fig:remaining 4-cycles}.

\begin{figure}[!ht]
    \begin{tikzpicture}[scale=0.7]
        \coordinate (s3 top)    at (0,2) ;
        \coordinate (s1 left)   at (-2,0);
        \coordinate (s2 middle) at ( 0,0) ;
        \coordinate (s1 right)  at ( 2,0) ;
        \coordinate (label)     at (0,-1.25) ;
        
        \draw (s3 top) -- (s1 left)  -- (s2 middle) -- 
            (s1 right) -- (s3 top);

        \filldraw (s3 top)    circle (0.05cm);
        \filldraw (s1 left)   circle (0.05cm);
        \filldraw (s2 middle) circle (0.05cm);
        \filldraw (s1 right)  circle (0.05cm);

        \node[label={[label distance=0mm]above: {$\hat s_3$}}] 
            at (s3 top)  {} ;
        \node[left] at (s1 left) {$\hat s_1$} ;
        \node[below] at (s2 middle) {$\hat s_2$} ;
        \node[label={[label distance=0mm]right: {$\hat s_1$}}] 
            at (s1 right)  {} ;

        \node at (label) {(a)};          
    \end{tikzpicture}
    \begin{tikzpicture}[scale=0.7]
        \coordinate (s3 top)    at (0,2) ;
        \coordinate (s1 left)   at (-2,0);
        \coordinate (s2 middle) at ( 0,0) ;
        \coordinate (s1 right)  at ( 2,0) ;
        \coordinate (label)     at (0,-1.25) ;
        
        \draw (s3 top) -- (s1 left)  -- (s2 middle) -- 
            (s1 right) -- (s3 top);

        \filldraw (s3 top)    circle (0.05cm);
        \filldraw (s1 left)   circle (0.05cm);
        \filldraw (s2 middle) circle (0.05cm);
        \filldraw (s1 right)  circle (0.05cm);

        \node[label={[label distance=0mm]above: {$\hat s_1$}}] 
            at (s3 top)  {} ;
        \node[left] at (s1 left) {$\hat s_3$} ;
        \node[below] at (s2 middle) {$\hat s_2$} ;
        \node[label={[label distance=0mm]right: {$\hat s_3$}}] 
            at (s1 right)  {} ;

        \node at (label) {(b)};          
    \end{tikzpicture}
    
    \caption{Remaining 4-cycles}
    \label{fig:remaining 4-cycles}
\end{figure}
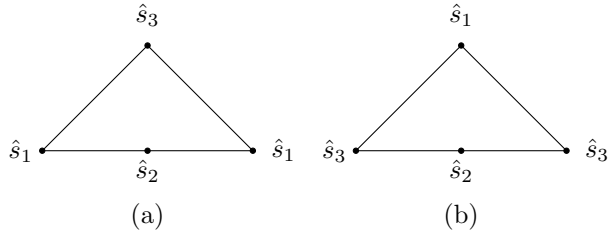

\begin{figure}[!ht]
    \begin{tikzpicture}[scale=0.7]
        \coordinate (s3 top)    at (0,2) ;
        \coordinate (s1 left)   at (-2,0);
        \coordinate (s2 middle) at ( 0,0) ;
        \coordinate (s1 right)  at ( 2,0) ;
        \coordinate (label)     at (0,-1.25) ;
        
        \draw[fill=gray!25] (s3 top) -- (s1 left)  -- (s2 middle) -- 
            (s1 right) -- (s3 top);
        \draw (s2 middle) -- (s3 top);

        \filldraw (s3 top)    circle (0.05cm);
        \filldraw (s1 left)   circle (0.05cm);
        \filldraw (s2 middle) circle (0.05cm);
        \filldraw (s1 right)  circle (0.05cm);

        \node[label={[label distance=0mm]above: {$\hat s_3$}}] 
            at (s3 top)  {} ;
        \node[left] at (s1 left) {$\hat s_1$} ;
        \node[below] at (s2 middle) {$\hat s_2$} ;
        \node[label={[label distance=0mm]right: {$\hat s_1$}}] 
            at (s1 right)  {} ;

        \node at (label) {(a)};          
    \end{tikzpicture}
    \begin{tikzpicture}[scale=0.7]
        \coordinate (s3 top)    at (0,2) ;
        \coordinate (s1 left)   at (-2,0);
        \coordinate (s2 middle) at ( 0,0) ;
        \coordinate (s1 right)  at ( 2,0) ;
        \coordinate (label)     at (0,-1.25) ;
        
        \draw[fill=gray!25] (s3 top) -- (s1 left)  -- (s2 middle) -- 
            (s1 right) -- (s3 top);
        \draw (s2 middle) -- (s3 top);

        \filldraw (s3 top)    circle (0.05cm);
        \filldraw (s1 left)   circle (0.05cm);
        \filldraw (s2 middle) circle (0.05cm);
        \filldraw (s1 right)  circle (0.05cm);

        \node[label={[label distance=0mm]above: {$\hat s_1$}}] 
            at (s3 top)  {} ;
        \node[left] at (s1 left) {$\hat s_3$} ;
        \node[below] at (s2 middle) {$\hat s_2$} ;
        \node[label={[label distance=0mm]right: {$\hat s_3$}}] 
            at (s1 right)  {} ;

        \node at (label) {(b)};          
    \end{tikzpicture}
    
    \caption{Filling the 4-cycles}
    \label{fig:filled 4-cycle path}
\end{figure}
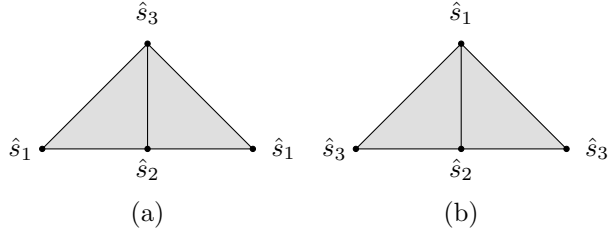

\begin{lem} \label{lemma:filling 4 cycles}
    Loops of the type in Figure \ref{fig:remaining 4-cycles} are contained in subcomplexes of the form in Figure \ref{fig:filled 4-cycle path}.
\end{lem}

\begin{proof}
    Suppose $\gamma$ is a(n embedded) loop of type (a). (The type (b) case is identical, switching the labeling of $\hat s_1$ and $\hat s_3$.) Let $v$ denote the vertex of type $\hat s_2$, $u$ the vertex of type $\hat s_2$, and $w_1,w_2$ the distinct vertices of type $\hat s_1$. 
    
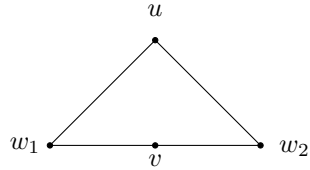
\[
    \begin{tikzpicture}[scale=0.7]
        \coordinate (s3 top)    at (0,2) ;
        \coordinate (s1 left)   at (-2,0);
        \coordinate (s2 middle) at ( 0,0) ;
        \coordinate (s1 right)  at ( 2,0) ;
        
        \draw (s3 top) -- (s1 left)  -- (s2 middle) -- 
            (s1 right) -- (s3 top);

        \filldraw (s3 top)    circle (0.05cm);
        \filldraw (s1 left)   circle (0.05cm);
        \filldraw (s2 middle) circle (0.05cm);
        \filldraw (s1 right)  circle (0.05cm);

        \node[label={[label distance=0mm]above: {$u$}}] 
            at (s3 top)  {} ;
        \node[left] at (s1 left) {$w_1$} ;
        \node[below] at (s2 middle) {$v$} ;
        \node[label={[label distance=0mm]right: {$w_2$}}] 
            at (s1 right)  {} ;        
    \end{tikzpicture}
    \]
    Since $lk(v)$ is complete bipartite, we can find a vertex $z$ of type $\hat s_3$ adjacent to $v$, $w_1$, and $w_2$:
    
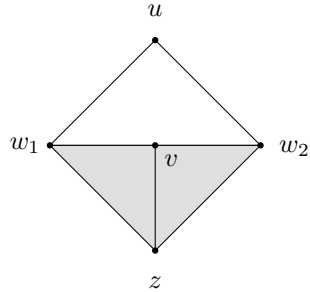
\[
    \begin{tikzpicture}[scale=0.7]
        \coordinate (s3 top)    at (0,2) ;
        \coordinate (s1 left)   at (-2,0);
        \coordinate (s2 middle) at ( 0,0) ;
        \coordinate (s1 right)  at ( 2,0) ;
        \coordinate (s3 bottom) at (0,-2);
        
        \draw (s3 top) -- (s1 left)  -- (s2 middle) -- 
            (s1 right) -- (s3 top);
        \draw[fill=gray!25] (s3 bottom) -- (s1 left)  -- (s2 middle) -- 
            (s1 right) -- (s3 bottom);
        \draw (s2 middle) -- (s3 bottom);

        \filldraw (s3 top)    circle (0.05cm);
        \filldraw (s1 left)   circle (0.05cm);
        \filldraw (s2 middle) circle (0.05cm);
        \filldraw (s1 right)  circle (0.05cm);
        \filldraw (s3 bottom)  circle (0.05cm);

        \node[label={[label distance=0mm]above: {$u$}}] 
            at (s3 top)  {} ;
        \node[left] at (s1 left) {$w_1$} ;
        \node[below right] at (s2 middle) {$v$} ;
        \node[label={[label distance=0mm]right: {$w_2$}}] 
            at (s1 right)  {} ;        
        \node[label={[label distance=0mm]below: {$z$}}] 
            at (s3 bottom)  {} ;   
    \end{tikzpicture}
    \]
    Now the induced cycle on the vertices $u$, $w_1$, $z$, and $w_2$ is an edge path of the type in Figure \ref{fig:checking edge loops}(a), thus can be filled by a vertex $v'$ of type $\hat s_2$. If $v \not= v'$, this would contradict Lemma \ref{lem:intersectionofstars}. Thus $\gamma$ can be completed to 
    
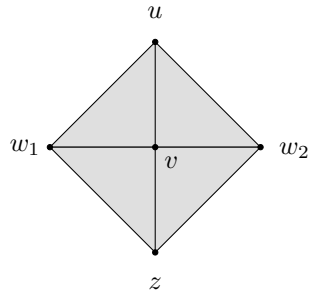
\[
    \begin{tikzpicture}[scale=0.7]
        \coordinate (s3 top)    at (0,2) ;
        \coordinate (s1 left)   at (-2,0);
        \coordinate (s2 middle) at ( 0,0) ;
        \coordinate (s1 right)  at ( 2,0) ;
        \coordinate (s3 bottom) at (0,-2);
        
        \draw[fill=gray!25] (s3 top) -- (s1 left)  -- (s2 middle) -- 
            (s1 right) -- (s3 top);
        \draw[fill=gray!25] (s3 bottom) -- (s1 left)  -- (s2 middle) -- 
            (s1 right) -- (s3 bottom);
        \draw (s3 top) -- (s3 bottom);

        \filldraw (s3 top)    circle (0.05cm);
        \filldraw (s1 left)   circle (0.05cm);
        \filldraw (s2 middle) circle (0.05cm);
        \filldraw (s1 right)  circle (0.05cm);
        \filldraw (s3 bottom)  circle (0.05cm);

        \node[label={[label distance=0mm]above: {$u$}}] 
            at (s3 top)  {} ;
        \node[left] at (s1 left) {$w_1$} ;
        \node[below right] at (s2 middle) {$v$} ;
        \node[label={[label distance=0mm]right: {$w_2$}}] 
            at (s1 right)  {} ;        
        \node[label={[label distance=0mm]below: {$z$}}] 
            at (s3 bottom)  {} ;   
    \end{tikzpicture} \qedhere
    \]
\end{proof}

\begin{lem} \label{lem:all 6 cycles are short hom}
    Every short embedded closed edge path in $D(B_3)$ of edge-length $6$ is short-homotopic to an embedded closed edge path of the type in Figure \ref{fig:checking edge loops}(b).
\end{lem}
\begin{proof}
    Let $\delta$ denote a segment which is the union of two adjacent edges of $X$ meeting at a vertex $v$ of type $\hat s_2$, such that the remaining two vertices $w_1$ and $w_2$ are both type $\hat s_1$ or both type $\hat s_3$. Then there is a vertex $v'$ in the link of $v$ which is adjacent to both $w_1$ and $w_2$ and has the opposite type (that being type $\hat s_3$ or $\hat s_1$, respectively). Let $\delta'$ denote the edge path between the vertices $w_1$, $v'$, and $w_2$.
    This results in $\delta$ and $\delta'$ bounding a subcomplex of $X$ of the form in Figure 
    \ref{fig:filled 4-cycle path}(b) or Figure 
    \ref{fig:filled 4-cycle path}(c).
    In either case, it is clear that $\delta$ is homotopic to $\delta'$ through a homotopy which is monotone in the lengths of its paths. In other words, we may replace an ``$\hat s_1$-$\hat s_2$-$\hat s_1$'' (or, an ``$\hat s_3$-$\hat s_2$-$\hat s_3$'') segment with an ``$\hat s_1$-$\hat s_3$-$\hat s_1$'' (resp., an ``$\hat s_3$-$\hat s_1$-$\hat s_3$'') segment in a ``monotone'' fashion.
    
    As a consequence of this, if $\gamma$ is a short edge loop containing a subpath $\delta$ as above, then $\gamma$ is short homotopic to the path $\gamma'$ obtained by replacing $\delta$ with $\delta'$ in $\gamma$. Any 6-cycle $\gamma$ in Table \ref{tab:second shortened list} which is not the 6-cycle in Figure \ref{fig:checking edge loops} contains a subpath $\delta$ of this form. After repeating this procedure on each such subpath, we see that $\gamma$ is short homotopic to a 6-cycle of the form in Figure \ref{fig:checking edge loops}.
\end{proof}

The only edge paths of length 8  that don't appear in Figure \ref{fig:checking edge loops} are those of the type seen in Figure \ref{fig:short 8 cycles}.

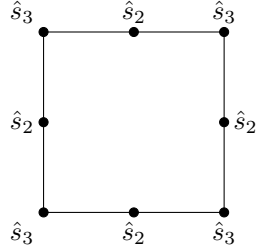
\begin{figure}[!ht]
    \centering
   \begin{tikzpicture}[scale=1.2]
        \coordinate (s31) at (0,0) ;
        \coordinate (s21) at (1,0) ;
        \coordinate (s32) at (2,0);
        \coordinate (s1) at (2,-1) ;
        \coordinate (s33) at (2,-2) ;
        \coordinate (s22) at (1,-2) ;
        \coordinate (s34) at (0,-2) ;
        \coordinate (s23) at (0,-1) ;
        
        \filldraw (s31) circle (0.05cm);
        \filldraw (s21) circle (0.05cm);
        \filldraw (s32) circle (0.05cm);
        \filldraw (s1) circle (0.05cm);
        \filldraw (s33) circle (0.05cm);
        \filldraw (s22) circle (0.05cm);
        \filldraw (s34) circle (0.05cm);
        \filldraw (s23) circle (0.05cm);

        \draw (s31) -- (s21)  -- (s32) -- (s1) --(s33) -- (s22) -- (s34)  -- (s23) -- (s31);

        \node[above left] at (s31) {$\hat s_3$} ;
        \node[above] at (s21) {$\hat s_2$} ;
        \node[above] at (s32) {$\hat s_3$} ;
        \node[right] at (s1)  {$\hat s_2$} ;
        \node[below] at (s33) {$\hat s_3$} ;
        \node[below] at (s22) {$\hat s_2$} ;
        \node[below left] at (s34) {$\hat s_3$} ;
        \node[left] at (s23) {$\hat s_2$} ;
    \end{tikzpicture}
    \caption{The other 8-cycle}
    \label{fig:short 8 cycles}
\end{figure}

The proof of the following Lemma is identical to the proof of Lemma \ref{lem:all 6 cycles are short hom}.
\begin{lem} \label{lem:8-cycle shrink}
    An embedded 8-cycle of the form in Figure \ref{fig:short 8 cycles} is short-homotopic to an embedded 8-cycle of the form in Figure \ref{fig:checking edge loops}.
\end{lem}

Finally, the only short 10-cycle is the one appearing in Figure \ref{fig:bad 10 cycle}. However, we have assumed that such a cycle is never embedded. Thus we can complete the proof of Theorem \ref{thm:B3 CAT1 Criteria}.

\begin{proof}[Proof (of Theorem \ref{thm:B3 CAT1 Criteria})]
    Let $\gamma_1$ be an arbitrary embedded short loop. Using Theorem \ref{prop:edgereduction}, we can find an embedded short edge loop $\gamma_2$ such that $\gamma_1 \ssim \gamma_2$. Then using Lemmas \ref{lemma:filling 4 cycles}, \ref{lem:all 6 cycles are short hom}, and \ref{lem:8-cycle shrink}, we can find an embedded short loop $\gamma_3$ of the form in Figure \ref{fig:checking edge loops} with $\gamma_2 \ssim \gamma_3$. Condition (5) guarantees that $\gamma_3 \ssim 0$, therefore $\gamma_1 \ssim 0$. 
\end{proof}

\section{Background on Artin groups and Coxeter groups}
\label{sec:Definitions}

We can now focus on the main application of this result: the Deligne complex of an Artin group of type $B_3$. In this section, we introduce the notation that will be used throughout the rest of the paper.
 
A \emph{Coxeter-Dynkin diagram} is a finite simplicial graph $\Gamma$ equipped with a labeling $m : E(\Gamma) \to \mathbb{Z}_{\geq 3} \cup \{\infty\}$ of edges. Since $\Gamma$ is simplicial, every edge $e \in E(\Gamma)$ can be written as an unordered pair of vertices $\{s,t\}$, and sometimes we write $m(s,t) = m(t,s) = m(\{s,t\}) = m(e)$. If $\{s,t\} \not\in E(\Gamma)$, define $m(s,t) = m(t,s) = 2$. When illustrating a Coxeter-Dynkin diagram, it is common to omit labels on edges $e$ with $m(e) = 3$.

Such a diagram defines a Coxeter group
\[
    W(\Gamma) = \langle\, s \in V(\Gamma) \mid (st)^{m(s,t)} = 1 \text{ if } m(s,t) \not= \infty, s^2 = 1 \,\rangle
\]
and an Artin group
\[
    A(\Gamma) = \langle\, s \in V(\Gamma) \mid \underbrace{sts\ldots}_{m(s,t) \text{ letters}} = \underbrace{tst\ldots}_{m(s,t) \text{ letters}} \text{ if } m(s,t) \not= \infty \,\rangle.
\]
We call a Coxeter-Dynkin diagram $\Gamma$ \emph{spherical} and an Artin group $A(\Gamma)$ \emph{spherical-type} if $W(\Gamma)$ is finite.

We will write $\Lambda \leq \Gamma$ to denote a full subgraph $\Lambda$ of $\Gamma$ (i.e., if two vertices $x,y \in V(\Lambda)$ are adjacent in $\Gamma$, then they are also adjacent in $\Lambda$) which inherits the edge labeling of $\Gamma$. Then $A(\Lambda)$ is an Artin group which is isomorphic to the subgroup of $\Gamma$ generated by $V(\Lambda)$ by sending the generators of $A(\Lambda)$ to the generators coming from the same vertex viewed inside $\Gamma$ \cite{lek1983homotopy}. By abuse of notation, we will identify these two groups. A similar statement holds for the Coxeter groups $W(\Lambda)$ and $W(\Gamma)$.
A particular type of full subgraphs are the \emph{maximal} full subgraphs. If $s \in V(\Gamma)$, then we let $\hat s$ denote the full subgraph of $\Gamma$ on vertex set $V(\Gamma) \setminus s$. By abusing notation, we sometimes let $\hat s = A(\hat s)$ or $\hat s = W(\hat s)$; the meaning will be obvious by context.

For any Coxeter-Dynkin diagram $\Gamma$, we may associate a number of relevant cell complexes, two of which are the \emph{Deligne complex} and the \emph{Davis-Moussong complex}. Let $G = A(\Gamma)$ or $W(\Gamma)$. 
We define 
\begin{align*}
\mathcal S^f &= \{\,\Lambda \leq \Gamma : \Lambda \text{ is spherical} \,\}, and \\
G\mathcal S^f &= \{\, g A(\Lambda) : g \in G, \Lambda \in \mathcal S^f \,\},
\end{align*}
both partially ordered by standard inclusion. Since $G\mathcal S^f$ is a poset, we can form its derived poset $(G\mathcal S^f)'$, the set of linearly ordered subsets (or \emph{chains}) of $G\mathcal S^f$, itself ordered by inclusion. Then $(G\mathcal S^f)'$ is an abstract simplicial complex, so we may form its geometric realization $|(G\mathcal S^f)'|$. When $G = W(\Gamma)$, we call this the \emph{Davis-Moussong complex} $\Sigma(\Gamma)$, and when $G = A(\Gamma)$, we call this the \emph{Deligne complex} $\Phi(\Gamma)$. (Note that in some literature, $\Phi(\Gamma)$ is called the \emph{modified Deligne complex}.) In either case, $G$ acts on $|(G\mathcal S^f)'|$ with strict fundamental domain isomorphic to $|(\mathcal S^f)'|$. A discussion of various natural metrics on $\Phi(\Gamma)$ and $\Sigma(\Gamma)$ are discussed in \cite{cd1995k}. We will be interested in the \emph{Moussong metric}, which we will not define explicitly here for reasons that will become evident. (The full definition can be found in \cite[\S 4.4]{cd1995k}.)

The next of the complexes are the \emph{Artin complex} and the \emph{Coxeter complex}. Again, let $G$ denote either $A(\Gamma)$ or $W(\Gamma)$, and define a simplicial complex $\Delta = \Delta(\Gamma)$ as follows. The vertex set of $\Delta$ is 
\[
\{\, g\hat s : g \in G, s \text{ a vertex of } \Gamma \,\},
\]
which is the collection of cosets of maximal standard parabolic subgroups of $G$. Note that every vertex $v$ has a well-defined ``type'', that is, a \emph{unique} vertex $s \in V(\Gamma)$ such that $v = g \hat s$; when this equality holds, we will say that $v$ is \emph{type $\hat s$}. 
A collection $\{g_0\hat s_1, \dots, g_n \hat s_n\}$ of vertices span an $n$-simplex of $\Delta$ if and only if 
\[
\bigcap g_i \hat s_i \not= \varnothing.
\]
When $G = A(\Gamma)$, we call $\Delta$ the \emph{Artin complex} and denote it by $D(\Gamma)$. (Note that in some literature, $D(\Gamma)$ is called the \emph{spherical Deligne complex}, or simply the \emph{Deligne complex}. However, we reserve the latter name for our $\Phi(\Gamma)$.) When $G = W(\Gamma)$, we call $\Delta$ the \emph{Coxeter complex} and denote it by $C(\Gamma)$.

The group $G$ (for $G = A(\Gamma)$ or $W(\Gamma)$) acts naturally on $\Delta$ by left multiplication.
This action has a strict fundamental domain of a $(|V(\Gamma)|-1)$-simplex with exactly one vertex of each type.
If $\Gamma$ is spherical, then $C(\Gamma)$ has a metric which makes it isometric to a standard round sphere in $\mathbb{R}^{|V(\Gamma)|-1}$ (hence the terminology ``spherical'' Coxeter-Dynkin diagram). Since the fundamental domains of the actions on $D(\Gamma)$ and $C(\Gamma)$ are isomorphic, when $\Gamma$ is spherical we may endow $D(\Gamma)$ with a metric under which this isomorphism of fundamental domains becomes an isometry. By slight abuse of notation, we will call this metric on $D(\Gamma)$ the \emph{(spherical) Moussong metric}. The key relationship between these various notions of Moussong metric is the following.

\begin{prop} \label{prop:Moussong reduction} \textup{\cite{cd1995k}}
    The Moussong metric on $\Phi(\Gamma)$ is $\cat(0)$ if and only if for every connected $\Lambda \in \mathcal S^f_\Gamma$, the spherical Moussong metric on $D(\Lambda)$ is $\cat(1)$.
\end{prop}

If $\Gamma$ is a diagram with the property that every $\Lambda \in \mathcal S^f_\Gamma$ has $|V(\Lambda)| \leq n$, then we call each of $\Gamma$, $W_\Gamma$, and $A_\Gamma$  \emph{$n$-dimensional}. (In this case, the Deligne complex and Davis-Moussong complex are $n$-dimensional simplicial complexes.)

\section{Geometry of \texorpdfstring{$D(B_n)$}{D(B\_n)}}
\label{sec:geometry}

To fix notation, we let $A_n$ denote the Coxeter-Dynkin diagram 
\[
\begin{tikzpicture}
    \coordinate (t1) at (0,0);
    \coordinate (t2) at (1,0);
    \coordinate (t3) at (2,0);
    \coordinate (t4) at (3.5,0);
    
    \filldraw (t1) circle (0.05cm);
    \draw (t1) -- (2.3,0);
    \filldraw (t2) circle (0.05cm);
    \filldraw (t3) circle (0.05cm);
    \node at (2.75,0) {. . .};
    \filldraw(t4) circle (0.05cm);
    \draw (3.25,0) -- (t4);

    \node[below] at (t1) {$t_1$};
    \node[below] at (t2) {$t_2$};
    \node[below] at (t3) {$t_3$};
    \node[below] at (t4) {$t_n$};
\end{tikzpicture}
\]
and let $B_n$ denote the Coxeter-Dynkin diagram
\[
\begin{tikzpicture}
    \coordinate (s1) at (0,0);
    \coordinate (s2) at (1,0);
    \coordinate (s3) at (2.5,0);
    \coordinate (s4) at (3.5,0);
    
    \filldraw (s1) circle (0.05cm);
    \filldraw (s2) circle (0.05cm);
    \filldraw (s3) circle (0.05cm);
    \filldraw (s4) circle (0.05cm);
    \draw (s1) -- (s2);
    \draw (s2) -- (1.3,0);
    \draw (2.25,0) -- (s3);
    \draw (s3) -- node[above] {$4$} (s4);

    \node at (1.75,0) {. . .};
    
    \node[below] at (s1) {$s_1$};
    \node[below] at (s2) {$s_2$};
    \node[below] at (s3) {$s_{n-1}$};
    \node[below] at (s4) {$s_n$};
\end{tikzpicture}
\]
We define a partial order on the vertices of $D(A_n)$ and $D(B_n)$ as follows. Let $u = s$ or $t$. For cosets $g_1 \hat u_{i_1}$ and $g_2 \hat u_{i_2}$, we say $g_1 \hat u_{i_1} \leq g_2 \hat u_{i_2}$ if $g_1 \hat u_{i_1} \cap g_2 \hat u_{i_2} \not= \varnothing$ and $i_1 \leq i_2$. 
To either vertex set we may formally adjoin a minimal element $0$ satisfying $0 < g \hat u_i$ for all $g$ and all $i$; we will denote the resulting poset by $D_0(A_n)$ and $D_0(B_n)$, respectively. One of the key technical lemmas we need is the following.

\begin{prop} \emph{\cite[Prop.~6.6]{haettel2023lattices}} \label{6.6}
    The poset $D_0(B_n)$ is a graded semi-lattice of rank $n$ with minimum $0$, such that any pairwise upper bounded subset of $D_0(B_n)$ has a join.
\end{prop}

Within the proof of this Proposition, three important maps are defined, the first of which is the morphism $\phi : A(B_n) \to A(A_{2n-1})$ given by
\begin{align*}
    s_i &\mapsto t_it_{2n-i} \text{ for } i \leq n-1, \\
    s_n &\mapsto t_n.
\end{align*}

The second important map defined within the proof of Proposition \ref{6.6} is the extension of $\phi$ to the Deligne complexes $\psi : D(B_n) \to D(A_{2n-1})$ by defining $\psi(g\hat s_i) = \phi(g)\hat t_i$ on the vertices. (This also induces a map $D_0(B_n) \to D_0(B_n)$ by sending $0$ to $0$.) This map is well-defined and preserves the ordering on the respective Deligne complexes.

The final map that we need that is defined within the proof of Proposition \ref{6.6} is the involution $\sigma : A(A_{2n-1}) \to A(A_{2n-1})$ by $t_i \mapsto t_{2n-i}$. 
Then $\sigma$ extends to an order-reversing involution on $D(A_{2n-1})$ by $\sigma(g \hat t_i) = \sigma(g)\hat t_{2n-i}$. 
Notice that for $g \in A(B_n)$,  $\sigma(\psi(g \hat s_i)) = \phi(g) \hat t_{2n-i}$.

Haettel notes that $\phi$ is injective.  Previous proofs of this use normal forms and Garside theory to initially show that $\phi$ restricted to positive monoids is injective.  (See, for example, \cite{crisp1997injective}.)  We provide a novel proof using mapping class groups.  The authors thank Dan Margalit for the idea and Johanna Mangahas for helping with the technical details.  

\begin{prop}\label{gggg2}
 The map $\phi: A(B_n)\to A(A_{2n-1})$ is a an injective homomorphism.  
\end{prop}

We break up this proof into the following parts:

    \begin{itemize}
        \item[1.] \begin{adjustwidth}{0pt}{-6pt} Define convenient homeomorphisms representing the generators of Mod$(\mathbb{D}_{1,n})$ and Mod$(\mathbb{D}_{1,2n})$, \end{adjustwidth}
        \item[2.]  Prove that $\Phi: \text{Mod}(\mathbb{D}_{1,n}) \to \text{Mod}(\mathbb{D}_{2n})$ is a homomorphism,
        \item[3.] Prove that $\Phi: \text{Mod}(\mathbb{D}_{1,n}) \to \text{Mod}(\mathbb{D}_{2n})$ is injective, 
        \item[4.] Define the \textit{mixed braid group} and recall some standard facts about the relationship between Artin groups, braid groups, and mapping class groups,
        \item[5.] Define a series of maps that connect $\phi: A(B_n)\to A(A_{2n-1})$ to\\ 
            $\Phi: \text{Mod}(\mathbb{D}_{1,n}) \to \text{Mod}(\mathbb{D}_{2n})$, and 
        \item[6.] Complete the proof that $\phi: A(B_n)\to A(A_{2n-1})$ is injective. 
    \end{itemize}

\vspace*{5mm}
\noindent \underline{Part 1}: \\

We begin by defining mapping classes that will be used to prove that $\Phi: \text{Mod}(\mathbb{D}_{1,n}) \to \text{Mod}(\mathbb{D}_{2n})$ is a homomorphism.  Define $\mathbb{D}_{2n,1}$ to be the closed disk with $2n$ punctures and one marked point in the middle as shown in the figure for part 1.a.  Define $\mathbb{D}_{1,n}$ to be the closed disk with $n$ punctures and one marked point at the end as shown in the figure for part 2.a.  We will define specific mapping classes in $\mathbb{D}_{1,2n}$ and $\mathbb{D}_{1,n}$ and a quotient map that takes the defined mapping classes in $\mathbb{D}_{1,2n}$ to those in $\mathbb{D}_{1,n}$.

    \begin{itemize}
        \item[1.]  For $1\leq i\leq n-1$ \\

            \begin{itemize}
                \item[a.] Define $\tilde f_i: \mathbb{D}_{1,2n} \to \mathbb{D}_{1,2n}$ to 
                    be the two half Dehn twists shown below where $\times$ denotes the marked point and the punctures are numbered.

\vspace*{5mm}
{
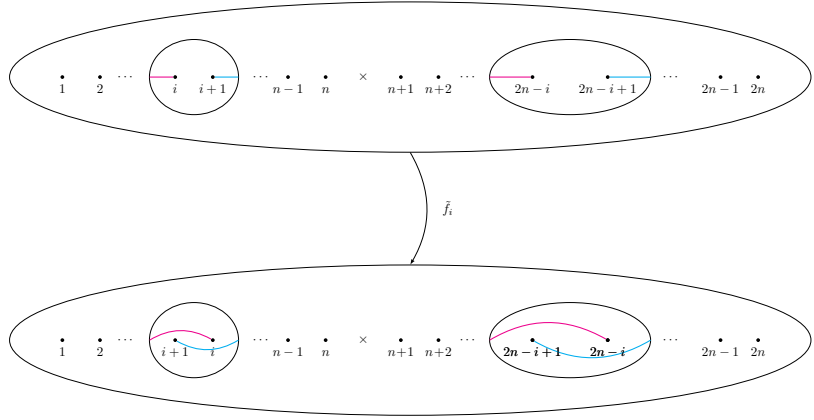
\begin{adjustwidth*}{}{-25mm}
\scalebox{0.5}
{
\begin{tikzpicture}[ele/.style={fill=black,circle,minimum width=0.8pt,inner sep=1pt},every fit/.style={ellipse,draw,inner sep=-2pt}]
    \node[ele,label=below:$1$] (a1) at (-8,5) {}; 
    \node[ele,label=below:$2$] (a2) at (-7,5) {}; 
    \node at (-6.3,5) {$\cdots$};
    \node[ele,label=below:$i$] (ai) at (-5,5) {}; 
    \node[ele,label=below:$i+1$] (ai+1) at (-4,5) {}; 
    \node at (-2.7,5) {$\cdots$};
    \node[ele,label=below:$n-1$] (an-1) at (-2,5) {}; 
    \node[ele,label=below:$n$] (an) at (-1,5) {}; 

  \node[] (x1) at (0,5) {$\times$};    
  \node[ele,label=below:$n\!+\!1$] (an+1) at (1,5) {};    
  \node[ele,label=below:$n\!+\!2$] (an+2) at (2,5) {};
  \node (cdots1) at (2.8,5) {$\cdots$};
  \node[ele,label=below:$2n-i$] (a2n-i) at (4.5,5) {};
  \node[ele,label=below:$2n-i+1$] (a2n-i+1) at (6.5,5) {};
  \node (cdots2) at (8.2,5) {$\cdots$};
  \node[ele,label=below:$2n-1$] (a2n-1) at (9.5,5) {};
  \node[ele,label=below:$2n$] (a2n) at (10.5,5) {};

    \node[draw, fit={($(ai.west)+(-10pt,0)$) ($(ai+1.east)+(10pt,0)$)}, minimum height=2cm]     (a) {};
    \draw[color=magenta, -, thick] (-5.68,5) to (ai);     
    \draw[color=cyan, -, thick] (-3.32,5) to (ai+1);     

    \node[draw, fit={($(a2n-i.west)+(-15pt,0)$) ($(a2n-i+1.east)+(15pt,0)$)}, minimum height=2cm]     (a) {};
    \draw[color=magenta, -, thick] (3.36,5) to (a2n-i);     
    \draw[color=cyan, -, thick] (7.64,5) to (a2n-i+1);     

  \node[draw,fit= {($(a2.west)+(20pt,0)$) ($(a2n-1.east)+(-20pt,0)$)},minimum height=4cm]  
        (A) {} ;

    \node[ele,label=below:$1$] (b1) at (-8,-2) {}; 
    \node[ele,label=below:$2$] (b2) at (-7,-2) {}; 
    \node at (-6.3,-2) {$\cdots$};
    \node[ele,label=below:$i+1$] (bi+1) at (-5,-2) {}; 
    \node[ele,label=below:$i$] (bi) at (-4,-2) {}; 
    \node at (-2.7,-2) {$\cdots$};
    \node[ele,label=below:$n-1$] (bn-1) at (-2,-2) {}; 
    \node[ele,label=below:$n$] (bn) at (-1,-2) {}; 

  \node[] (x2) at (0,-2) {$\times$};    
  \node[ele,label=below:$n\!+\!1$] (bn+1) at (1,-2) {};    
  \node[ele,label=below:$n\!+\!2$] (bn+2) at (2,-2) {};
  \node (cdots1) at (2.8,-2) {$\cdots$};
  \node[ele,label=below:$2n-i+1$] (b2n-i+1) at (4.5,-2) {};
  \node[ele,label=below:$2n-i$] (b2n-i) at (6.5,-2) {};
  \node (cdots2) at (8.2,-2) {$\cdots$};
  \node[ele,label=below:$2n-1$] (b2n-1) at (9.5,-2) {};
  \node[ele,label=below:$2n$] (b2n) at (10.5,-2) {};

    \node[draw, fit={($(bi+1.west)+(-10pt,0)$) ($(bi.east)+(10pt,0)$)}, minimum height=2cm]     (b) {};
    \draw[color=magenta, -, thick] (-5.68,-2) to[bend left] (bi);     
    \draw[color=cyan, -, thick] (-3.32,-2) to[bend left] (bi+1);     

    \node[draw, fit={($(b2n-i+1.west)+(-15pt,0)$) ($(b2n-i.east)+(15pt,0)$)}, minimum height=2cm]     (b) {};
    \draw[color=magenta, -, thick] (3.36,-2) to[bend left] (b2n-i);     
    \draw[color=cyan, -, thick] (7.64,-2) to[bend left] (b2n-i+1);     
  \node[ele,label=below:$2n-i+1$] (b2n-i+1) at (4.5,-2) {};
  \node[ele,label=below:$2n-i$] (b2n-i) at (6.5,-2) {};

  \node[draw,fit= {($(b2.west)+(20pt,0)$) ($(b2n-1.east)+(-20pt,0)$)},minimum height=4cm]  
        (B) {} ;

    \draw [-latex] (A.-90) to[bend left] (B.90);
    \node (fi) at (2.25,1.5) {$\tilde f_i$};
  
 \end{tikzpicture}
 }
\end{adjustwidth*}
}

\vspace*{5mm}
                
        \item[b.] Define $f_i: \mathbb{D}_{1,n} \to \mathbb{D}_{1,n}$ to be the half Dehn 
            twist shown below where $\times$ denotes the marked point and the punctures are numbered. 

\vspace*{5mm}
 {       
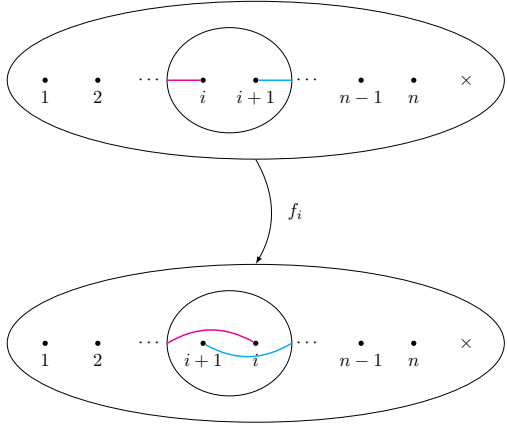
\begin{center}
{
\scalebox{0.7}
{
\begin{tikzpicture}[ele/.style={fill=black,circle,minimum width=0.8pt,inner sep=1pt},every fit/.style={ellipse,draw,inner sep=-2pt}]
  \node[ele,label=below:$1$] (a1) at (0,5) {};    
  \node[ele,label=below:$2$] (a2) at (1,5) {};
  \node (cdots1) at (2,5) {$\cdots$};
  \node[ele,label=below:$i$] (ai) at (3,5) {};
  \node[ele,label=below:$i+1$] (ai+1) at (4,5) {};
  \node (cdots2) at (5,5) {$\cdots$};
  \node[ele,label=below:$n-1$] (an-1) at (6,5) {};
  \node[ele,label=below:$n$] (an) at (7,5) {};
  \node[] (x1) at (8,5) {$\times$};    

    \node[draw, fit={($(ai.west)+(-10pt,0)$) ($(ai+1.east)+(10pt,0)$)}, minimum height=2cm]     (a) {};
    \draw[color=magenta, -, thick] (2.32,5) to (ai);     
    \draw[color=cyan, -, thick] (4.68,5) to (ai+1);     
  \node[draw,fit= {($(a2.west)+(-10pt,0)$) ($(an.east)+(10pt,0)$)},minimum height=3cm]  
        (A) {} ;

  \node[ele,label=below:$1$] (b1) at (0,0) {};    
  \node[ele,label=below:$2$] (b2) at (1,0) {};
  \node (cdots3) at (2,0) {$\cdots$};
  \node[ele,label=below:$i+1$] (bi+1) at (3,0) {};
  \node[ele,label=below:$i$] (bi) at (4,0) {};
  \node (cdots4) at (5,0) {$\cdots$};
  \node[ele,label=below:$n-1$] (bn-1) at (6,0) {};
  \node[ele,label=below:$n$] (bn) at (7,0) {};
  \node[] (x2) at (8,0) {$\times$};    

    \node[draw, fit={($(bi+1.west)+(-10pt,0)$) ($(bi.east)+(10pt,0)$)}, minimum height=2cm]     (b) {};
    \draw[color=magenta, -, thick] (2.32,0) to[bend left] (bi);     
    \draw[color=cyan, -, thick] (4.68,0) to[bend left] (bi+1);     
    
  \node[draw,fit= {($(b2.west)+(-10pt,0)$) ($(bn.east)+(10pt,0)$)},minimum height=3cm] 
        (B) {} ;

    \draw [-latex] (A.-90) to[bend left] (B.90);
    \node (fi) at (4.75,2.5) {$f_i$};
  
 \end{tikzpicture}
 }
}
\end{center}}
\vspace*{5mm}

        \item[c.] Next we justify our notation by defining a double branched cover 
            $q_\times: \mathbb{D}_{1,2n} \to \mathbb{D}_{1,n}$.  Consider $\mathbb{D}_{n+1}$ and $\mathbb{D}_{2n+1}$ as $\mathbb{D}_{1,n}$ and $\mathbb{D}_{1,2n}$ punctured at their marked points, respectively. Define $q: \mathbb{D}_{2n+1} \to \mathbb{D}_{n+1}$ as the double cover obtained by quotienting $\mathbb{D}_{2n+1}$ by a 180-degree rotation about the central puncture---formerly the marked point.  Now we define $q_\times$ as follows:  $q_\times=q$ on $\mathbb{D}_{2n+1}\subseteq \mathbb{D}_{1,2n}$; and $q_\times$ takes the marked point of $\mathbb{D}_{1,2n}$ to the marked point of $\mathbb{D}_{1,n}$.  Thus, we see that $f_iq_\times=q_\times\tilde f_{i}$, which makes $\tilde f_{i}$ a lift of $f_i$.\\

    \end{itemize}

        \item[2.] For $i=n$ in $\mathbb{D}_{1,2n}$.
            \begin{itemize}
                \item[a.] Define $\tilde f_n: \mathbb{D}_{1,2n} \to \mathbb{D}_{1,2n}$ to be the (full) Dehn twist about the annulus shown below. Observe that $\tilde f_n$ is a representative of the mapping class corresponding to the half Dehn twist swapping punctures $n$ and $n+1$.

\vspace*{5mm}
{
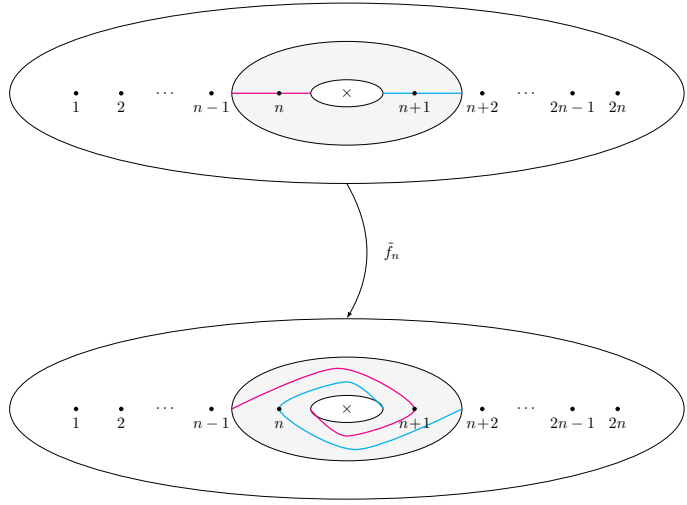
\begin{center}
\scalebox{0.6}
{
\begin{tikzpicture}[ele/.style={fill=black,circle,minimum width=0.8pt,inner sep=1pt},every fit/.style={ellipse,draw,inner sep=-2pt}]
    \node[ele,label=below:$1$] (a1) at (-6,5) {}; 
    \node[ele,label=below:$2$] (a2) at (-5,5) {}; 
    \node at (-4,5) {$\cdots$};
    \node[ele,label=below:$n-1$] (an-1) at (-3,5) {}; 
    \node[ele,label=below:$n$] (an) at (-1.5,5) {}; 

  \node[] (x1) at (0,5) {$\times$};    
  \node[ele,label=below:$n\!+\!1$] (an+1) at (1.5,5) {};    
  \node[ele,label=below:$n\!+\!2$] (an+2) at (3,5) {};
  \node (cdots1) at (4,5) {$\cdots$};
  \node[ele,label=below:$2n-1$] (a2n-1) at (5,5) {};
  \node[ele,label=below:$2n$] (a2n) at (6,5) {};

    \node[fill=gray!8, draw, fit={($(an.west)+(-9pt,0)$) ($(an+1.east)+(9pt,0)$)}, 
        minimum height=2.3cm] (a) {};
    \draw[fill=white] (0,5) ellipse (0.8 and 0.3);

    \draw[color=magenta, -, thick] (-2.55,5) to (-0.8,5);     
    \draw[color=cyan, -, thick] (2.55,5) to (0.8,5);     
    \node[ele,label=below:$n$] (an) at (-1.5,5) {}; 

  \node[] (x1) at (0,5) {$\times$};    
  \node[ele,label=below:$n\!+\!1$] (an+1) at (1.5,5) {};    

  \node[draw,fit= {($(a1.west)+(20pt,0)$) ($(a2n.east)+(-20pt,0)$)},minimum height=4cm]  
        (A) {} ;

    \node[ele,label=below:$1$] (b1) at (-6,-2) {}; 
    \node[ele,label=below:$2$] (b2) at (-5,-2) {}; 
    \node at (-4,-2) {$\cdots$};
    \node[ele,label=below:$n-1$] (bn-1) at (-3,-2) {}; 
    \node[ele,label=below:$n$] (bn) at (-1.5,-2) {}; 

  \node[] (x2) at (0,-2) {$\times$};    
  \node[ele,label=below:$n\!+\!1$] (bn+1) at (1.5,-2) {};    
  \node[ele,label=below:$n\!+\!2$] (bn+2) at (3,-2) {};
  \node (cdots1) at (4,-2) {$\cdots$};
  \node[ele,label=below:$2n-1$] (b2n-1) at (5,-2) {};
  \node[ele,label=below:$2n$] (b2n) at (6,-2) {};

    \node[fill=gray!8, draw, fit={($(bn.west)+(-9pt,0)$) ($(bn+1.east)+(9pt,0)$)}, 
        minimum height=2.3cm] (b) {};

    \draw[fill=white] (0,-2) ellipse (0.8 and 0.3);
    \draw [color=magenta, -, thick] plot [smooth, tension=1.1] coordinates { (-2.55,-2) (-0.2,-1.1) (1.5,-2) (0,-2.6) (-0.82,-2)};
    \draw [color=cyan, -, thick] plot [smooth, tension=1.1] coordinates { (2.55,-2) (0.2,-2.9) (-1.5,-2) (0,-1.4) (0.82,-2)};    
    \node[ele,label=below:$n$] (bn) at (-1.5,-2) {}; 

  \node[] (x2) at (0,-2) {$\times$};    
  \node[ele,label=below:$n\!+\!1$] (bn+1) at (1.5,-2) {};    

  \node[draw,fit= {($(b1.west)+(20pt,0)$) ($(b2n.east)+(-20pt,0)$)},minimum height=4cm]  
        (B) {} ;

    \draw [-latex] (A.-90) to[bend left] (B.90);
    \node (fi) at (1,1.5) {$\tilde f_n$};
  
 \end{tikzpicture}
}

 \end{center}
}

\vspace*{5mm}

        \item[b.] Note that map $\tilde f_n$ commutes with the 180-degree rotation used in 
            defining $q_x$.  Therefore there exists $f_n:\mathbb{D}_{1,n} \to \mathbb{D}_{1,n}$ such that $f_nq_x=q_x\tilde f_n$.  Therefore $\tilde f_n$ is a lift of $f_n$. \\
         \end{itemize}

        \item[3.] Define $\Phi$ on generators such that $\Phi([f_i])= \text{\textit{forget}}([\tilde f_{i}])$, where $[g]$ is the mapping class of $g$ and the \textit{forget} map forgets the marked point in Mod$(\mathbb{D}_{1,2n}$).  
    \end{itemize}

\vspace*{5mm}
\noindent \underline{Part 2}: \\

We are now ready to prove that map $\Phi$ is a homomorphism.

\begin{prop} \label{prop:MCG homomorphism}
    Map $\Phi: \text{Mod}(\mathbb{D}_{n,1})\to \text{Mod}(\mathbb{D}_{2n})$ is a homomorphism.  
\end{prop}

\begin{proof}
It suffices to show that if a product of generators $\prod[f_{\sigma(i)}]$ is the identity in Mod$(\mathbb{D}_{1,n})$, then $\Phi(\prod[f_{\sigma(i)}])$ is the identity in Mod$(\mathbb{D}_{2n})$.  
Suppose $\prod [f_{\sigma(i)}]$ is the identity in Mod$(\mathbb{D}_{1,n})$.  In particular, this means that there is an isotopy $h_t$ from $\prod f_{\sigma(i)}$ to the identity that fixes the marked point in $\mathbb{D}_{1,n}$.    

Let $h_t$ be a homotopy at time $t$ from $\mathbb{D}_{n+1}$ to $\mathbb{D}_{n+1}$.  Then the homotopy lifting property applies to $h_t\circ q: \mathbb{D}_{2n+1}\to \mathbb{D}_{n+1}$ and fixes the marked puncture in $\mathbb{D}_{n+1}$, which gives us a unique homotopy $\tilde h_t: \mathbb{D}_{2n+1}\to \mathbb{D}_{2n+1}$ that fixes the marked puncture in $\mathbb{D}_{2n+1}$.  Therefore, we can restrict $\tilde h_t$ to go from $\mathbb{D}_{1,2n}$ to $\mathbb{D}_{1,2n}$.  

Because $\prod [f_{\sigma(i)}]$ is the identity in Mod$(\mathbb{D}_{1,n})$, we have that $\tilde h_t$ is a homotopy from $\prod \tilde f_{\sigma(i)}$ to the identity in Mod$(\mathbb{D}_{1,2n})$\footnote{Note that $\tilde h_t$ cannot be a homotopy from $\prod \tilde f_{\sigma(i)}$ to the involution in Mod$(\mathbb{D}_{1,2n})$ because the boundary points must be fixed.}.  This means that $[\prod \tilde f_{\sigma(i)}]$ and \textit{forget}$\big([\prod \tilde f_{\sigma(i)}]\big)$ both equal the identity in Mod$(\mathbb{D}_{2n})$.  Therefore, the homomorphism $\Phi$ is well defined.
\end{proof}

\vspace*{5mm}
\noindent \underline{Part 3}: \\

Now that we know $\Phi$ is a homomorphism, we can prove that it is also injective.  But first, we need to define some terminology that is needed for the proof.

\begin{defn} 
    \cite{margalit2021braid}
    Let $q: S\to X$ be a covering map of surfaces.  Map $\tilde f: S\to S$ is \textit{fiber preserving} if there exists a map $f:X\to X$ such that the below diagram commutes:

    \begin{figure}[H]
        \centering
        \begin{tikzcd}[row sep=2em, column sep=1em]
        S \arrow[swap]{d}{q} \arrow{rr} {\tilde f} & & 
                S \arrow{d} {q} \\
        X \arrow{rr} {f} & & X 
    \end{tikzcd}
        \label{fig:fiber-preserving def}
    \end{figure}
\end{defn}

\begin{defn} \cite{margalit2021braid} Covering map $q$ has the \textit{Birman-Hilden property} if every fiber-preserving homeomorphism that is homotopic to the identity is also homotopic to the identity through fiber-preserving homeomorphisms. 

\end{defn}

\begin{thm} \label{thm:BHMH}
    (Birman-Hilden-Maclachlan–Harvey as cited in \cite{margalit2021braid})
    Let $q : S \to X$ be a finite-sheeted regular branched covering map where $S$ is a hyperbolic surface. Then $q$ has the Birman–Hilden property.
\end{thm}

\begin{thm} 
    $\Phi: \text{Mod}(\mathbb{D}_{1,n})\to \text{Mod}(\mathbb{D}_{2n})$ is injective.  
\end{thm}

\begin{proof}
    Suppose $\Phi(\prod[f_{\sigma(i)}])$ is the homotopy class of the identity $id_{\mathbb{D}_{2n}}: \mathbb{D}_{2n}\to \mathbb{D}_{2n}$.  Let $\tilde f: \mathbb{D}_{2n}\to \mathbb{D}_{2n}$ be defined such that $$\tilde f(p)= 
        \begin{cases}
            (\prod \tilde f_{\sigma(i)})(p)  & p\in \mathbb{D}_{2n+1}\subseteq \mathbb{D}_{2n} \\
            p  & p=\mathbb{D}_{2n}\setminus \mathbb{D}_{2n+1}
        \end{cases}.$$
    By definition, $\Phi(\prod[f_{\sigma(i)}])=[\tilde f]\in \text{Mod}(\mathbb{D}_{2n})$.  Because $\Phi(\prod[f_{\sigma(i)}])=id_{\mathbb{D}_{2n}}$, then $\tilde f$ is isotopic to $id_{\mathbb{D}_{2n}}$.  Map $\tilde f$ is also fiber-preserving because it lifts $\prod f_{\sigma(i)}$.  

    Because $q_\times: \mathbb{D}_{1,2n}\to \mathbb{D}_{1,n}$ is a two-sheeted regular branched covering map and $\mathbb{D}_{1,2n}$ is a hyperbolic surface, then $q_\times$ has the Birman-Hilden property by Theorem \ref{thm:BHMH}.  This means that $\tilde f$ is homotopic to $id_{\mathbb{D}_{2n}}$ through a family of fiber-preserving homeomorphisms $\tilde h_{t}$ that induce a homotopy $h_t:X\to X$ from $\prod f_{\sigma(i)}$ to $id_{\mathbb{D}_{1,n}}$.  That is $[\prod f_{\sigma(i)}]=\prod[f_{\sigma(i)}]$ is the identity in Mod$(\mathbb{D}_{1,n})$ as required.
\end{proof}

\vspace*{5mm}
\noindent \underline{Part 4}: \\

We now define the \textit{mixed braid group} and state the relationship between it and the Artin group corresponding to $B_n$.  Then we recall some known isomorphisms between between Artin groups and braid groups and between  braid groups and mapping class groups. 

\begin{defn}
    
    In \cite{lambropoulou2000braid}, Lambropoulou defined the mixed braid group $\mathscr{B}_{g,n}$ to be the subgroup of $\mathscr{B}_{g+n}$ consisting of the set of all braids with $g$ fixed strands and $n$ moving strands. (See also \cite{cavicchioli2023mixed}.) Group $\mathscr{B}_{g,n}$ has two types of generators:  
    \begin{enumerate}
        \item[1.] $g$ ``loop generators" $\alpha_1, ..., \alpha_g$ where $\alpha_i$ is a looping of the $1^{\text{st}}$ moving strand (ms) around the i-th fixed strand (fs):

\hspace*{-30mm}
\scalebox{0.8}
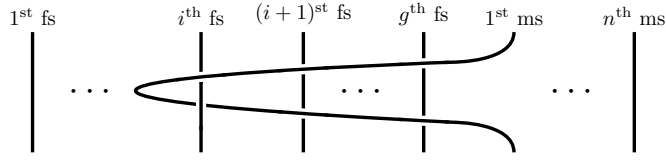
{
\begin{tikzpicture}
\begin{knot}[
flip crossing=2,
flip crossing=3,
flip crossing=4,
flip crossing=5,
flip crossing=6,
]

    \strand[black, ultra thick] (-10,0) .. controls +(0,0) and +(0,0) .. (-10,2);
    \strand[black, ultra thick] (-7.2,0) .. controls +(0,0) and +(0,0) .. (-7.2,2);
    \strand[black, ultra thick] (-5.5,0) .. controls +(0,0) and +(0,0) .. (-5.5,2);
    \strand[black, ultra thick] (-3.5,0) .. controls +(0,0) and +(0,0) .. (-3.5,2);
    
    \strand[black, ultra thick] (-2,0) .. controls +(0,0.5) and +(0,0) .. (-3,0.5) ..
        controls +(-11.9,0.6) and +(0,0) .. (-3,1.5).. controls +(0,0) and +(0,-0.5) .. (-2,2);
    
    \strand[black, ultra thick] (0,0) .. controls +(0,0) and +(0,0) .. (0,2);
    
    \node[font=\huge] at (-1,1){$\cdots$};
    \node[font=\huge] at (-4.5,1){$\cdots$};
    \node[font=\huge] at (-9,1){$\cdots$};
            
    \node[anchor=south] at (0,2) 
        {$n^{\text{th}} \text{ ms}$ };
    \node[anchor=south] at (-2,2) 
        {$1^{\text{st}} \text{ ms}$};
    \node[anchor=south] at (-3.5,2) 
        {$g^{\text{th}} \text{ fs}$};
    \node[anchor=south] at (-5.5,2) 
        {$(i+1)^{\text{st}} \text{ fs}$};
    \node[anchor=south] at (-7.2,2) 
        {$i^{\text{th}} \text{ fs}$};
    \node[anchor=south] at (-10,2) 
        {$1^{\text{st}} \text{ fs}$};

\end{knot}

\end{tikzpicture}
}
\hfill\\

        Thus we obtain the following relations between $\alpha_i$ and the standard braid relations $\theta_1$, ..., $\theta_g$ where $\theta_i$ is the generator between the $i^{\text{th}}$ fixed strand and the $(i+1)^{\text{st}}$ fixed strand:  
        \begin{itemize}
            \item[] $\alpha_g=\theta_g^2$
            \item[] $\alpha_{g-1}=\theta_g\theta_{g-1}^2\theta_g^{-1}$
            \item[] $\alpha_i = \theta_g\theta_{g-1}...\theta_{i+1}\theta_{i}^2
                    \theta_{i+1}^{-1}...\theta_{g-1}^{-1}\theta_g^{-1}$ for $1\leq i\leq g$.
        \end{itemize}
        
        \item[2.] $n-1$ elementary crossings $\sigma_{1}, \ ... \, , \ \sigma_{n-1}$ 
            (i.e., moving strands $j$ and $j+1$ interchange positions)

  \end{enumerate}

The relations for $\mathscr{B}_{g,n}$ are as follows where the first two are the standard braid relations:
\begin{enumerate}
    \item[1.] $\sigma_i\sigma_j=\sigma_j\sigma_i$, $|i-j|>1$, $1\leq i,j \leq n-1$
    \item[2.] $\sigma_i\sigma_{i+1}\sigma_i=\sigma_{i+1}\sigma_i\sigma_{i+1}$, $1\leq i\leq n-2$
    
    \item[3.] $\alpha_i\sigma_j=\sigma_j\alpha_i$, $1\leq i\leq g$, $2\leq j\leq n-1$
    \item[4.] $\alpha_i\sigma_{1}\alpha_i\sigma_{1}=\sigma_{1}\alpha_i\sigma_{1}\alpha_i$, 
            $1\leq i\leq g$
    \item[5.] $\alpha_i(\sigma_{1}\alpha_r\sigma_{1}^{-1})=
            (\sigma_{1}\alpha_r\sigma_{1}^{-1})\alpha_i$, $1\leq r<i\leq g$ 
\end{enumerate}
\end{defn}

In particular, the mixed braid group $\mathscr{B}_{1,n}$ is isomorphic to the Artin group corresponding to $B_n$.  In $\mathscr{B}_{1,n}$, we will call $\alpha_1$ just $\alpha$ because there is only one fixed strand.  Moreover, it is a classic result that $\mathscr{B}_{2n}$ is isomorphic to the Artin group $A(B_{2n-1})$ such that the generators $\sigma_i$ of $\mathscr{B}_{2n}$ get mapped to the generators $t_i$ of $A(B_{2n-1})$.  

Furthermore, we know that the standard braid group $\mathscr{B}_n$ is isomorphic to the mapping class group of the closed disk with $n$ punctures where each generator $\sigma_i$ of $\mathscr{B}_{n}$ is mapped to the homotopy class of a homeomorphism of $\mathbb{D}_{n}$ that corresponds to a half Dehn twist interchanging the $i^{th}$ and $(i+1)^{st}$ punctures (see, for example, Chapter 9 in \cite{farb2011primer}).  
Thus, $\mathscr{B}_{2n}$ is isomorphic to Mod$(\mathbb{D}_{2n})$.  We can deduce that the mixed braid group $\mathscr{B}_{1,n}$ is also isomorphic to the mapping class group of the closed disk with one marked point and $n$ punctures, $\mathbb{D}_{1,n}$.

\vspace*{5mm}
\noindent \underline{Part 5}: \\

We now define the following maps in order to connect $\phi: A(B_n)\to A(A_{2n-1})$ to $\Phi: \text{Mod}(\mathbb{D}_{1,n}) \to \text{Mod}(\mathbb{D}_{2n})$ and show that the two squares in the below diagram commute.

    \begin{figure}[!ht]
        \centering
        \begin{tikzcd}[row sep=2em, column sep=1em]
        A(B_n) \arrow[->]{d}{\cong} [swap] {\iota_1} \arrow{rr} {\phi} & & 
                A(A_{2n-1}) \arrow[->]{d} {I_1} [swap]{\cong} \\
        \mathscr{B}_{1,n} \arrow[->]{d}{\cong} [swap]{\iota_2} \arrow{rr} {\varphi} & & 
                \mathscr{B}_{2n} \arrow[->]{d} {I_2} [swap] {\cong}\\
        \text{Mod}(\mathbb{D}_{1,n}) \arrow{dr} \arrow{rr} {\Phi} & & 
                \text{Mod}(\mathbb{D}_{2n}) \\
        & \text{Mod}(\mathbb{D}_{1,2n}) \arrow[swap]{ur} {\text{\textit{forget}}}
    \end{tikzcd}
        \label{fig:blueprint}
        \caption{Blueprint of maps}
    \end{figure}

Definition of Maps:
    \begin{itemize}
        \item[1.] Map $\iota_1$ is the standard isomorphism between $A(B_n)$ and 
            $\mathscr{B}_{1,n}$ after reversing the order of the strands in $\mathscr{B}_{1,n}$.
        \item[2.] Map $I_1$ is the standard isomorphism between $A(A_{2n-1})$ and 
            $\mathscr{B}_{2n}$.  
        \item[3.] Define $\iota_2$ on generators such that $\iota_2(\alpha)=f_n$ and for 
            $i<n$, $\iota_2(\sigma_{n-i})=f_i$. Observe that this recovers the usual identification of braids as mapping classes.  
        
        \item[4.] Define $I_2$ on generators corresponding to a half-twist of neighboring 
            strands in the braid group to a half Dehn twist of the corresponding neighboring punctures.  Observe that $I_2(\sigma'_i\sigma'_{2n-i})=\tilde f_i$ and $I_2(\sigma'_n)=\tilde f_n$. 
        \item[5.] Define $\varphi$ on generators such that $\varphi(\alpha)=\sigma'_n$ and 
            for $i<n$, 
            $\varphi(\sigma_{n-i}) = \sigma'_{n-(n-i)}\sigma'_{n+(n-i)}= \sigma'_{i}\sigma'_{2n-i}$.  \\
    \end{itemize}

With these maps, we see that $I_1\phi = \varphi\iota_1$ and $I_2\varphi = \Phi\iota_2$.

\vspace*{5mm}
\noindent \underline{Part 6}: \\

Proof of Proposition \ref{gggg2} that $\phi: A(B_n)\to A(A_{2n-1})$ is an injective homomorphism:  Because each square diagram in Figure \ref{fig:blueprint} commutes, we have that $I_2I_1\phi = \Phi\iota_2\iota_1$.  Because $\Phi$ is an injective homomorphism and $\iota_j$ and $I_j$ are isomorphisms for $j=1,2$, then $\phi$ is also an injective homomorphism.    \hfill $\square$\\

We finish this section with two remaining statements which will come into use later. 
The first is implicitly proven in \cite[Prop.~6.6]{haettel2023lattices}. We will extract its statement and proof here for completeness.

\begin{lem} \label{cor:lessiffimage}
   The vertex $g \hat t_i$ is in the image of $\psi$ if and only if $g \hat t_i \leq \sigma(g \hat t_i)$. 
\end{lem}

\begin{proof}
First suppose $g \hat t_i$ is in the image of $\psi$. By choosing a different $g \in g \hat t_i$, we may assume $g$ is in the image of $\phi$. Then
\[
    \sigma(g \hat t_i) = \sigma(g) \hat t_{2n-i} = g \hat t_{2n-i}. 
\]
Therefore $g \hat t_i \cap \sigma(g \hat t_i) \not= \varnothing$ (both sets contain $g$), so these vertices are either adjacent or equal. In either case, since $i \leq (n+1)/2$, the definition of the partial order implies $g \hat t_i \leq \sigma(g \hat t_i)$.

Now suppose $g \hat t_i \leq \sigma(g \hat t_i)$. This implies $g \hat t_i \cap \sigma(g \hat t_i) = g \hat t_i \cap \sigma(g) \hat t_{2n-i} \not= \varnothing$.
So up to replacing $g$ with some element of $g \hat t_i \cap \sigma(g) \hat t_{2n-i}$, we may assume $\sigma(g) \in g \hat t_{2n-i}$. Choose $h \in \hat t_{2n-i}$ such that $\sigma(g) = gh$. Since $\sigma$ is an involution,
\[
    g = \sigma(\sigma(g)) = \sigma(gh) = \sigma(g)\sigma(h) = gh\sigma(h).
\]
This implies $\sigma(h) = h^{-1} \in \hat t_i \cap \hat t_{2n-i}$.
By \cite{charney1995geodesic}, $h$ can be factored uniquely as $h = h_1 h_2^{-1}$, with $h_1$ and $h_2$ elements of the ``positive monoid'' $A(A_{2n-1})^+$, such that the right greatest common divisor of $h_1$ and $h_2$ is the identity in the monoid. Moreover, since $h \in \hat t_i \cap \hat t_{2n-i}$, both $h_1$ and $h_2$ are in $\hat t_i \cap \hat t_{2n-i}$. Then since $\sigma$ preserves the positive monoid, 
\[
    \sigma(h_1)\sigma(h_2)^{-1} = \sigma(h) = h^{-1} = h_2h_1^{-1}.  
\]
But by uniqueness, $h_2 = \sigma(h_1)$, so we can write $h = h_1\sigma(h_1)^{-1}$. 
Now,
\[
    \sigma(gh_1) = \sigma(g)\sigma(h_1) = gh \sigma(h_1) = gh_1\sigma(h_1)^{-1}\sigma(h_1) = gh_1,
\]
so $gh_1$ is in the image of $\phi$ by \cite[Thm.~4]{crisp2000symmetrical}.
Since $h_1 \in \hat t_i$,
this means $g \hat t_i = gh_1\hat t_i$ is in the image of $\psi$.
\end{proof}

The following is a restatement of \cite[Thm.~5.6]{huang2025cycles}.

\begin{thm} \label{thm:jingyinlemma}
    Suppose $v_1$, $v_2$, and $v_3$ are vertices of $D(A_n)$ of type $\hat t_i$ for some $i \geq 3$. Suppose also that each pair of these vertices has a lower bound of type $\hat t_i$ or $\hat t_{i-1}$. Then there is a vertex $v$ of $D(A_n)$ which is a lower bound of $\{v_1,v_2,v_3\}$ and has type $\hat t_i$, $\hat t_{i-1}$ or $\hat t_{i-2}$.
\end{thm}

\section{Filling loops in \texorpdfstring{$D(B_3)$}{D(B\_3)}}
\label{sec:shrinking}

Now we undertake verifying conditions (5) and (6) of Theorem \ref{thm:B3 CAT1 Criteria} for $D(B_3)$. 

\subsection{The 4-cycle and 6-cycle}

We begin with the 4-cycle and the 6-cycle, both of which follow easily from our setup.

\begin{lem} \label{lemma:filling 4 cycles in DB3}
    Loops in $D(B_3)$ of the type in Figure \ref{fig:checking edge loops}(a) can be filled to subcomplexes of $D(B_3)$ of the type in Figure \ref{fig:filling the edge loops}(a).
\end{lem}

\begin{proof}
    Suppose $\gamma$ is a(n embedded) loop of this type. The vertices of type $\hat s_1$ have a common upper bound, and thus have a least upper bound $v$ By Proposition \ref{6.6}. The join $v$ cannot be type $\hat s_3$ or $\hat s_1$, since this would contradict the assumption that $\gamma$ is embedded. Thus $v$ is type $\hat s_2$. Since $v$ is a least upper bound, both vertices of $\gamma$ are comparable to $v$, and in particular there is an edge from $v$ to both of these vertices. Since $D(B_3)$ is flag, this means $\gamma$ can be filled as claimed.
\end{proof}

\begin{lem} \label{lem:filling 6 cycle in DB3}
    Suppose $\gamma$ is an embedded closed edge loop in $D(B_3)$ of the form in Figure \ref{fig:checking edge loops}(b). Then $\gamma$ is contained in a subcomplex of the form in Figure \ref{fig:filling the edge loops}(b). 
\end{lem}

\begin{proof}
    The vertices of $\gamma$ of type $\hat s_1$ are pairwise upper bounded by vertices of type $\hat s_3$. By Proposition \ref{6.6}, this means they have a join, say, $v$, whose type is $\hat s_i$ for some $i$. 
    This gives a subcomplex of the form
    
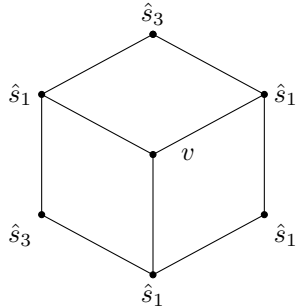
\[
    \begin{tikzpicture}[scale=0.8]
    \coordinate (s1 top) at (0,2) ;
    \coordinate (s1 right) at (1.85,-1);
    \coordinate (s1 left) at (-1.85,-1) ;
    \coordinate (s2 bottom) at (0,-2) ;
    \coordinate (s2 right) at (1.85,1) ;
    \coordinate (s2 left) at (-1.85,1) ;
    \coordinate (center) at (0,0);
    
    \draw (s1 top) -- (s2 right)  -- (s1 right) -- 
        (s2 bottom) -- (s1 left) -- (s2 left) -- (s1 top);

    \draw (center) -- (s2 bottom);
    \draw (center) -- (s2 left);
    \draw (center) -- (s2 right);

    \filldraw (s1 top) circle (0.05cm);
    \filldraw (s2 right) circle (0.05cm);
    \filldraw (s1 right) circle (0.05cm);
    \filldraw (s2 bottom) circle (0.05cm);
    \filldraw (s1 left) circle (0.05cm);
    \filldraw (s2 left) circle (0.05cm);
    \filldraw (s1 top) circle (0.05cm);
    
    \filldraw (center) circle (0.05cm);

    \node[above] at (s1 top) {$\hat s_3$} ;
    \node[right] at (s2 right) {$\hat s_1$} ;
    \node[below right] at (s1 right) {$\hat s_1$} ;
    \node[below] at (s2 bottom)  {$\hat s_1$} ;
    \node[below left] at (s1 left) {$\hat s_3$} ;
    \node[left] at (s2 left) {$\hat s_1$} ;
    \node[right=0.25] at (center) {$v$} ;
    \end{tikzpicture}
    \]
    Since $\gamma$ is embedded, we cannot have $i = 1$. If $i = 2$, then the 4-cycles in this diagram may be filled to obtain a subcomplex of the form in Figure \ref{fig:bad 6 filling} using Lemma \ref{lemma:filling 4 cycles} (the proof of which still applies to $D(B_3)$, even though we have not fully verified conditions (5) and (6)).
    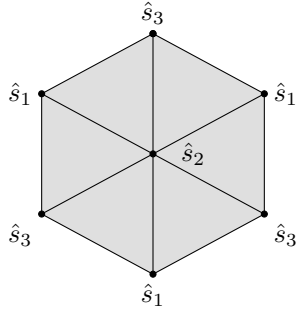
\begin{figure}[!ht]
        \centering
        \begin{tikzpicture}[scale=0.8]
        \coordinate (s1 top) at (0,2) ;
        \coordinate (s1 right) at (1.85,-1);
        \coordinate (s1 left) at (-1.85,-1) ;
        \coordinate (s2 bottom) at (0,-2) ;
        \coordinate (s2 right) at (1.85,1) ;
        \coordinate (s2 left) at (-1.85,1) ;
        \coordinate (center) at (0,0);
        
        \draw[fill=gray!25] (s1 top) -- (s2 right)  -- (s1 right) -- 
            (s2 bottom) -- (s1 left) -- (s2 left) -- (s1 top);
    
        \draw (s1 top) -- (s2 bottom);
        \draw (s1 right) -- (s2 left);
        \draw (s1 left) -- (s2 right);

        \filldraw (s1 top) circle (0.05cm);
        \filldraw (s2 right) circle (0.05cm);
        \filldraw (s1 right) circle (0.05cm);
        \filldraw (s2 bottom) circle (0.05cm);
        \filldraw (s1 left) circle (0.05cm);
        \filldraw (s2 left) circle (0.05cm);
        \filldraw (s1 top) circle (0.05cm);
        
        \filldraw (center) circle (0.05cm);
    
        \node[above] at (s1 top) {$\hat s_3$} ;
        \node[right] at (s2 right) {$\hat s_1$} ;
        \node[below right] at (s1 right) {$\hat s_3$} ;
        \node[below] at (s2 bottom)  {$\hat s_1$} ;
        \node[below left] at (s1 left) {$\hat s_3$} ;
        \node[left] at (s2 left) {$\hat s_1$} ;
        \node[right=0.25] at (center) {$\hat s_2$} ;
        \end{tikzpicture}
        \caption{A bad filling}
        \label{fig:bad 6 filling}
    \end{figure}
    But in this case, since $v$ is type $\hat s_2$, the link of $v$ is complete bipartite, thus has diameter 4 (under the edge-path metric), implying $\gamma$ could not be embedded. So we must have $i = 3$, in which case the 4-cycles in the diagram may be filled to obtain a subcomplex of the form in Figure \ref{fig:filling the edge loops}(b) using Lemma \ref{lemma:filling 4 cycles in DB3}.
\end{proof}

\subsection{The 8-cycle}
We next turn our attention to the 8-cycle, which is, in some sense, the first non-trivial case. To do this, we must utilize the material from Section \ref{sec:geometry} and the geometry of $D(A_5)$. 

\begin{lem} \label{lem:filling 8 cycle in DB3}
    If $\gamma$ is an embedded edge path in $D(B_3)$ of the form in Figure \ref{fig:checking edge loops}(c), then it is contained in a subcomplex of $D(B_3)$ of the form in Figure \ref{fig:filling the edge loops}(c). 
\end{lem}

\begin{proof}
Label the vertices of $\gamma$ by $v_{i,j}$, where $v_{i,j}$ is type $\hat s_i$ and $j$ is assigned as follows:
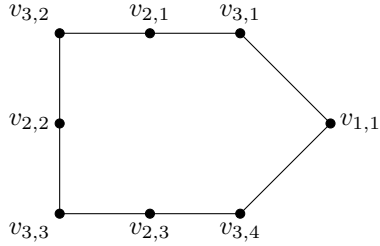
\[
\begin{tikzpicture}[scale=1.2]
        \coordinate (s31) at (0,0) ;
        \coordinate (s21) at (1,0) ;
        \coordinate (s32) at (2,0);
        \coordinate (s1) at (3,-1) ;
        \coordinate (s33) at (2,-2) ;
        \coordinate (s22) at (1,-2) ;
        \coordinate (s34) at (0,-2) ;
        \coordinate (s23) at (0,-1) ;
        
        \filldraw (s31) circle (0.05cm);
        \filldraw (s21) circle (0.05cm);
        \filldraw (s32) circle (0.05cm);
        \filldraw (s1) circle (0.05cm);
        \filldraw (s33) circle (0.05cm);
        \filldraw (s22) circle (0.05cm);
        \filldraw (s34) circle (0.05cm);
        \filldraw (s23) circle (0.05cm);

        \draw (s31) -- (s21)  -- (s32) -- (s1) --(s33) -- (s22) -- (s34)  -- (s23) -- (s31);

        \node[above left] at (s31) {$v_{3,2}$} ;
        \node[above] at (s21) {$v_{2,1}$} ;
        \node[above] at (s32) {$v_{3,1}$} ;
        \node[right] at (s1)  {$v_{1,1}$} ;
        \node[below] at (s33) {$v_{3,4}$} ;
        \node[below] at (s22) {$v_{2,3}$} ;
        \node[below left] at (s34) {$v_{3,3}$} ;
        \node[left] at (s23) {$v_{2,2}$} ;
        \end{tikzpicture}
\]
The Hasse diagram of the vertices $V(\gamma)$ of $\gamma$ under the ordering on the vertices of $D(B_3)$ is as follows:
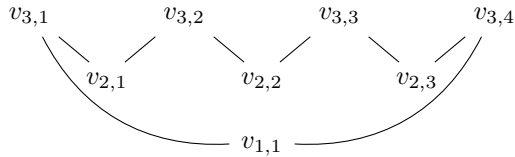
\[
    \begin{tikzcd}[row sep=1em, column sep=0.5em]
    v_{3,1} 
            &
            & v_{3,2} %
            &
            & v_{3,3} %
            & 
            & v_{3,4} %
            \\
    & v_{2,1} \arrow[dash]{ul}{} \arrow[dash]{ur}{}     
            &
            & v_{2,2} \arrow[dash]{ul} \arrow[dash]{ur}
            &
            & v_{2,3} \arrow[dash]{ul} \arrow[dash]{ur} \\
    & & & v_{1,1} \arrow[dash,bend left=33]{uulll}{} 
                    \arrow[dash,swap,bend right=33]{uurrr}
\end{tikzcd}
\]
For each $i$ and $j$, let $u_{i,j} = \psi(v_{i,j})$. 
Since $\psi$ is an order-preserving injection, the Hasse diagram for the $u_{i,j}$ is identical to that of the $v_{i,j}$. Note that $\sigma(u_{i,j})$ is a type $\hat t_{6-i}$ vertex. 
Since each $u_{i,j}$ is in the image of $\psi$, we know $u_{i,j} \leq \sigma(u_{i,j})$. 
In particular, for each $j$, since $u_{3,j}$ and $\sigma(u_{3,j})$ are the same type, they must be equal. Since $\sigma$ is an order-reversing bijection, the Hasse diagram of the $u_{i,j}$ along with their images under $\sigma$ is as follows:
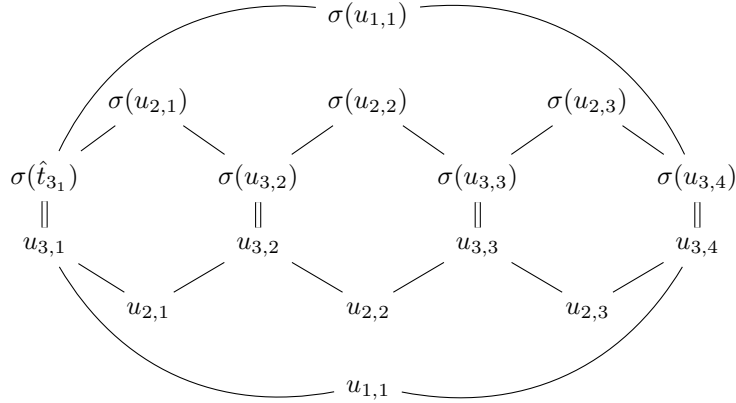
\[
\begin{tikzcd}[row sep=1em, column sep=0.25em]
    & & & \sigma(u_{1,1}) \\[0.5em]
    & \sigma(u_{2,1})   
            &
            & \sigma(u_{2,2}) 
            &
            & \sigma(u_{2,3}) \\    
    \sigma(\hat t_{3_1}) \arrow[dash]{ur} \arrow[dash, bend left=35]{uurrr}
            &
            & \sigma(u_{3,2}) \arrow[dash]{ur} \arrow[dash]{ul}
            &
            & \sigma(u_{3,3}) \arrow[dash]{ur} \arrow[dash]{ul}
            &
            & \sigma(u_{3,4}) \arrow[dash]{ul} \arrow[dash, bend right=35]{uulll}
            \\
    u_{3,1} \arrow[equal]{u} %
            &
            & u_{3,2} \arrow[equal]{u}
            &
            & u_{3,3} \arrow[equal]{u}
            & 
            & u_{3,4} \arrow[equal]{u}\\
    & u_{2,1} \arrow[dash]{ul} \arrow[dash]{ur}     
            &
            & u_{2,2} \arrow[dash]{ul} \arrow[dash]{ur}
            &
            & u_{2,3} \arrow[dash]{ul} \arrow[dash]{ur} \\[0.5em]
    & & & u_{1,1} \arrow[dash,bend left=35]{uulll}{} 
                   \arrow[dash,swap,bend right=35]{uurrr}
\end{tikzcd}
\]
Let $w_1$ be the join (i.e., the least upper bound) of $u_{2,1}$ and $u_{2,3}$ in $D_0(A_5)$ and 
let $w_2$ be the join of $u_{1,1}$ and $u_{2,2}$ in $D_0(A_5)$.
Notice that $\sigma(u_{2,2})$ is an upper bound of $u_{2,1}$ and $u_{2,3}$, and $\sigma(u_{2,1})$ is an upper bound of $u_{1,1}$ and $u_{2,2}$.  This means $w_1$ and $w_2$ are actually contained in $D(A_5)$ (i.e., neither are $0$ or $1$).
Again since $\sigma$ is an order-reversing injection, $\sigma(w_1)$ is the meet (i.e., greatest lower bound) of $\sigma(u_{2,1})$ and $\sigma(u_{2,3})$, and $\sigma(w_2)$ is the meet of $\sigma(u_{1,1})$ and $\sigma(u_{2,2})$.
Adding these points to the above Hasse diagram gives the following:

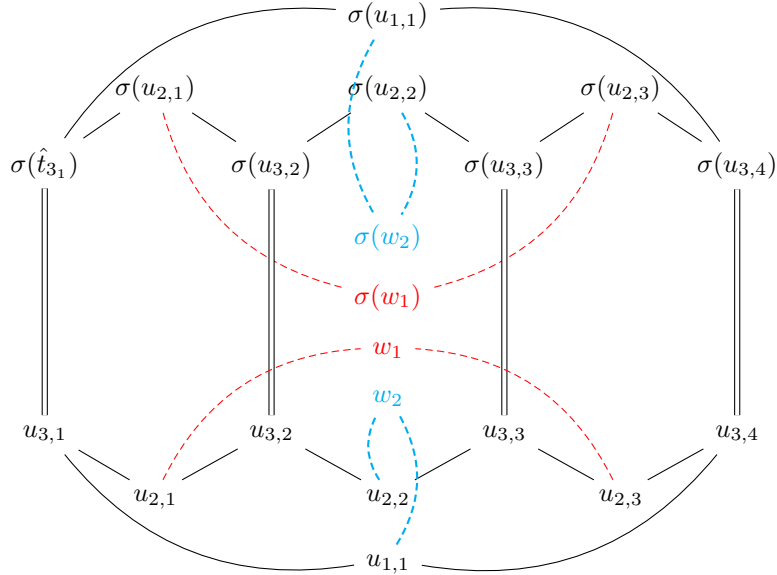
\[
 \begin{tikzcd}[row sep=1em, column sep=0.5em]
    & 
         & & 
             \sigma(u_{1,1}) \arrow[dash, dashed,bend right, thick, cyan]{dddd} \\
     & \sigma(u_{2,1}) \arrow[dash, dashed, bend right, red]{ddddrr}
             &
             & \sigma(u_{2,2}) \arrow[dash, dashed, bend left, thick, cyan]{ddd}
             &
             & \sigma(u_{2,3}) \arrow[dash, dashed, bend left, red]{ddddll} \\    
     \sigma(\hat t_{3_1}) \arrow[dash]{ur} \arrow[dash, bend left]{uurrr}
             &
             & \sigma(u_{3,2}) \arrow[dash]{ur} \arrow[dash]{ul}
             &
             & \sigma(u_{3,3}) \arrow[dash]{ur} \arrow[dash]{ul}
             &
             & \sigma(u_{3,4}) \arrow[dash]{ul} \arrow[dash, bend right]{uulll}
             \\
             \\[-1em]
                 & & & {\color{cyan}\sigma(w_2)}& & &
             \\[-0.5em]
                 & & & {\color{red}\sigma(w_1)} & & &
             \\[-1em]
             \\[-0.5em]
     & & & {\color{red}w_1}& & &
             \\[-0.5em]
     & & & {\color{cyan}w_2} & & &
             \\[-1em]
  u_{3,1} \arrow[equal]{uuuuuuu} 
            &
            & u_{3,2} \arrow[equal]{uuuuuuu}
             &
             & u_{3,3} \arrow[equal]{uuuuuuu}
             & 
             & u_{3,4} \arrow[equal]{uuuuuuu}\\
    & u_{2,1} \arrow[dash]{ul}{} \arrow[dash]{ur}
                \arrow[dash, dashed, bend left, red]{uuurr}
             &
             & u_{2,2} \arrow[dash]{ul} \arrow[dash]{ur}
                        \arrow[dash, dashed, bend left, thick, cyan]{uu}
             &
             &  u_{2,3} \arrow[dash]{ul} \arrow[dash]{ur} 
                        \arrow[dash, dashed, bend right, red]{uuull}\\
    & 
     & & u_{1,1} \arrow[dash,bend left]{uulll} 
                    \arrow[dash,swap,bend right]{uurrr}
                    \arrow[dash, dashed,bend right, thick, cyan]{uuu}
 \end{tikzcd}
\]
Since we're assuming $\gamma$ is embedded, both $w_1$ and $w_2$ cannot be type $\hat t_1$ or type $\hat t_2$; otherwise, they would necessarily be equal to one of the $u_{i,j}$ and thus $\gamma$ would not have been embedded to start with. This consequently means $\sigma(w_1)$ and $\sigma(w_2)$ are both not type $\hat t_5$ or type $\hat t_4$. 

Now note that $u_{2,1}$ and $u_{2,3}$ are both lower bounds of the set $\{\sigma(u_{1,1}), \sigma(u_{2,2})\}$. Since $\sigma(w_2)$ is the greatest lower bound of this set, $u_{2,1} \leq \sigma(w_2)$ and $u_{2,3} \leq \sigma(w_2)$. But then $\sigma(w_2)$ is an upper bound for the set $\{u_{2,1},u_{2,3}\}|$. Since $w_1$ is the least upper bound of this set, $w_1 \leq \sigma(w_2)$. An identical argument shows $w_2 \leq \sigma(w_1)$. This means the Hasse diagram is actually as follows:
\[
\begin{tikzcd}[row sep=1em, column sep=0.5em]
    & 
        & & 
            \sigma(u_{1,1}) \arrow[dash, dashed,bend right, thick, cyan]{dddd} \\
    & \sigma(u_{2,1}) \arrow[dash, dashed, bend right, red]{ddddrr}
            &
            & \sigma(u_{2,2}) \arrow[dash, dashed, bend left, thick, cyan]{ddd}
            &
            & \sigma(u_{2,3}) \arrow[dash, dashed, bend left, red]{ddddll} \\    
    \sigma(\hat t_{3_1}) \arrow[dash]{ur} \arrow[dash, bend left]{uurrr}
            &
            & \sigma(u_{3,2}) \arrow[dash]{ur} \arrow[dash]{ul}
            &
            & \sigma(u_{3,3}) \arrow[dash]{ur} \arrow[dash]{ul}
            &
            & \sigma(u_{3,4}) \arrow[dash]{ul} \arrow[dash, bend right]{uulll}
            \\
            \\[-1em]
                & & & {\color{cyan}\sigma(w_2)} \arrow[dash, thick, blue, in=30,out=-40]{ddd}& & &
            \\[-0.5em]
                & & & {\color{red}\sigma(w_1)} 
                    \arrow[dash, blue, thick, in=150,out=220]{ddd}& & &
            \\[-1em]
            \\[-0.5em]
    & & & {\color{red}w_1}& & &
            \\[-0.5em]
    & & & {\color{cyan}w_2} & & &
            \\[-1em]
 u_{3,1} \arrow[equal]{uuuuuuu} 
            &
            & u_{3,2} \arrow[equal]{uuuuuuu}
            &
            & u_{3,3} \arrow[equal]{uuuuuuu}
            & 
            & u_{3,4} \arrow[equal]{uuuuuuu}\\
    & u_{2,1} \arrow[dash]{ul}{} \arrow[dash]{ur}
               \arrow[dash, dashed, bend left, red]{uuurr}
            &
            & u_{2,2} \arrow[dash]{ul} \arrow[dash]{ur}
                       \arrow[dash, dashed, bend left, thick, cyan]{uu}
            &
            &  u_{2,3} \arrow[dash]{ul} \arrow[dash]{ur} 
                       \arrow[dash, dashed, bend right, red]{uuull}\\
    & 
    & & u_{1,1} \arrow[dash,bend left]{uulll} 
                   \arrow[dash,swap,bend right]{uurrr}
                   \arrow[dash, dashed,bend right, thick, cyan]{uuu}
\end{tikzcd}
\]
Since $w_1$ is not type $\hat t_1$ or $\hat t_2$, and $\sigma(w_2)$ is not type $\hat t_4$ or $\hat t_5$, it follows that in fact $w_1$ and $\sigma(w_2)$ are both type $\hat t_3$, and so $w_1 = \sigma(w_2)$. Similarly, $w_2 = \sigma(w_1)$ and this vertex is type $\hat t_3$. 

Consider the vertices $w_1$, $u_{3,2}$, and $u_{3,3}$. They are type $\hat t_3$ vertices which are pairwise lower bounded by vertices of type $\hat t_2$. By Theorem \ref{thm:jingyinlemma}, we can find a  vertex $z$ which is a lower bound of all three vertices. Since $\gamma$ is embedded, $z$ cannot be type $\hat t_3$ or $\hat t_2$, so it must be type $\hat t_1$. Adding $z$ and $\sigma(z)$ to the diagram gives
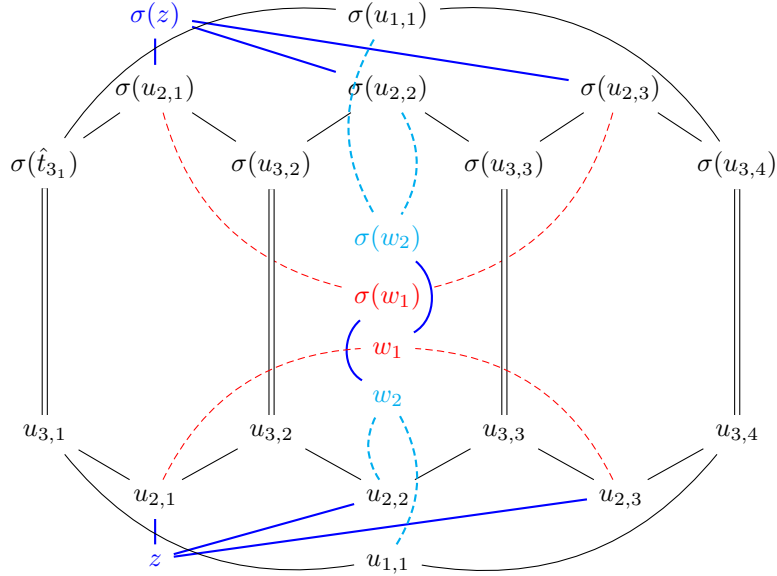
\[
\begin{tikzcd}[row sep=1em, column sep=0.5em]
  & {\color{blue}\sigma(z)} \arrow[dash, thick, blue]{d}
                                \arrow[dash, thick, blue]{drr}
                                \arrow[dash, thick, blue]{drrrr}
        & & 
            \sigma(u_{1,1}) \arrow[dash, dashed,bend right, thick, cyan]{dddd} \\
    & \sigma(u_{2,1}) \arrow[dash, dashed, bend right, red]{ddddrr}
            &
            & \sigma(u_{2,2}) \arrow[dash, dashed, bend left, thick, cyan]{ddd}
            &
            & \sigma(u_{2,3}) \arrow[dash, dashed, bend left, red]{ddddll} \\    
    \sigma(\hat t_{3_1}) \arrow[dash]{ur} \arrow[dash, bend left]{uurrr}
            &
            & \sigma(u_{3,2}) \arrow[dash]{ur} \arrow[dash]{ul}
            &
            & \sigma(u_{3,3}) \arrow[dash]{ur} \arrow[dash]{ul}
            &
            & \sigma(u_{3,4}) \arrow[dash]{ul} \arrow[dash, bend right]{uulll}
            \\
            \\[-1em]
                & & & {\color{cyan}\sigma(w_2)} \arrow[dash, thick, blue, in=30,out=-40]{ddd}& & &
            \\[-0.5em]
                & & & {\color{red}\sigma(w_1)} 
                    \arrow[dash, blue, thick, in=150,out=220]{ddd}& & &
            \\[-1em]
            \\[-0.5em]
    & & & {\color{red}w_1}& & &
            \\[-0.5em]
    & & & {\color{cyan}w_2} & & &
            \\[-1em]
 u_{3,1} \arrow[equal]{uuuuuuu} 
          &
          & u_{3,2} \arrow[equal]{uuuuuuu}
            &
            & u_{3,3} \arrow[equal]{uuuuuuu}
            & 
            & u_{3,4} \arrow[equal]{uuuuuuu}\\
  & u_{2,1} \arrow[dash]{ul}{} \arrow[dash]{ur}
               \arrow[dash, dashed, bend left, red]{uuurr}
            &
            & u_{2,2} \arrow[dash]{ul} \arrow[dash]{ur}
                       \arrow[dash, dashed, bend left, thick, cyan]{uu}
            &
            &  u_{2,3} \arrow[dash]{ul} \arrow[dash]{ur} 
                       \arrow[dash, dashed, bend right, red]{uuull}\\
  & {\color{blue}z} \arrow[dash, thick, blue]{u}
                                \arrow[dash, thick, blue]{urr}
                                \arrow[dash, thick, blue]{urrrr}
    & & u_{1,1} \arrow[dash,bend left]{uulll} 
                   \arrow[dash,swap,bend right]{uurrr}
                   \arrow[dash, dashed,bend right, thick, cyan]{uuu}
\end{tikzcd}
\]
And now we see that $z \leq \sigma(z)$, so by Lemma \ref{cor:lessiffimage}, $z$ is in the image of $\psi$, say $z = \psi(z')$ for some vertex $z'$ of type $\hat s_1$. So, we have a subcomplex of $D(B_3)$ of the form 
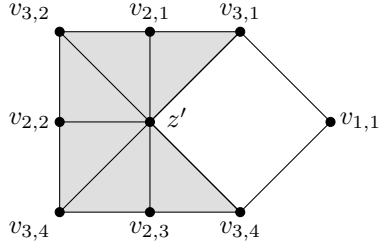
\[
\begin{tikzpicture}[scale=1.2]
    \coordinate (s31) at (0,0) ;
    \coordinate (s21) at (1,0) ;
    \coordinate (s32) at (2,0);
    \coordinate (s1) at (3,-1) ;
    \coordinate (s33) at (2,-2) ;
    \coordinate (s22) at (1,-2) ;
    \coordinate (s34) at (0,-2) ;
    \coordinate (s23) at (0,-1) ;

    \coordinate (s11) at (1,-1) ;
    \coordinate (s24) at (2,-1) ;
    
    \draw[fill=gray!25] (s31) -- (s21)  -- (s32) -- (s11) --(s33) -- (s22) -- (s34)  -- (s23) -- (s31);
    
    \filldraw (s31) circle (0.05cm);
    \filldraw (s21) circle (0.05cm);
    \filldraw (s32) circle (0.05cm);
    \filldraw (s1) circle (0.05cm);
    \filldraw (s33) circle (0.05cm);
    \filldraw (s22) circle (0.05cm);
    \filldraw (s34) circle (0.05cm);
    \filldraw (s23) circle (0.05cm);
    \filldraw (s11) circle (0.05cm);

    \draw (s31) -- (s33);
    \draw (s32) -- (s34);
    \draw (s21) -- (s22);
    \draw (s23) -- (s11);
    \draw (s32) -- (s1);
    \draw (s33) -- (s1);

    \node[above left] at (s31) {$v_{3,2}$} ;
    \node[above] at (s21) {$v_{2,1}$} ;
    \node[above] at (s32) {$v_{3,1}$} ;
    \node[right] at (s1)  {$v_{1,1}$} ;
    \node[below] at (s33) {$v_{3,4}$} ;
    \node[below] at (s22) {$v_{2,3}$} ;
    \node[below left] at (s34) {$v_{3,4}$} ;
    \node[left] at (s23) {$v_{2,2}$} ;
    \node at (1.3,-0.95) {$z'$};
    \end{tikzpicture}
\]
By Lemma \ref{lemma:filling 4 cycles in DB3}, the 4-cycle made of $z$, $v_{1,1}$, $v_{3,1}$, and $v_{3,4}$ can be filled with a central vertex $v_{2,4}$ of type $\hat s_2$, giving the subcomplex 
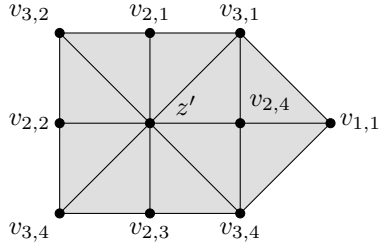
\[
\begin{tikzpicture}[scale=1.2]
        \coordinate (s31) at (0,0) ;
        \coordinate (s21) at (1,0) ;
        \coordinate (s32) at (2,0);
        \coordinate (s1) at (3,-1) ;
        \coordinate (s33) at (2,-2) ;
        \coordinate (s22) at (1,-2) ;
        \coordinate (s34) at (0,-2) ;
        \coordinate (s23) at (0,-1) ;

        \coordinate (s11) at (1,-1) ;
        \coordinate (s24) at (2,-1) ;
        
        \draw[fill=gray!25] (s31) -- (s21)  -- (s32) -- (s1) --(s33) -- (s22) -- (s34)  -- (s23) -- (s31);
        
        \filldraw (s31) circle (0.05cm);
        \filldraw (s21) circle (0.05cm);
        \filldraw (s32) circle (0.05cm);
        \filldraw (s1) circle (0.05cm);
        \filldraw (s33) circle (0.05cm);
        \filldraw (s22) circle (0.05cm);
        \filldraw (s34) circle (0.05cm);
        \filldraw (s23) circle (0.05cm);
        \filldraw (s24) circle (0.05cm);
        \filldraw (s11) circle (0.05cm);

        \draw (s31) -- (s33);
        \draw (s32) -- (s34);
        \draw (s32) -- (s33);
        \draw (s21) -- (s22);
        \draw (s23) -- (s1);

        \node[above left] at (s31) {$v_{3,2}$} ;
    \node[above] at (s21) {$v_{2,1}$} ;
    \node[above] at (s32) {$v_{3,1}$} ;
    \node[right] at (s1)  {$v_{1,1}$} ;
    \node[below] at (s33) {$v_{3,4}$} ;
    \node[below] at (s22) {$v_{2,3}$} ;
    \node[below left] at (s34) {$v_{3,4}$} ;
    \node[left] at (s23) {$v_{2,2}$} ;
        \node[above right] at (s24) {$v_{2,4}$} ;
    \node at (1.4,-0.8) {$z'$};
        \end{tikzpicture}
\]
as claimed.
\end{proof}

\subsection{The 10-cycle} 
Last we approach the 10-cycle. In general, a cycle of this length is very difficult to deal with, and progress was only possible with the advent of new methods by Huang \cite{huang2025cycles}.

For the rest of this section, let $\tilde \gamma$ a cycle of the type in Figure \ref{fig:bad 10 cycle}. In order to show $\tilde \gamma$ is not embedded, our analysis must be broken into two cases, which we describe now. 
We begin by briefly recalling material defined in \cite[\S 4]{huang2025cycles},  rephrasing the definitions in a manner more immediately useful to us (one may readily verify that they are equivalent).

\begin{defn}
    Let $\Gamma$ be a spherical-type Coxeter-Dynkin diagram.
    Denote by $\rho : D(\Gamma) \to C(\Gamma)$ the quotient map induced by the natural quotient $A(\Gamma) \to W(\Gamma)$. Let $s \in V(\Gamma)$ be a standard generator of $W(\Gamma)$, and let $r$ be any conjugate of $s$.
    (In other words, $r$ is a reflection in $W(\Gamma)$.)  
    Let $H$ be the subcomplex of $C(\Gamma)$ fixed (pointwise) by $r$. Then $C(\Gamma) \setminus H$ has two connected components, permuted isomorphically by $r$. Choose one of these open components (called a \emph{halfspace}), and let $U$ be the largest (closed) subcomplex of $C(\Gamma)$ contained in this component.
    We call any lift of $U$ along $\rho$ a \emph{Falk subcomplex of type $s$} in $D(\Gamma)$. 
\end{defn}

\begin{rem}
    Different choices of lifts of $U$ will give rise to distinct, but isomorphic, choices of Falk subcomplex. For our purposes, the specific choice of Falk subcomplex will be less important than its type. In general, two Falk subcomplexes of the same type will be isomorphic, hence the ``type'' of a Falk subcomplex is well-defined. To this end, we also emphasize that if two generators $t$ and $s$ are conjugate in $W(\Gamma)$, then a Falk subcomplex of type $t$ is isomorphic to a Falk subcomplex of type $s$. 
\end{rem}

We now return to the specific case of $B_3$. We have the following result regarding Falk subcomplexes in $D(B_3)$:

\begin{prop}
    Suppose $D'$ is a Falk subcomplex of type $s_3$ in $D(B_3)$. Then the ``orthoscheme metric'' on $D'$ is $\cat(0)$. 
\end{prop}

\begin{proof}
    The orthoscheme metric on $D'$ is a piecewise Euclidean metric which, for all 2-simplices, assigns an angle of $\pi/2$ to vertices of type $\hat s_2$, and an angle of $\pi/4$ to vertices of both type $\hat s_1$ and $\hat s_3$. (In other words, every 2-simplex is a Euclidean isosceles right triangle.)
    After unraveling definitions, \cite[Prop.~A.1]{huang2025cycles} implies 
    \begin{enumerate}
        \item The link of a vertex of type $\hat s_2$ in $D'$ is complete bipartite (in particular, has girth 4), and
        \item The link of a vertex of type $\hat s_1$ or $\hat s_3$ in $D'$ has girth at least $8$. 
    \end{enumerate}
    Hence this metric is $\cat(0)$.
\end{proof}

Since these subcomplexes carry a nice $\cat(0)$ metric, we will call any Falk subcomplex of type $\hat s_3$ a \emph{good Falk subcomplex}. Our first case in the analysis of 10-cycles is the following.

\begin{lem} \label{lem:good falk 10}
    If $\tilde \gamma$ is contained in a good Falk subcomplex, then it is not embedded. 
\end{lem}

\begin{proof}
    Suppose $D'$ is a good Falk subcomplex containing $\tilde\gamma$. If $\tilde\gamma$ is locally geodesic under the orthoscheme metric, then since this metric is $\cat(0)$, we would have that $\tilde\gamma$ is a global geodesic, and hence not closed. So there must be some point at which $\tilde \gamma$ is not locally geodesic in the orthoscheme metric. It is clear that interiors of edges are always geodesic.
    It follows from the fact that the link of a $\hat s_2$ vertex is complete bipartite that $\tilde\gamma$ is always locally geodesic at vertices of type $\hat s_2$.
    Hence there must be some vertex $v$ of type $\hat s_3$ at which $\gamma$ fails to be locally geodesic. 
    Let $v_{-2}, v_{-1}, v_0, v_1,v_2$ be a sequence of adjacent (distinct) vertices with $v_0 = v$, and let $\delta$ denote the subpath of $\tilde \gamma$ from $v_{-2}$ to $v_2$ passing through $v_0$.
    In the orthoscheme metric on $D'$, the link of $v$ is a bipartite graph with edge lengths all $\pi/4$. Hence the only way for $\tilde \gamma$ to fail to be locally geodesic at $v$ is for there to exist a vertex $w$ of type $\hat s_1$ adjacent to each of $v_{-1}$, $v_0$, and $v_1$ (see Figure \ref{fig:Falk1}).
\begin{figure}[!ht]
    \centering
    \scalebox{1}
{
\begin{tikzpicture}
[
    regular polygon colors/.style 2 args={
        append after command={%
            \pgfextra
    \foreach \j/\k\l [count=\ni, remember=\ni as \lasti (initially #1)] in 
            {
                solid/100/(\tikzlastnode.corner \lasti) --    
                    (\tikzlastnode.corner \ni), 
                solid/100/(\tikzlastnode.corner \lasti) --    
                    (\tikzlastnode.corner \ni), 
                /100/(\tikzlastnode.corner \lasti) edge 
                    ($(\tikzlastnode.corner \lasti)!0.45!(\tikzlastnode.corner \ni)$)
                    edge [dotted] 
                    ($(\tikzlastnode.corner \lasti)!0.75!(\tikzlastnode.corner \ni)$),
                /0/(\tikzlastnode.corner \lasti) --    
                    (\tikzlastnode.corner \ni), 
                /0/(\tikzlastnode.corner \lasti) --    
                    (\tikzlastnode.corner \ni), 
                /0/(\tikzlastnode.corner \lasti) --    
                    (\tikzlastnode.corner \ni), 
                /0/(\tikzlastnode.corner \lasti) --    
                    (\tikzlastnode.corner \ni), 
                dotted/100/
                    (\tikzlastnode.corner \ni) edge [solid]
                    ($(\tikzlastnode.corner \ni)!0.45!(\tikzlastnode.corner \lasti)$)
                    edge [dotted] 
                    ($(\tikzlastnode.corner \ni)!0.75!(\tikzlastnode.corner \lasti)$),
                solid/100/(\tikzlastnode.corner \lasti) --    
                    (\tikzlastnode.corner \ni), 
                solid/100/(\tikzlastnode.corner \lasti) --    
                    (\tikzlastnode.corner \ni)
            }
            {
            \draw[\j, opacity = \k] \l;
            } %
        \endpgfextra
    }
},
]
  \node
    [
    minimum size=4cm, regular polygon,
    regular polygon sides=10, 
    regular polygon colors={10}{}
    ] 
        (a) {};
        
    \fill (a.corner 1) circle[radius=1pt] node[above] (p1) {$v_{-1}$};
    \fill (a.corner 2) circle[radius=1pt] node[above] (p2) {$v_{-2}$};
    \fill (a.corner 8) circle[radius=1pt] node[below right] (p8) {$v_{2}$};
    \fill (a.corner 9) circle[radius=1pt] node[right] (p9) {$v_{1}$};
    \fill (a.corner 10) circle[radius=1pt] node[above right] (p10) {$v$};

    \fill[color=red] (a.center) circle[radius=1pt] node[below left] (ac) {$w$};

    \fill [fill=gray!30] (a.center)  -- (a.corner 9)  -- (a.corner 10) -- (a.corner 1);

    \draw[color=red] (a.center) -- (a.corner 1);
    \draw[color=red] (a.center) -- (a.corner 9);
    \draw[color=red] (a.center) -- (a.corner 10);

\end{tikzpicture}
}
    \caption{Adding a vertex of type $\hat s_1$}
    \label{fig:Falk1}
\end{figure}
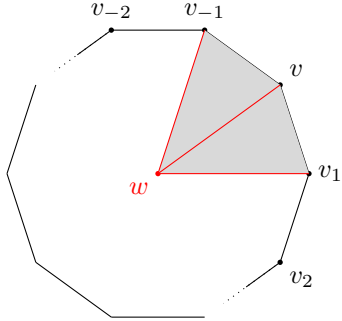
    Since $v_{-1}$ and $v_1$ are type $\hat s_2$, their links are also complete bipartite, so $w$ must also be adjacent to $v_{-2}$ and $v_{2}$ (see Figure \ref{fig:Falk2}). 
    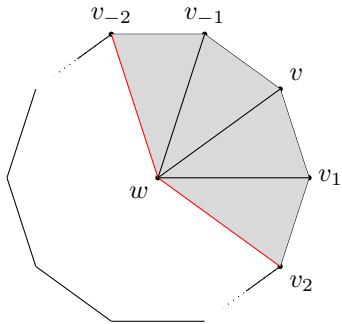
\begin{figure}[!ht]
        \centering
        \scalebox{1}
{
\begin{tikzpicture}
[
    regular polygon colors/.style 2 args={
        append after command={%
            \pgfextra
    \foreach \j/\k\l [count=\ni, remember=\ni as \lasti (initially #1)] in 
            {
                solid/100/(\tikzlastnode.corner \lasti) --    
                    (\tikzlastnode.corner \ni), 
                solid/100/(\tikzlastnode.corner \lasti) --    
                    (\tikzlastnode.corner \ni), 
                /100/(\tikzlastnode.corner \lasti) edge 
                    ($(\tikzlastnode.corner \lasti)!0.45!(\tikzlastnode.corner \ni)$)
                    edge [dotted] 
                    ($(\tikzlastnode.corner \lasti)!0.75!(\tikzlastnode.corner \ni)$),
                /0/(\tikzlastnode.corner \lasti) --    
                    (\tikzlastnode.corner \ni), 
                /0/(\tikzlastnode.corner \lasti) --    
                    (\tikzlastnode.corner \ni), 
                /0/(\tikzlastnode.corner \lasti) --    
                    (\tikzlastnode.corner \ni), 
                /0/(\tikzlastnode.corner \lasti) --    
                    (\tikzlastnode.corner \ni), 
                dotted/100/
                    (\tikzlastnode.corner \ni) edge [solid]
                    ($(\tikzlastnode.corner \ni)!0.45!(\tikzlastnode.corner \lasti)$)
                    edge [dotted] 
                    ($(\tikzlastnode.corner \ni)!0.75!(\tikzlastnode.corner \lasti)$),
                solid/100/(\tikzlastnode.corner \lasti) --    
                    (\tikzlastnode.corner \ni), 
                solid/100/(\tikzlastnode.corner \lasti) --    
                    (\tikzlastnode.corner \ni)
            }
            {
            \draw[\j, opacity = \k] \l;
            } %
        \endpgfextra
    }
},
]
  \node
    [
    minimum size=4cm, regular polygon,
    regular polygon sides=10, 
    regular polygon colors={10}{}
    ] 
        (a) {};
        
    \fill (a.corner 1) circle[radius=1pt] node[above] (p1) {$v_{-1}$};
    \fill (a.corner 2) circle[radius=1pt] node[above] (p2) {$v_{-2}$};
    \fill (a.corner 8) circle[radius=1pt] node[below right] (p8) {$v_{2}$};
    \fill (a.corner 9) circle[radius=1pt] node[right] (p9) {$v_{1}$};
    \fill (a.corner 10) circle[radius=1pt] node[above right] (p10) {$v$};

    \fill[] (a.center) circle[radius=1pt] node[below left] (ac) {$w$};

    \fill [fill=gray!30] (a.center) -- (a.corner 8) -- (a.corner 9)  
            -- (a.corner 10) -- (a.corner 1)  -- (a.corner 2);

    \draw[] (a.center) -- (a.corner 1);
    \draw[color=red] (a.center) -- (a.corner 2);
    \draw[color=red] (a.center) -- (a.corner 8);
    \draw[] (a.center) -- (a.corner 9);
    \draw[] (a.center) -- (a.corner 10);

\end{tikzpicture}
}
        \caption{Connecting the new vertex to the adjacent type $\hat s_2$ vertices}
        \label{fig:Falk2}
    \end{figure}
    Let $\delta'$ denote the concatenation of the segments between $v_{-2}$ and $w$, and between $w$ and $v_2$. 
    Let $\tilde\gamma'$ be the cycle obtained from $\tilde\gamma$ by replacing $\delta$ with $\delta'$. Then $\tilde\gamma'$ is an 8-cycle in $D(B_3)$
    of the form seen in Figure \ref{fig:checking edge loops}(c), thus can be filled to the following subcomplex:
   
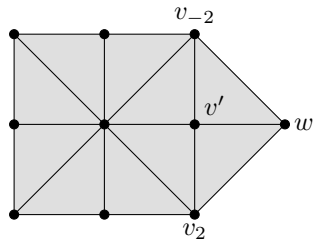
\[
    \begin{tikzpicture}[scale=1.2]
        \coordinate (s31) at (0,0) ;
        \coordinate (s21) at (1,0) ;
        \coordinate (s32) at (2,0);
        \coordinate (s1) at (3,-1) ;
        \coordinate (s33) at (2,-2) ;
        \coordinate (s22) at (1,-2) ;
        \coordinate (s34) at (0,-2) ;
        \coordinate (s23) at (0,-1) ;

        \coordinate (s11) at (1,-1) ;
        \coordinate (s24) at (2,-1) ;
        
        \draw[fill=gray!25] (s31) -- (s21)  -- (s32) -- (s1) --(s33) -- (s22) -- (s34)  -- (s23) -- (s31);
        
        \filldraw (s31) circle (0.05cm);
        \filldraw (s21) circle (0.05cm);
        \filldraw (s32) circle (0.05cm);
        \filldraw (s1) circle (0.05cm);
        \filldraw (s33) circle (0.05cm);
        \filldraw (s22) circle (0.05cm);
        \filldraw (s34) circle (0.05cm);
        \filldraw (s23) circle (0.05cm);
        \filldraw (s24) circle (0.05cm);
        \filldraw (s11) circle (0.05cm);

        \draw (s31) -- (s33);
        \draw (s32) -- (s34);
        \draw (s32) -- (s33);
        \draw (s21) -- (s22);
        \draw (s23) -- (s1);

    \node[above] at (s32) {$v_{-2}$} ;
    \node[right] at (s1)  {$w$} ;
    \node[below] at (s33) {$v_{2}$} ;
        \node[above right] at (s24) {$v'$} ;
    \end{tikzpicture}
    \]
    Let $v'$ be the new type $\hat s_2$ vertex adjacent to each of $v_{-2}$, $w$, and $v_2$.
    The vertices $v'$, $v_{-2}$, $v_{-1}, v, v_1$, and $v_2$ give rise to a 6-cycle in the link of $w$. However, $w$ is type $\hat s_1$, and the link of such a vertex has girth 8. Hence this 6-cycle cannot be embedded. Regardless of which vertices are identified in this cycle, it results in $\tilde\gamma$ not being embedded in $D(B_3)$, as well.
    
\end{proof}

It remains to consider those 10-cycles which do not remain in a good Falk subcomplex. 
In order to complete this argument, we must utilize the ``pure Salvetti complex'' of $A(\Gamma)$. As to not get bogged down in details, we will leave the full definition to \cite[\S 4.2]{huang2025cycles}; however, here, we adopt the notation $\widehat \Sigma(\Gamma)$ (rather than $\widehat \Sigma_{\mathcal A}$) for consistency. The following are facts regarding $\widehat \Sigma(\Gamma)$ which will be relevant for us:
\begin{enumerate}
    \item $\pi_1(\widehat\Sigma(\Gamma))$ is the kernel of the quotient homomorphism $A(\Gamma) \to W(\Gamma)$ (called the \emph{pure Artin group}, denoted $PA(\Gamma)$)
    \item There is a projection map $p = p_\Gamma : \widehat\Sigma(\Gamma) \to \Sigma(\Gamma)$ which is a bijection of vertex sets.
    \item If $F$ is a face of $\Sigma(\Gamma)$ and $\widehat F = p^{-1}(F)$, then there is a retraction $\Proj_{\widehat F} : \widehat\Sigma(\Gamma) \to \widehat F$ which restricts to the ``closest point'' projection on the 1-skeleton (under the path metric). In particular, the inclusion $\widehat F \hookrightarrow \widehat\Sigma(\Gamma)$ is $\pi_1$-injective.
\end{enumerate}
We can give a more explicit (and standard) description of (1). Recall that the 1-skeleton of $\Sigma(\Gamma)$ is the (unoriented) Cayley graph of $W(\Gamma)$ with respect to the standard generating set, so we may label edges by generators. 
We can then pull back this labeling along the projection $p : \widehat\Sigma(\Gamma) \to \Sigma(\Gamma)$ so that all edges of $\widehat\Sigma(\Gamma)$ are labeled by a generator. 
For each edge $e$ of $\Sigma(\Gamma)$, $p^{-1}(e)$ consists of exactly two edges (and two vertices) of $\widehat\Sigma(\Gamma)$. We assign these edges ``opposite'' orientations. Now we fix a base point $x \in V(\widehat\Sigma(\Gamma))$ and choose a word $w = s_1^{k_1}\dots s_n^{k_n}$ in $PA(\Gamma)$. If $k_1 > 0$, then we loop around the edges adjacent to $x$ labeled $s_1$ in the positive orientation $k_1$ times. If $k_1 < 0$, then we do the same but in the negative orientation. The end of this path may or may not be $x$, but regardless, we repeat this with $s_2$, and so on. This procedure will give a closed loop in $\widehat\Sigma(\Gamma)$ representing $w$ in $\pi_1(\widehat\Sigma(\Gamma))$. In particular, if $w$ is trivial, then this loop will be nullhomotopic.

We can now return to our analysis of $\tilde\gamma$. 
Our goal is to construct a word in $PA(\Gamma)$ using $\tilde\gamma$, then pass to a loop in $\widehat\Sigma(B_3)$ using the above construction in order to determine more information about $\tilde\gamma$.
Since the vertices of type $\hat s_2$ are complete bipartite, we can fill $\tilde \gamma$ into the following diagram:
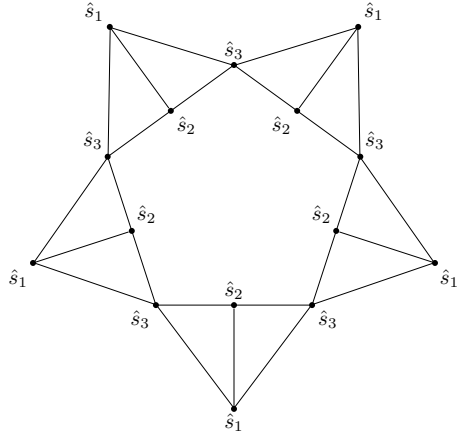
\begin{center}
\scalebox{0.8}{
\begin{tikzpicture}
   \node [star, star point height=1.3cm, minimum size=7cm, draw, rotate=36]
       at (0,0) (A) {};

    \foreach \i/\j in {1/2,2/3,3/4,4/5,5/1}
        \draw[] (A.inner point \i) to (A.inner point \j)
        node [circle, label={[label distance=-1mm]160+72*(\i-1):$\hat s_3$}, fill=black, inner sep=1pt] 
            at (A.inner point \i) {}
        node [circle, label={[label distance=-1mm]125+72*(\i-1):$\hat s_1$}, fill=black, inner sep=1pt] 
            at (A.outer point \i) {}
        node [circle, label={[label distance=-1mm]20+72*(\i-1):$\hat s_2$}, fill=black, inner sep=1pt]
            at ($(A.inner point \i)!0.5!(A.inner point \j)$) (B \i) {}
        ;
    \foreach \r/\s in {1/5,2/1,3/2,4/3,5/4}
        \draw[] (A.outer point \r) to (B \s);

\end{tikzpicture}
}
\end{center}
Just as an edge is labeled by the intersection of its vertices, we can label a 2-simplex by the intersection of its vertices. This will be a coset of the trivial subgroup, i.e., a single group element, so we may as well label a 2-simplex solely by this element. 
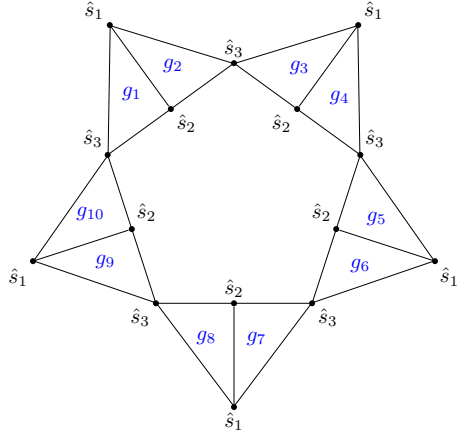
\begin{center}
\scalebox{0.8}{
\begin{tikzpicture}
   \node [star, star point height=1.3cm, minimum size=7cm, draw, rotate=36]
       at (0,0) (A) {};

    \foreach \i/\j in {1/2,2/3,3/4,4/5,5/1}
        \draw[] (A.inner point \i) to (A.inner point \j)
        node [circle, label={[label distance=-1mm]160+72*(\i-1):$\hat s_3$}, fill=black, inner sep=1pt] 
            at (A.inner point \i) {}
        node [circle, label={[label distance=-1mm]125+72*(\i-1):$\hat s_1$}, fill=black, inner sep=1pt] 
            at (A.outer point \i) {}
        node [circle, label={[label distance=-1mm]20+72*(\i-1):$\hat s_2$}, fill=black, inner sep=1pt]
            at ($(A.inner point \i)!0.5!(A.inner point \j)$) (B \i) {}
        ;
    \foreach \r/\s in {1/5,2/1,3/2,4/3,5/4}
        \draw[] (A.outer point \r) to (B \s);

    \foreach \t/\u in {1/2,3/4,5/6,7/8,9/{10}}
        \node[color=blue] at ({135-(\t -1)*36}:2.4cm) (C \t) {$g_\t$}
        node[color=blue] at ({115-(\u -2)*36}:2.4cm) (C \u) {$g_{\u}$}
        ;
\end{tikzpicture}
}
\end{center}
Suppose we have adjacent 2-simplices $\sigma_1 = \{h_1\}$ and $\sigma_2 = \{h_2\}$ with $\alpha$ the edge or vertex at which they meet. Then there is an element $a \in \alpha$ (recall $\alpha$ is a coset of a standard parabolic subgroup) such that $\sigma_1 = a \sigma_2$, or in the group, $h_1 = a h_2$. 
Returning to the diagram and utilizing this observation, we see there is some $a_1 \in A(s_3)$ such that $g_2 = a_1 g_1$. Of course, since $a_1 \in A(s_3)$, we may write $a_1 = s_3^{n_1}$ for some $n_1 \in \mathbb{Z}$. 
Similarly, there is some $a_2 \in A(\hat s_3)$ such that $g_3 = a_2 g_2$. Repeating, we have $g_{2k} = a_{2k} g_{2k-1}$ where $a_{2k} = s_3^{n_{2k}}$, and 
$g_{2k+1} = a_{2k+1} g_{2k}$ where $a_{2k+1} \in A(\hat s_3)$. 
But importantly, since this is a cycle, we also have that $g_1 = a_{10} g_{10}$ where $a_{10} \in  A(\hat s_3)$.
Combining these equations gives us the trivial word
\[
    w = a_1a_2\dots a_{10}. 
\] 
In particular, $w \in PA(\Gamma)$, so there is a loop $\gamma$ representing $w$ in $\widehat\Sigma(\Gamma)$ as constructed previously (and since $w$ is trivial, $\gamma$ is nullhomotopic).

\begin{defn}
    Let $E$ be an edge of $\Sigma(B_3)$ labeled $s$, for $s = s_1, s_2$, or $s_3$. We call $\hat E \coloneqq p^{-1}(E)$ an \emph{$s$ loop}. 
    Let $F$ be a 2-cell of $\Sigma(B_3)$. Then $F$ is either a square, hexagon, or octagon, and we will call $\hat F \coloneqq p^{-1}(F)$ a \emph{square, hexagon, or octagon} of $\widehat\Sigma(B_3)$, respectively. We will call $\hat E$ and $\hat F$ \emph{special subcomplexes} of $\widehat\Sigma(B_3)$.
\end{defn}

Suppose $\hat F$ and $\hat G$ are special subcomplexes of $\widehat\Sigma(B_3)$. We will say that $\gamma$ \emph{leaves $\hat F$ and enters $\hat G$} at time $t$ if for all small $\varepsilon > 0$, we have that 
\begin{enumerate}
    \item $\gamma(t) \in \hat F \cap \hat G$,
    \item $\gamma(t + \varepsilon) \in \hat G$, and
    \item $\gamma(t + \varepsilon) \not\in \hat F$.
\end{enumerate}
A subword $a_i \in A(\hat s_3)$ of $w$ gives rise to a time $t$ where $\gamma$ leaves a square and enters a hexagon. 
A subword $a_i \in A(s_3)$ gives rise to a time $t$ where $\gamma$ leaves a hexagon and enters a square; more specifically, it enters an $s_3$ loop.
Notice that this means $\gamma$ must enter hexagons 5 times and must enter squares 5 times.

The next lemma follows immediately from unraveling the full definitions found in \cite{huang2025cycles}.
\begin{lem} 
    Let $U$ denote the subcomplex of $\Sigma(B_3)$ seen in Figure \ref{fig:good Falk in Davis}, and let $\widehat U = p^{-1}(U) \subseteq \widehat\Sigma(B_3)$.
    Then $\tilde\gamma$ stays in a good Falk subcomplex if and only if $\gamma$ stays in a subcomplex of $\widehat\Sigma(B_3)$ isomorphic to $\widehat U$. 
\end{lem}

\begin{figure}
    \centering
    \tikzset{
      A/.style={regular polygon, regular polygon sides=8,fill=gray!#1,minimum size=2cm, draw},
      B/.style={regular polygon, regular polygon sides=4,fill=gray!#1,minimum size=1.1cm, draw},
    }
    
    \scalebox{1.5}
    {
    \begin{tikzpicture}[node distance=0pt, every node/.style={outer sep=0pt}]
      \node (C1) [A=0] {};
      \node (B1) [B=0, anchor=corner 3] at (C1.corner 2) {};
      \node (B2) [B=0, anchor=corner 1] at (C1.corner 3) {};
      \node (B3) [B=0, anchor=corner 2] at (C1.corner 5) {};
      \node (B4) [B=0, anchor=corner 2] at (C1.corner 8) {};
    
      \begin{scope}[rotate=60]
            \draw[] (B1.corner 2) -- ([shift={(75:1.1cm)}]B1.corner 3) -- 
                ([shift={(75:1.1cm)}]B2.corner 1) -- (B2.corner 2);
            \draw[] (B2.corner 3) -- ([shift={(165:1.1cm)}]B2.corner 4) -- 
                ([shift={(165:1.1cm)}]B3.corner 2) -- (B3.corner 3);
            \draw[] (B3.corner 4) -- ([shift={(255:1.1cm)}]B3.corner 1) -- 
                ([shift={(255:1.1cm)}]B4.corner 3) -- (B4.corner 4);
            \draw[] (B1.corner 1) -- ([shift={(345:1.1cm)}]B1.corner 4) -- 
                ([shift={(345:1.1cm)}]B4.corner 2) -- (B4.corner 1);
      \end{scope}
    
        \coordinate (d1) at 
            ($(B1.corner 1)!0.5 !(B1.corner 2)$) {};
        \coordinate (d3) at 
            ($(B2.corner 2)!0.5 !(B2.corner 3)$) {};
        \coordinate (d5) at 
            ($(B3.corner 3)!0.5 !(B3.corner 4)$) {};
        \coordinate (d7) at 
            ($(B4.corner 4)!0.5 !(B4.corner 1)$) {};
            
        \coordinate (d2) at (135:2.03cm) {};
        \coordinate (d4) at (225:2.03cm) {};
        \coordinate (d6) at (315:2.03cm) {};
        \coordinate (d8) at (45:2.03cm) {};

        \foreach \r/\s in {1/2,3/4,5/6,7/8}
             \coordinate[] (e\r) at ({115+(\r-1)*45}: 1.88cm) {}
              coordinate[] (e\s) at ({155+(\r-1)*45}: 1.88cm) {};
    
        \coordinate (f3) at ($(B2.corner 1) ! 0.5 ! (B2.corner 2)$) {};
        \coordinate (f4) at ($(B2.corner 3) ! 0.5 ! (B2.corner 4)$) {};
        \coordinate (f5) at ($(B3.corner 2) ! 0.5 ! (B3.corner 3)$) {};
        \coordinate (f6) at ($(B3.corner 1) ! 0.5 ! (B3.corner 4)$) {};
         \coordinate (f7) at ($(B4.corner 3) ! 0.5 ! (B4.corner 4)$) {};
         \coordinate (f8) at ($(B4.corner 1) ! 0.5 ! (B4.corner 2)$) {};
         \coordinate (f1) at ($(B1.corner 1) ! 0.5 ! (B1.corner 4)$) {};
         \coordinate (f2) at ($(B1.corner 2) ! 0.5 ! (B1.corner 3)$) {};

        \foreach \i/\j in {1/5,2/6,3/7,4/8}
            \draw[red] (d\i) to (d\j);
    
        \draw[red] (e1) to[curve through = {(f3) (f4)}] (e4);
        \draw[red] (e3) to[curve through = {(f5) (f6)}] (e6);
        \draw[red] (e5) to[curve through = {(f7) (f8)}] (e8);
        \draw[red] (e7) to[curve through = {(f1) (f2)}] (e2);
    
    \end{tikzpicture}
    }
        \caption{A portion of $\Sigma(B_3)$ corresponding to the good Falk subcomplex, with the Coxeter complex cell structure overlaid in red.}
        \label{fig:good Falk in Davis}
    \end{figure}
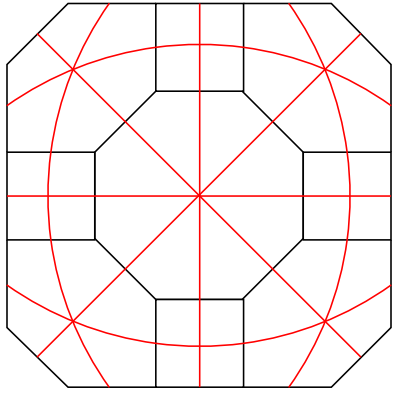
    
    From this, we can show that 10-cycles leaving a good Falk subcomplex must have a certain form in $\widehat\Sigma(B_3)$.
    
    \begin{lem} \label{Lemma: Good hyperplane}
        If $\tilde \gamma$ is not contained in a good Falk complex, then there is (at least) one square $S \subseteq \Sigma(B_3)$ such that 
        \begin{enumerate}
            \item $\gamma$ enters $\hat S$ exactly one time, and
            \item If $H$ is the hyperplane of $\Sigma(B_3)$ dual to the $s_3$ edges of $S$, and $\hat F \not= \hat S$ is a hexagon or square which $\gamma$ enters, then $p(\hat F) \cap H = \varnothing$.
        \end{enumerate}
    \end{lem}
    
\begin{figure}[!ht]
    \centering
    \tikzset{
      A/.style={regular polygon, regular polygon sides=6,fill=gray!#1,minimum size=2cm, draw},
      B/.style={regular polygon, regular polygon sides=4,fill=gray!#1,minimum size=1.41cm, draw},
    }
    
    \scalebox{0.75}
    {
    \begin{tikzpicture}[node distance=0pt, every node/.style={outer sep=0pt}]
      \node (C1) [A=0] {};
      \node (B1) [B=0, anchor=corner 3, rotate=30] at (C1.corner 6) {};
      \node (B2) [B=0, anchor=corner 3, rotate=60] at (C1.corner 3) {};
      \node (B3) [B=0, anchor=corner 2] at (C1.corner 4) {};
      
            \node at (0,-4.5) {(a)};    
    \end{tikzpicture}
    
    \qquad\quad
    
    \begin{tikzpicture}[node distance=0pt, every node/.style={outer sep=0pt}]
      \node (C1) [A=0] {};
      \node (B1) [B=0, anchor=corner 3, rotate=30] at (C1.corner 6) {};
      \node (B2) [B=0, anchor=corner 3, rotate=60] at (C1.corner 3) {};
      \node (B3) [B=0, anchor=corner 2] at (C1.corner 4) {};
      \node (C2) [A=0, anchor=corner 1] at (B3.corner 4) {};
      \node (B4) [B=0, anchor=corner 4, rotate=30] at (C2.corner 4) {};
      \node (B5) [B=0, anchor=corner 1, rotate=60] at (C2.corner 6) {};
            \node at (0,-5.5) {(b)};    
    \end{tikzpicture}
    
    \qquad\quad
    
    \begin{tikzpicture}[node distance=0pt, every node/.style={outer sep=0pt}]
      \node (C1) [A=0] {};
      \node (B1) [B=0, anchor=corner 3, rotate=30] at (C1.corner 6) {};
      \node (B2) [B=0, anchor=corner 3, rotate=60] at (C1.corner 3) {};
      \node (B3) [B=0, anchor=corner 2] at (C1.corner 4) {};
      \node (C2) [A=0, anchor=corner 1] at (B3.corner 4) {};
      \node (B4) [B=0, anchor=corner 4, rotate=30] at (C2.corner 4) {};
      \node (B5) [B=0, anchor=corner 1, rotate=60] at (C2.corner 6) {};
      \node (C3) [A=0, anchor=corner 2] at (B5.corner 4) {};
      \node (B6) [B=0, anchor=corner 1, rotate=0] at (C3.corner 5) {};
      \node (B7) [B=0, anchor=corner 2, rotate=30] at (C3.corner 1) {};
      
            \node at (0,-6.5) {(c)};    
    \end{tikzpicture}
    }
    
    \caption{Various subcomplexes of $\Sigma(B_3)$}
    \label{fig:Leaving Falk}
\end{figure}
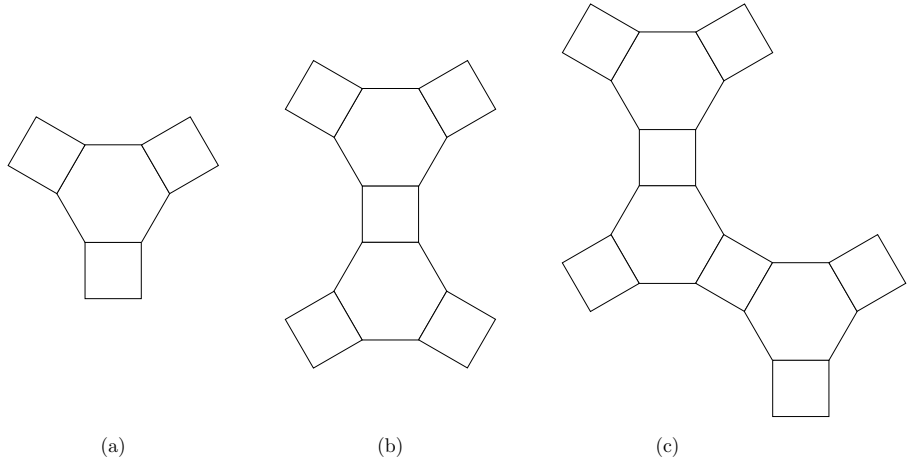
\begin{proof}
We can group the possibilities for $\gamma$ based on the number of distinct hexagons it enters. Throughout, let $U$ be a subcomplex of $\Sigma(B_3)$ of the form shown in Figure \ref{fig:good Falk in Davis}, and let $\hat U = p^{-1}(U)$.

Suppose $\gamma$ enters exactly one hexagon (possibly multiple times). Then $\gamma$ remains in a subcomplex of $\widehat\Sigma(B_3)$ of the form $p^{-1}(K)$, where $K$ has the shape seen in Figure \ref{fig:Leaving Falk}(a). If it enters one or two of the adjacent squares, then it is contained in some $\hat U$, and so $\tilde \gamma$ is contained in a good Falk complex.
If instead it enters three squares, then it must enter one of them exactly once, and the hyperplane dual to this square's $s_3$ edges satisfies the conclusion of the Lemma.

Suppose $\gamma$ enters exactly two hexagons. Then $\gamma$ remains in a subcomplex of $\widehat\Sigma(B_3)$ of the form $p^{-1}(K)$, where $K$ has the shape seen in Figure \ref{fig:Leaving Falk}(b).
Note that there is exactly one square between these hexagons, and $\gamma$ must enter it an even number of times in order to be a closed loop. Thus $\gamma$ may enter exactly one, two, or, three other squares. If $\gamma$ enters one other square, or two other squares in the configuration seen in Figure \ref{fig:hexagons good}(a), then it is contained in some $\hat U$, so $\tilde \gamma$ is contained in a good Falk subcomplex.
Otherwise, there must be some square which it intersects exactly once, and the hyperplane dual to this square's $s_3$ edges satisfies the conclusion of the Lemma.

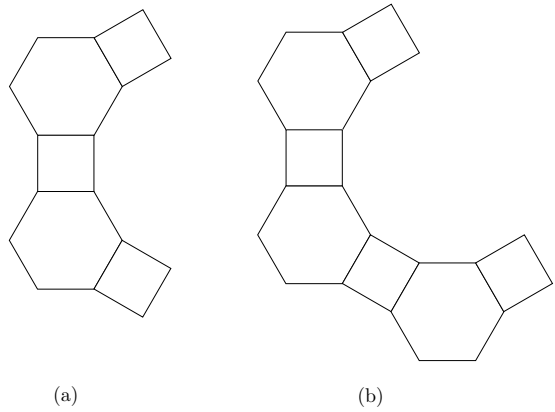
\begin{figure}[!ht]
    \centering
    \tikzset{
      A/.style={regular polygon, regular polygon sides=6,fill=gray!#1,minimum size=2cm, draw},
      B/.style={regular polygon, regular polygon sides=4,fill=gray!#1,minimum size=1.41cm, draw},
    }
    
    \scalebox{0.75}
    {
    \begin{tikzpicture}[node distance=0pt, every node/.style={outer sep=0pt}]
      \node (C1) [A=0] {};
      \node (B1) [B=0, anchor=corner 3, rotate=30] at (C1.corner 6) {};
      \node (B3) [B=0, anchor=corner 2] at (C1.corner 4) {};
      \node (C2) [A=0, anchor=corner 1] at (B3.corner 4) {};
      \node (B5) [B=0, anchor=corner 1, rotate=60] at (C2.corner 6) {};
            \node at (0,-5.5) {(a)};    
    \end{tikzpicture}

    \qquad\qquad
    
    \begin{tikzpicture}[node distance=0pt, every node/.style={outer sep=0pt}]
      \node (C1) [A=0] {};
      \node (B1) [B=0, anchor=corner 3, rotate=30] at (C1.corner 6) {};
      \node (B3) [B=0, anchor=corner 2] at (C1.corner 4) {};
      \node (C2) [A=0, anchor=corner 1] at (B3.corner 4) {};
      \node (B5) [B=0, anchor=corner 1, rotate=60] at (C2.corner 6) {};
      \node (C3) [A=0, anchor=corner 2] at (B5.corner 4) {};
      \node (B7) [B=0, anchor=corner 2, rotate=30] at (C3.corner 1) {};
      
            \node at (1,-5.6) {(b)};    
    \end{tikzpicture}
    }
    \caption{Good configurations}
    \label{fig:hexagons good}
\end{figure}

Suppose $\gamma$ enters exactly three hexagons. Then $\gamma$ remains in a subcomplex of $\widehat\Sigma(B_3)$ of the form $p^{-1}(K)$, where $K$ has the shape seen in Figure \ref{fig:Leaving Falk}(c). In order to enter all three hexagons, it must enter the squares between them at least 4 times. Therefore there is at most one other square it can enter. If it stays within
a subcomplex of the form seen in Figure \ref{fig:hexagons good}(b), then it is contained in some $\hat U$, so $\tilde \gamma$ is contained in a good Falk subcomplex.
If not, then the $s_3$ edges of any of the three remaining squares satisfy the conclusions of the Lemma.

Suppose $\gamma$ enters exactly four hexagons. Then since $\gamma$ must enter the squares between each hexagon at least once, it must remain in a subcomplex of $\widehat\Sigma(B_3)$ of the form $p^{-1}(K)$, where $K$ has the shape seen in Figure \ref{fig:4 hexagon Falk}. If $\gamma$ does not enter the extra top-right square, then it is contained in some $\hat U$, so $\tilde \gamma$ is contained in a good Falk subcomplex. If it does enter this square, then it can do so only once, and the $s_3$ edges of this square satisfy the conclusions of the Lemma.

\begin{figure}[!ht]
    \centering
        
    \tikzset{
      A/.style={regular polygon, regular polygon sides=8,fill=gray!#1,minimum size=2cm, draw},
      B/.style={regular polygon, regular polygon sides=4,fill=gray!#1,minimum size=1.1cm, draw},
    }
    
    \scalebox{1.5}
    {
    \begin{tikzpicture}[node distance=0pt, every node/.style={outer sep=0pt}]
      \node (C1) [A=0] {};
      \node (B1) [B=0, anchor=corner 3] at (C1.corner 2) {};
      \node (B2) [B=0, anchor=corner 1] at (C1.corner 3) {};
      \node (B3) [B=0, anchor=corner 2] at (C1.corner 5) {};
      \node (B4) [B=0, anchor=corner 2] at (C1.corner 8) {};
    
      \begin{scope}[rotate=60]
            \draw[] (B1.corner 2) -- ([shift={(75:1.1cm)}]B1.corner 3) -- 
                ([shift={(75:1.1cm)}]B2.corner 1) -- (B2.corner 2);
            \draw[] (B2.corner 3) -- ([shift={(165:1.1cm)}]B2.corner 4) -- 
                ([shift={(165:1.1cm)}]B3.corner 2) -- (B3.corner 3);
            \draw[] (B3.corner 4) -- ([shift={(255:1.1cm)}]B3.corner 1) -- 
                ([shift={(255:1.1cm)}]B4.corner 3) -- (B4.corner 4);
             \draw[] (B1.corner 1) -- ([shift={(345:1.1cm)}]B1.corner 4)  -- 
                 ([shift={(345:1.1cm)}]B4.corner 2) -- (B4.corner 1);
                 
        \node[draw, rectangle, minimum height=0.75cm, minimum width=0.75cm, anchor=south west, rotate=-45.5] at ([shift={(345:1.1cm)}, rotate=0]B1.corner 4) {};        
    
      \end{scope}

        \coordinate (d1) at 
            ($(B1.corner 1)!0.5 !(B1.corner 2)$) {};
        \coordinate (d3) at 
            ($(B2.corner 2)!0.5 !(B2.corner 3)$) {};
        \coordinate (d5) at 
            ($(B3.corner 3)!0.5 !(B3.corner 4)$) {};
        \coordinate (d7) at 
            ($(B4.corner 4)!0.5 !(B4.corner 1)$) {};
            
        \coordinate (d2) at (135:2.03cm) {};
        \coordinate (d4) at (225:2.03cm) {};
        \coordinate (d6) at (315:2.03cm) {};
        \coordinate (d8) at (45:2.03cm) {};

        \foreach \r/\s in {1/2,3/4,5/6,7/8}
             \coordinate[] (e\r) at ({115+(\r-1)*45}: 1.88cm) {}
              coordinate[] (e\s) at ({155+(\r-1)*45}: 1.88cm) {};
    
        \coordinate (f3) at ($(B2.corner 1) ! 0.5 ! (B2.corner 2)$) {};
        \coordinate (f4) at ($(B2.corner 3) ! 0.5 ! (B2.corner 4)$) {};
        \coordinate (f5) at ($(B3.corner 2) ! 0.5 ! (B3.corner 3)$) {};
        \coordinate (f6) at ($(B3.corner 1) ! 0.5 ! (B3.corner 4)$) {};
         \coordinate (f7) at ($(B4.corner 3) ! 0.5 ! (B4.corner 4)$) {};
         \coordinate (f8) at ($(B4.corner 1) ! 0.5 ! (B4.corner 2)$) {};
         \coordinate (f1) at ($(B1.corner 1) ! 0.5 ! (B1.corner 4)$) {};
         \coordinate (f2) at ($(B1.corner 2) ! 0.5 ! (B1.corner 3)$) {};

    \end{tikzpicture}
    }
    \caption{The 4-hexagon subcomplex}
    \label{fig:4 hexagon Falk}
\end{figure}
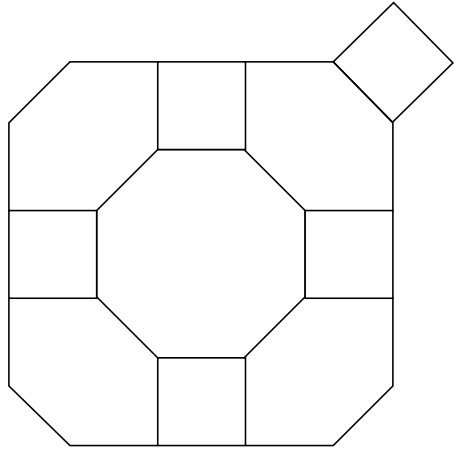

Note that $\gamma$ may not enter more than 4 distinct hexagons, as this would require it to enter squares more than 5 times. Thus we have exhausted all possibilities for $\gamma$.
\end{proof}

Now we may complete the proof of the remaining case. Our argument will follow a similar structure to many of the core ideas of \cite{huang2025cycles}.

\begin{lem} \label{lem: Not good Falk not embedded}
    If $\tilde \gamma$ is not contained in a good Falk subcomplex, then it is not embedded.
\end{lem}

\begin{proof}
Consider the square $S$ and hyperplane $H$ guaranteed by \ref{Lemma: Good hyperplane}.
Our prior remarks imply that $\gamma$ only enters a single $s_3$ loop $\hat E$ of $\hat S$, say at time $t_0$. Consider the projection $\Proj_{\hat E}$. If $\hat F \not= \hat S$ is a square or hexagon which $\gamma$ enters, then since $p(\hat F) \cap H = \varnothing$ and $H$ is the hyperplane dual to $E$, we must have that $\Proj_{\hat E}(\hat F)$ is a vertex $\hat v_{\hat F}$ of $\hat E$. But since $\gamma$ is closed and enters $\hat E$ exactly once, every such hexagon or square must lay on the same side of $H$, so there is a vertex $\hat v$ of $\hat E$ such that $\hat v = \hat v_{\hat F}$ for all such $\hat F$. 

Since $\Proj_{\hat E}$ is a topological retraction, it induces a group-theoretic retraction \[(\Proj_{\hat E})_* : \pi_1(\widehat\Sigma(B_3)) \cong PA(B_3) \to \pi_1(\hat E) \cong PA(s_3).\] The group-theoretic translation of the previous paragraph is simply
\[
(\Proj_{\hat E})_*(w) = a_i,
\]
where $a_i \in A(s_3)$ is the subword representing $\gamma$ entering $\hat S$ at time $t_0$.
But $w$ is trivial (or, topologically, $\gamma$ is nullhomotopic), and this is a retract, so this implies $a_i$ is also trivial (or, topologically, $\Proj_{\hat E}(\gamma)$ is still nullhomotopic). Then $g_{i+1} = a_ig_i = g_i$, so $\tilde \gamma$ is not embedded.
\end{proof}

This concludes the verification that all closed loops can be filled, so we may finally conclude

\moussongiscatone 

\begin{proof}
    We will show that $D(B_3)$ satisfies the hypotheses of Theorem \ref{thm:B3 CAT1 Criteria}.
    First, it follows quickly from \cite{cd1995k} that $D(B_3)$ satisfies conditions (1), (2), (3), and (4).
    Lemmas \ref{lemma:filling 4 cycles in DB3},  \ref{lem:filling 6 cycle in DB3}, and \ref{lem:filling 8 cycle in DB3} imply condition (5) is satisfied. Last, Lemmas \ref{lem:good falk 10} and \ref{lem: Not good Falk not embedded} show that condition (6) is satisfied.
\end{proof}

This quickly implies the following.

\generalcatzero

\begin{proof}
    This follows from applying Proposition \ref{prop:Moussong reduction}, noticing that, by the classification of spherical Coxeter-Dynkin diagrams, the (maximal) connected diagrams $\Lambda$ in $\mathcal S^f_\Gamma$ are either dihedral, type $A_3$, or type $B_3$. In the first case, the Moussong metric on $D(\Lambda)$ is $\cat(1)$ by \cite{cd1995k}. In the second case, the Moussong metric on $D(\Lambda)$ is $\cat(1)$ by \cite{charney2004deligne}. And, in the last case, the Moussong metric on $D(\Lambda)$ is $\cat(1)$ by Theorem \ref{prop:edgereduction}.
\end{proof}

\bibliographystyle{alpha}
\bibliography{refs}

\end{document}